\documentclass[a4paper,10pt]{amsart}

\usepackage[english]{babel}
\usepackage[english]{translator}
\usepackage[T1]{fontenc}
\usepackage[utf8]{inputenc}

\usepackage{amsmath}
\usepackage{amssymb}
\usepackage{amsfonts}
\usepackage{amstext}
\usepackage{amsthm}
\usepackage{bbm}

\usepackage{graphicx}
\usepackage{color}
\usepackage{psfrag}
\usepackage{subfigure}
\usepackage{float}

\usepackage{marvosym}
\usepackage{enumitem}

\usepackage{hyperref}

\usepackage[
  hmarginratio={1:1},     
  vmarginratio={1:1},     
  textwidth=460pt,        
  heightrounded,          
  bottom=3.65cm,
  top=3.65cm,
]{geometry}

\hypersetup{
  pdffitwindow=false,
  pdfhighlight=/O,
  pdfnewwindow,
  colorlinks=true,
  citecolor=red,            
  linkcolor=blue,             
  menucolor=blue,            
  urlcolor=blue,             
  pdfpagemode=UseOutlines,
  bookmarksnumbered=true,
  linktocpage,
  pdftitle={Fredholm Properties and $L^p$-Spectra of Localized Rotating Waves in Parabolic Systems},
  pdfsubject={Fredholm Properties and $L^p$-Spectra of Localized Rotating Waves in Parabolic Systems},
  pdfauthor={Wolf-J"urgen Beyn, Denny Otten},
  pdfkeywords={Rotating wave, systems of reaction-diffusion equations, Fredholm theory, $L^p$-spectrum, essential spectrum, point spectrum, Ornstein-Uhlenbeck operator.},
  pdfcreator={pdflatex},
  pdfproducer={LaTeX with hyperref}
}

\newcommand{\D}{\mathcal{D}}              
\newcommand{\K}{\mathbb{K}}               
\newcommand{\C}{\mathbb{C}}               
\newcommand{\R}{\mathbb{R}}               
\newcommand{\Z}{\mathbb{Z}}                
\newcommand{\N}{\mathbb{N}}                
\newcommand{\SE}{\mathrm{SE}}              
\renewcommand{\Re}{\mathrm{Re}\,}          
\renewcommand{\Im}{\mathrm{Im}\,}          
\renewcommand{\L}{\mathcal{L}}             
\newcommand{\A}{\mathcal{A}}             
\newcommand{\B}{\mathcal{B}}             
\newcommand{\diag}{\mathrm{diag}}          

\DeclareMathOperator{\esssup}{ess\,sup\,}
\DeclareMathOperator{\dcup}{\dot{\cup}}

\newcommand{\herm}{{\mathsf{H}}}

\newcommand{\amin}{a_{\mathrm{min}}} 
\newcommand{\amax}{a_{\mathrm{max}}} 
\newcommand{\azero}{a_0}
\newcommand{\aone}{a_1}
 
\newcommand{\bzero}{b_0}

\newcommand{\begriff}[1]{\textit{#1}}

\makeatletter
\renewcommand{\@secnumfont}{\bfseries}
\makeatother
\makeatletter
  \def\section{\@startsection{section}{1}%
    \z@{.7\linespacing\@plus\linespacing}{.5\linespacing}%
    {\normalfont\LARGE\bfseries}}
\makeatother
\makeatletter
\def\@seccntformat#1{%
  \protect\textup{%
    \protect\@secnumfont
    \expandafter\protect\csname format#1\endcsname 
    \csname the#1\endcsname
    \protect\@secnumpunct
  }%
}

\newcommand{\sect}
{
  \setcounter{equation}{0}
  \setcounter{figure}{0}
  \section
}

\theoremstyle{definition}
\newtheorem{definition}{Definition}[section]
\newtheorem{assumption}[definition]{Assumption}
\newtheorem{example}[definition]{Example}
\newtheorem{remark}[definition]{Remark}
\newtheorem{remarks}[definition]{Remarks}

\theoremstyle{plain}
\newtheorem{theorem}[definition]{Theorem}
\newtheorem{lemma}[definition]{Lemma}
\newtheorem{corollary}[definition]{Corollary}

\allowdisplaybreaks

\begin{document}
\title[Fredholm Properties and $L^p$-Spectra of\\Localized Rotating Waves in Parabolic Systems]{Fredholm Properties and $L^p$-Spectra of\\Localized Rotating Waves in Parabolic Systems}
\setlength{\parindent}{0pt}
\vspace*{0.75cm} 
\begin{center}
\normalfont\huge\bfseries{\shorttitle}\\
\vspace*{0.25cm} 
\end{center}

\vspace*{0.8cm}
\noindent
\begin{minipage}[t]{0.99\textwidth}
\begin{minipage}[t]{0.48\textwidth}
\hspace*{1.8cm} 
{\bf Wolf-J{\"u}rgen Beyn}${}^{ }$\hspace*{-0.025cm}\footnotemark[1]${}^{,}$\footnotemark[3] \\
\hspace*{1.8cm}
Department of Mathematics \\
\hspace*{1.8cm}
Bielefeld University \\
\hspace*{1.8cm}
33501 Bielefeld \\
\hspace*{1.8cm}
Germany
\end{minipage}
\begin{minipage}[t]{0.48\textwidth}
\hspace*{1.8cm}
\textbf{Denny Otten}\footnotemark[2]${}^{,}$\footnotemark[3] \\
\hspace*{1.8cm}
Department of Mathematics \\
\hspace*{1.8cm}
Bielefeld University \\
\hspace*{1.8cm}
33501 Bielefeld \\
\hspace*{1.8cm}
Germany
\end{minipage}
\end{minipage}\\

\footnotetext[1]{e-mail: \textcolor{blue}{beyn@math.uni-bielefeld.de}, phone: \textcolor{blue}{+49 (0)521 106 4798}, \\
                                          fax: \textcolor{blue}{+49 (0)521 106 6498}, homepage: \url{http://www.math.uni-bielefeld.de/~beyn/AG\_Numerik/}.}
\footnotetext[2]{e-mail: \textcolor{blue}{dotten@math.uni-bielefeld.de}, phone: \textcolor{blue}{+49 (0)521 106 4784}, \\
                                 fax: \textcolor{blue}{+49 (0)521 106 6498}, homepage: \url{http://www.math.uni-bielefeld.de/~dotten/}.}
\footnotetext[3]{supported by CRC 701 'Spectral Structures and Topological Methods in Mathematics',  Bielefeld University}

\vspace*{0.8cm} 
\noindent
\hspace*{5.4cm}
Date: \today
\normalparindent=12pt

\vspace{1.05cm} 
\noindent
\begin{center}
\begin{minipage}{0.9\textwidth}
  {\small
  \textbf{Abstract.} 
  In this paper we study spectra and Fredholm properties of Ornstein-Uhlenbeck operators
  \begin{equation*}
    \L v(x) := A\triangle v(x) + \langle Sx,\nabla v(x)\rangle + Df(v_{\star}(x))v(x), \,x\in\mathbb{R}^d,\,d\geqslant 2,
  \end{equation*}
  where $v_{\star}:\R^d\rightarrow\R^m$ is the profile of a rotating wave satisfying $v_{\star}(x)\to v_{\infty}\in\R^m$ as $|x|\to\infty$, 
  the map $f:\R^m\rightarrow\R^m$ is smooth, the matrix $A\in\mathbb{R}^{m,m}$ has eigenvalues with positive real parts and commutes with 
  the limit matrix $Df(v_{\infty})$. The  matrix $S\in\R^{d,d}$ is assumed to be skew-symmetric with eigenvalues
  $(\lambda_1,\ldots,\lambda_d)=(\pm i\sigma_1,\ldots,\pm i \sigma_k,0,\ldots,0)$. The spectra of these linearized operators are crucial 
  for the nonlinear stability of rotating waves in reaction diffusion systems. 
  We prove under appropriate conditions that every $\lambda\in\C$ satisfying the dispersion relation
  \begin{equation*}
    \det\Big(\lambda I_m + \eta^2 A - Df(v_{\infty}) + i\langle n,\sigma\rangle I_m\Big)=0\quad\text{for some $\eta\in\R$ and $n\in\Z^k$}
  \end{equation*}
  belongs to the essential spectrum $\sigma_{\mathrm{ess}}(\L)$  in $L^p$. For values $\Re\lambda$ to the right of the spectral bound for 
  $Df(v_{\infty})$ we show that the operator $\lambda I-\L$ is Fredholm of index $0$, solve the identification problem for the adjoint operator 
  $(\lambda I-\L)^*$, and formulate the Fredholm alternative. Moreover, we show that the set
  \begin{align*}
    \sigma(S)\cup\{\lambda_i+\lambda_j:\;\lambda_i,\lambda_j\in\sigma(S),\,1\leqslant i<j\leqslant d\}
  \end{align*}
  belongs to the point spectrum $\sigma_{\mathrm{pt}}(\L)$  in $L^p$. We determine the associated eigenfunctions and show that they decay exponentially
  in space. As an application we analyze spinning soliton solutions which occur in the Ginzburg-Landau equation and compute their numerical spectra 
  as well as associated eigenfunctions. 
  Our results form the basis for investigating nonlinear stability of rotating waves in higher space dimensions and truncations to bounded domains.
  }
\end{minipage}
\end{center}




%
%
\sect{Introduction}
\label{sec:1}

In the present paper we study operators obtained from linearizing reaction-diffusion systems
\begin{equation}
  \begin{aligned}
  \label{equ:1.1}
  \begin{split}
    u_t(x,t) & = A\triangle u(x,t) + f(u(x,t)), &&t>0,\,x\in\R^d,\,d\geqslant 2,\\
      u(x,0) & = u_0(x),                        &&t=0,\,x\in\R^d,
  \end{split}
  \end{aligned}
\end{equation}
where $A\in\R^{m,m}$ has eigenvalues with positive real part, $f:\R^m\rightarrow\R^m$ is sufficiently smooth, $u_0:\R^d\rightarrow\R^m$ are  initial data and 
$u:\R^d\times[0,\infty)\rightarrow\R^m$ denotes a vector-valued solution.

Our main interest is in rotating wave solutions of \eqref{equ:1.1} of the form
\begin{equation}
  \label{equ:1.2}
  u_{\star}(x,t) = v_{\star}(e^{-tS}x),\,t\geqslant 0,\,x\in\R^d,\,d\geqslant 2
\end{equation}
with space-dependent profile $v_{\star}:\R^d\rightarrow\R^m$ and skew-symmetric matrix $S\in\R^{d,d}$. The term $e^{-tS}$ describes 
rotations in $\R^d$, and hence $u_{\star}$ is a solution rotating at constant velocity while maintaining its shape determined by $v_{\star}$. The profile $v_{\star}$ 
is called (exponentially) localized, if it tends (exponentially) to some constant vector $v_{\infty}\in\R^m$ as $|x|\to\infty$. 

Transforming \eqref{equ:1.1} via $u(x,t)=v(e^{-tS}x,t)$ into a co-rotating frame yields the evolution equation
\begin{equation}
  \begin{aligned}
  \label{equ:1.3}
  \begin{split}
    v_t(x,t) =& A\triangle v(x,t) + \left\langle Sx,\nabla v(x,t)\right\rangle + f(v(x,t)), &&t>0,\,x\in\R^d,\,d\geqslant 2,\\
      v(x,0) =& u_0(x),                                                                     &&t=0,\,x\in\R^d.
  \end{split}
  \end{aligned}
\end{equation}
The diffusion and drift term are given by
\begin{equation} 
  \label{equ:1.4}
  A\triangle v(x):=A\sum_{i=1}^{d}\frac{\partial^2v}{\partial x_i^2}(x)\quad\text{and}\quad
  \left\langle Sx,\nabla v(x)\right\rangle:=\sum_{i=1}^{d}\sum_{j=1}^{d}S_{ij}x_j \frac{\partial v}{\partial x_i} (x).
\end{equation}
The pattern $v_{\star}$ itself appears as a stationary solution of \eqref{equ:1.3}, i.e. $v_{\star}$ solves the rotating wave equation
\begin{equation}
  \label{equ:1.5}
  A\triangle v_{\star}(x)+\left\langle Sx,\nabla v_{\star}(x)\right\rangle+f(v_{\star}(x))=0,\,x\in\R^d,\,d\geqslant 2.
\end{equation}
We write \eqref{equ:1.5} as $[\L_0 v_{\star}](x)+f(v_{\star}(x))=0$ with
the Ornstein-Uhlenbeck operator defined by
\begin{equation} 
  \label{equ:1.6}
  \left[\L_0 v\right](x) := A\triangle v(x)+\left\langle Sx,\nabla v(x)\right\rangle,\,x\in\R^d.
\end{equation}
When proving nonlinear stability of the rotating wave, a crucial role is played by the linearized operator 
\begin{equation} 
  \label{equ:1.8}
  \left[\L v\right](x) := \left[\L_0 v\right](x) + Df\left(v_{\star}(x)\right)v(x),\,x\in\R^d.
\end{equation}

The aim of this paper is to analyze Fredholm properties and spectra of the $L^p$-eigenvalue problem associated with the linearization $\L$,
\begin{equation}
  \label{equ:1.9}
  \left[\left(\lambda I-\L\right)v\right](x) = 0,\,x\in\R^d.
\end{equation}
As usual, the $L^p$-spectrum $\sigma(\L)$ of $\L$ is  decomposed into the disjoint union of point spectrum 
$\sigma_{\mathrm{pt}}(\L)$ and essential spectrum $\sigma_{\mathrm{ess}}(\L)$, cf. Definition \ref{def:1.1},
\begin{equation}
  \label{equ:1.10}
  \sigma(\L)=\sigma_{\mathrm{ess}}(\L)\dcup\sigma_{\mathrm{pt}}(\L).
\end{equation}
In Section \ref{sec:2} we evaluate the dispersion relation associated
with the limit operator,
\begin{equation}
  \label{equ:Linfty}
  \L_{\infty}=\L_0 + Df(v_{\infty})
\end{equation}
and show that its solutions belong to
$\sigma_{\mathrm{ess}}(\L)$. For every $\lambda \in \C$ with $\Re \lambda$
larger than $\Re\sigma_{\mathrm{ess}}(\L)$ we prove in Section \ref{sec:4} that the operator $\lambda I-\L$ is Fredholm of
index $0$ in $L^p$-spaces.  Finally,  in Section \ref{sec:3} we compute those eigenvalues on the imaginary axis
which are caused by Euclidean equivariance of the underlying
equation, and we prove their exponential decay in space. The whole
approach makes extensive use of our previous results on the
identification problem and on exponential decay estimates
for the wave itself and for solutions of the linearized equation, 
see \cite{BeynOtten2016a,Otten2014a,Otten2015a}.

We emphasize that the results from Section \ref{sec:2}, \ref{sec:3} and \ref{sec:5} extend results from the PhD thesis \cite{Otten2014}. 
Major novelties of this paper concern the Fredholm properties of the linearized operator in Section \ref{sec:4}.

%
%
\sect{Assumptions and main results}
\label{sec:1b}

\subsection{Assumptions}
\label{subse:1.2}

The following conditions will be needed in this paper and relations among them will be discussed below. The conditions are essential for applying previous
results from  \cite{BeynOtten2016a,Otten2014,Otten2014a,Otten2015a}.

\begin{assumption}
  \label{ass:Assumption1}
  For $A\in\K^{m,m}$ with $\K\in\{\R,\C\}$, $1<p<\infty$ and $q=\frac{p}{p-1}$ consider the conditions
  \begin{flalign}
    &\text{$A$ is diagonalizable (over $\C$)},         \tag{A1}\label{cond:A1} &\\
    &\Re\sigma(A)>0, \tag{A2}\label{cond:A2} &\\
    &\Re\left\langle w,Aw\right\rangle\geqslant\beta_A\;\forall\,w\in\K^m,\,|w|=1\text{ for some $\beta_A>0$,} \tag{A3}\label{cond:A3} &\\
    &\text{There exists $\gamma_A>0$ such that for all $z,w\in\K^m$} \tag{A4${}_p$}\label{cond:A4DC} &\\
    &\quad|z|^2\Re\langle w,Aw\rangle + (p-2)\Re\langle w,z\rangle\Re\langle z,Aw\rangle\geqslant\gamma_A |z|^2|w|^2, \nonumber &\\
    &\text{There exists $\delta_A>0$ such that for all $z,w\in\K^m$} \tag{A4${}_q$}\label{cond:A4DCq} &\\
    &\quad|z|^2\Re\langle w,A^{\herm} w\rangle + (q-2)\Re\langle w,z\rangle\Re\langle z,A^{\herm} w\rangle\geqslant\delta_A |z|^2|w|^2, \nonumber &\\
    &A\text{ is invertible and } \mu_1(A)>\frac{|p-2|}{p} \tag{A5${}_p$}\label{cond:A4}
    \; \text{(to be read as $A>0$ in case $N=1$, $\K=\R$),}  \\    
    &A\text{ is invertible and } \mu_1(A^{\herm})>\frac{|q-2|}{q} \tag{A5${}_q$}\label{cond:A4q}
     \; \text{(to be read as $A^{\herm}>0$ in case $N=1$, $\K=\R$).}
  \end{flalign}
\end{assumption}

Assumption \eqref{cond:A1} is a \begriff{system condition} and ensures that all results for scalar equations can be extended to system cases. This condition 
is independent of all other conditions and is used in \cite{Otten2014,Otten2014a} to derive an explicit formula for the heat kernel 
of $\L_{0}$. A typical case where \eqref{cond:A1} holds, is a scalar complex-valued equation when transformed into a  real-valued system of dimension $2$. 
The \begriff{positivity condition} \eqref{cond:A2} guarantees that the diffusion part $A\triangle$ 
is an elliptic operator. All eigenvalues $\lambda\in \sigma(A)$ of $A$ lie in the open right half-plane $\{\lambda\in\C\mid\Re\lambda>0\}$. 
The \begriff{strict accretivity condition} \eqref{cond:A3} is more restrictive than \eqref{cond:A2}. In \eqref{cond:A3} we use $\left\langle u,v\right\rangle:=u^{\herm} v$ 
with $u^{\herm}:=\bar{u}^{\top}$ to denote the standard inner product on $\K^m$. Recall that condition \eqref{cond:A2} is satisfied iff there exists an inner product
  $\left[\cdot,\cdot\right]$  and some $\beta_A>0$ such that  $\Re\left[w,Aw\right]\geqslant\beta_A$ for all $w\in\K^m$ with $\left[w,w\right]=1$.
 Condition \eqref{cond:A3} ensures that the differential operator $\L_{0}$ 
is closed on its (local) domain $\D^p_{\mathrm{loc}}(\L_0)$, see \cite{Otten2014,Otten2015a}. The \begriff{$L^p$-dissipativity condition} \eqref{cond:A4DC} is 
more restrictive than \eqref{cond:A3} and imposes additional requirements on the spectrum of $A$. This condition originating from \cite{CialdeaMazya2005,CialdeaMazya2009}, 
is used in \cite{Otten2014,Otten2015a} to prove $L^p$-resolvent estimates for $\L_{0}$. A geometric meaning of \eqref{cond:A4DC} can be given in terms of the 
antieigenvalues of the diffusion matrix $A$. In \cite{Otten2014,Otten2015b}, it is proved that condition \eqref{cond:A4DC} is equivalent to the \begriff{$L^p$-antieigenvalue condition} 
\eqref{cond:A4}. Condition \eqref{cond:A4} requires that the \begriff{first antieigenvalue of $A$} (see \cite{Gustafson1968,Gustafson2012}), defined by
\begin{align*}
  \mu_1(A) := \inf\left\{\frac{\Re\left\langle w,Aw\right\rangle}{|w||Aw|}:w\in\K^m,w\neq 0,Aw\neq 0 \right \}
\end{align*}
is bounded from below by a non-negative $p$-dependent constant. Condition \eqref{cond:A4} is also equivalent to the following $p$-dependent upper bound for the 
(\begriff{real}) \begriff{angle of $A$} (cf. \cite{Gustafson1968}),
\begin{align*}
  \Phi_{\R}(A):=\cos^{-1}\left(\mu_1(A)\right)<\cos^{-1}\left(\frac{|p-2|}{p}\right)\in\big(0,\frac{\pi}{2}\big],\quad 1<p<\infty.
\end{align*}
The first antieigenvalue $\mu_1(A)$ can be considered as the cosine of the maximal (real) turning angle of vectors mapped by the matrix $A$. 
Some special cases in which the first antieigenvalue can be given explicitly are treated in \cite{Otten2014,Otten2015b}. 
The \begriff{$L^q$-dissipativity condition} \eqref{cond:A4DCq} for the \begriff{conjugate index} $q=\frac{p}{p-1}$ will be used to
analyze the adjoint operator $(\lambda I-\L)^{*}$. The condition enables us to solve the identification problem for the adjoint operator, which becomes crucial 
when analyzing Fredholm properties of $\lambda I-\L$. Condition \eqref{cond:A4DCq} is more restrictive than \eqref{cond:A3} but equivalent to
 \eqref{cond:A4q}, 
and similar comments apply to the first antieigenvalue $\mu_1(A^{\herm})$ and the (real) angle $\Phi_{\R}(A^{\herm})$. However, we emphasize that $L^p$-dissipativity 
\eqref{cond:A4DC} and $L^q$-dissipativity \eqref{cond:A4DCq} for the conjugate index $q$ are generally independent of each other, even in case $p=q=2$. 
Only in case of Hermitian matrices $A\in\C^{m,m}$ both conditions coincide. 

We continue with the \begriff{rotational condition} \eqref{cond:A5} and a \begriff{smoothness condition} \eqref{cond:A6}.

\begin{assumption}
  \label{ass:Assumption2}
  The matrix $S\in\R^{d,d}$ satisfies
  \begin{flalign}
    &\text{$S$ is skew-symmetric, i.e. $S=-S^T$.}        \tag{A6}\label{cond:A5} &
  \end{flalign}
\end{assumption}

\begin{assumption}
  \label{ass:Assumption3}
  The function $f:\R^m\rightarrow\R^m$ satisfies
  \begin{flalign}
    &f\in C^2(\R^m,\R^m).                             \tag{A7}\label{cond:A6} &
  \end{flalign}
\end{assumption}

Later on in Section \ref{sec:5} we apply our results to complex-valued nonlinearities of the form
\begin{equation} \label{equ:complexversion}
  f:\C^m\rightarrow\C^m,\quad f(u)=g\left(|u|^2\right)u,
\end{equation}
where $g:\R\rightarrow\C^{m,m}$ is a sufficiently smooth function. Such nonlinearities arise for example in Ginzburg-Landau equations, Schr\"odinger equations, 
$\lambda-\omega$ systems and many other equations from physical sciences, see \cite{Otten2014} and references therein. 
Note, that in this case, the function $f$ is not holomorphic in $\C$, but its real-valued version in $\R^2$ satisfies \eqref{cond:A6} if $g$ is in $C^2$. For 
differentiable functions $f:\R^m\rightarrow\R^m$ we denote by $Df$ the \begriff{Jacobian matrix} in the real sense.

\begin{assumption}
  \label{ass:Assumption4}
  For $v_{\infty}\in\R^m$ consider the following conditions:
  \begin{flalign}
    &f(v_{\infty})=0,                                                                     \tag{A8}\label{cond:A7} &\\
    &\text{$A,Df(v_{\infty})\in\R^{m,m}$ are simultaneously diagonalizable (over $\C$),}  \tag{A9}\label{cond:A8} &\\
    &\Re\sigma\left(Df(v_{\infty})\right)<0,                                              \tag{A10}\label{cond:A9} &\\
    &\Re\left\langle w,Df(v_{\infty})w\right\rangle\leqslant-\beta_{\infty}\;\forall\,w\in\K^m,\,|w|=1\text{ for some $\beta_{\infty}>0$.} \tag{A11}\label{cond:A10} &
  \end{flalign}
\end{assumption}

The \begriff{constant asymptotic state condition} \eqref{cond:A7} requires $v_{\infty}$ to be a steady state of the nonlinear equation. The \begriff{system condition} 
\eqref{cond:A8} is an extension of Assumption \eqref{cond:A1}, and
the \begriff{coercivity condition}  \eqref{cond:A10} is again  more restrictive
than the \begriff{spectral condition} \eqref{cond:A9}.

For matrices $C\in\K^{m,m}$ with spectrum $\sigma(C)$ we denote by $\rho(C):=\max_{\lambda\in\sigma(C)}\left|\lambda\right|$ its \begriff{spectral radius} 
and by $s(C):=\max_{\lambda\in\sigma(C)}\Re\lambda$ its \begriff{spectral abscissa} (or \begriff{spectral bound}). 
With this notation, we define the following constants which appear in the linear theory from \cite{BeynOtten2016a,Otten2014,Otten2014a,Otten2015a}:
\begin{equation}
  \begin{aligned}
  \amin :=& \left(\rho\left(A^{-1}\right)\right)^{-1}, &&\amax  := \rho(A),              &&\azero := -s(-A),\\
  \aone :=& \left(\frac{\amax^2}{\amin\azero}\right)^{\frac{d}{2}},                      &&\bzero := -s(Df(v_{\infty})).               && 
  \end{aligned}
  \label{equ:aminamaxazerobzero}
\end{equation}
Recall the relations $0<\azero\leqslant\beta_A$ and $0<\bzero\leqslant\beta_{\infty}$ to the coercivity constants from \eqref{cond:A3} and \eqref{cond:A10}.

The theory in this paper is partially developed for more general differential operators, see \eqref{equ:1.12a} below. For this purpose we transfer Assumption \ref{ass:Assumption4} 
to general matrices $B_{\infty}$. Later on, we apply the results to
$B_{\infty}=-Df(v_{\infty})$.


\begin{assumption}
  \label{ass:Assumption5}
  For $B_{\infty}\in\K^{m,m}$ consider the conditions
  \begin{flalign}
    &\text{$A,B_{\infty}\in\K^{m,m}$ are simultaneously diagonalizable (over $\C$) with transformation $Y\in\C^{m,m}$,}          \tag{A9${}_{B_{\infty}}$}\label{cond:A8B} &\\
    &\Re\sigma(B_{\infty})>0,                     \tag{A10${}_{B_{\infty}}$}\label{cond:A9B} & \\
    &\Re\left\langle w,B_{\infty}w\right\rangle\geqslant\beta_{\infty}\;\forall\,w\in\K^m,\,|w|=1\text{ for some $\beta_{\infty}>0$.} \tag{A11${}_{B_{\infty}}$}\label{cond:A10B} &
  \end{flalign}
\end{assumption}

Finally, let us recall some standard definitions for closed, densely defined operators, see e.g.  \cite{Henry1981}.

\begin{definition}
  \label{def:1.1}
  Let $X$ be a (complex-valued) Banach space, $\A:X\supseteq\D(\A)\rightarrow X$ be a closed, densely defined, linear operator 
  and let $\lambda\in\C$.
  \begin{itemize}[leftmargin=0.43cm]\setlength{\itemsep}{0.1cm}
  \item[a)] The sets $\mathcal{N}(\lambda I-\A):=\{v\in\D(\A):\;(\lambda I-\A)v=0\}$ and $\mathcal{R}(\lambda I-\A):=\{(\lambda I-\A)v:\;v\in\D(\A)\}$ are called 
  the \begriff{kernel} and the \begriff{range of $\lambda I-\A$}, respectively.
  \item[b)] A value $\lambda\in\C$ belongs to the \begriff{resolvent set} $\rho(\A)$ \begriff{of $\A$} if $\left(\lambda I-\A\right):\D(\A)\rightarrow X$ 
  is bijective and has bounded inverse $\left(\lambda I-\A\right)^{-1}:X\rightarrow\D(\A)$. The inverse $\left(\lambda I-\A\right)^{-1}$ is called the 
  \begriff{resolvent of $\A$ at $\lambda$}.
  \item[c)] The set $\sigma(\A):=\C\backslash\rho(\A)$ is called the \begriff{spectrum of $\A$}. An element $\lambda\in\sigma(\A)$ satisfying $(\lambda I-\A)v=0$ for some 
  $0\neq v\in\D(\A)$ is called an \begriff{eigenvalue of $\A$} and $v$ is called an \begriff{eigenfunction of $\A$}. 
  \item[d)] $\lambda\in\sigma(\A)$ is called \begriff{isolated} if there is $\varepsilon>0$ such that $\mu\in\rho(\A)$ for all $\lambda\neq\mu\in\C$ with 
  $|\lambda-\mu|<\varepsilon$.
  $\lambda\in\sigma(\A)$ has \begriff{finite multiplicity} if 
$\cup_{k\in \N}\, \mathcal{N}((\lambda I-\A)^k)$ is finite dimensional.
  The \begriff{point spectrum of $\A$} is defined by
  \begin{align*}
    \sigma_{\mathrm{pt}}(\A):=\{\lambda\in\sigma(\A)\mid\lambda\text{ is isolated eigenvalue of finite multiplicity}\}.
  \end{align*}
  \item[e)] $\lambda\in\C$ is called a \begriff{normal point of $\A$} if $\lambda\in\rho(\A)\cup \sigma_{\mathrm{pt}}(\A)$.
  The \begriff{essential spectrum of $\A$} is defined by
  \begin{align*}
    \sigma_{\mathrm{ess}}(\A)=\left\{\lambda\in\C\mid \lambda\text{ is not a normal point of $\A$}\right\}.
  \end{align*}
  \end{itemize}
\end{definition}

Note that $\sigma(A)=\sigma_{\mathrm{ess}}(A)\dcup\sigma_{\mathrm{pt}}(A)$ according to  Definition \ref{def:1.1}.

\subsection{Outline and main results}
\label{subsec:1.1}
In Section \ref{sec:2} we investigate the essential spectrum $\sigma_{\mathrm{ess}}(\L)$ of $\L$ from \eqref{equ:1.8}, which is determined by the limiting 
behavior of $v_{\star}$ at infinity. By a far-field linearization and an angular Fourier decomposition, bounded eigenfunctions of the problem
\eqref{equ:1.9} are obtained from the $m$-dimensional eigenvalue problem (Section \ref{subsec:2.1})
\begin{equation}
  \label{equ:1.16}
  \Big(\lambda I_m+\eta^2 A+i\langle n,\sigma\rangle I_m-Df(v_{\infty})\Big)z = 0\quad\text{for some $\eta\in\R$ and $n\in\Z^k$},
\end{equation}
with nonzero eigenvalues $\pm i\sigma_1,\ldots,\pm i\sigma_k$ of $S$, $\sigma_1,\ldots,\sigma_k\in\R$ and $1\leqslant k\leqslant\lfloor\frac{d}{2}\rfloor$. 
Obviously, \eqref{equ:1.16} has a solution $0\neq z\in\C^m$ if and only if $\lambda\in\C$ satisfies the \begriff{dispersion relation for localized rotating waves}
\begin{equation}
  \label{equ:1.11}
  \det\Big(\lambda I_m + \eta^2 A - Df(v_{\infty}) + i\langle n,\sigma\rangle I_m\Big)=0
\end{equation}
for some $\eta\in\R$ and $n\in\Z^k$. Defining the \begriff{dispersion set}
\begin{equation}
  \label{equ:1.12}
  \sigma_{\mathrm{disp}}(\L):=\{\lambda\in\C:\;\lambda\text{ satisfies \eqref{equ:1.11} for some $\eta\in\R$ and $n\in\Z^k$}\},
\end{equation}
we show that $\sigma_{\mathrm{disp}}(\L)$ belongs to the essential spectrum $\sigma_{\mathrm{ess}}(\L)$ of $\L$ in $L^p$. 

\begin{theorem}[Essential spectrum at localized rotating waves]\label{thm:EssSpecLRW}
  Let $f\in C^{\max\{2,r-1\}}(\R^m,\R^m)$ for some $r\in\N$ and let the assumptions \eqref{cond:A4DC}, \eqref{cond:A5}, \eqref{cond:A7}, \eqref{cond:A8} 
  and \eqref{cond:A10} be satisfied for $\K=\C$ and for some $1<p<\infty$ with $\frac{d}{p}<r$ if $r\leqslant 2$ and $\frac{d}{p}\leqslant 2$ if $r\geqslant 3$. 
  Moreover, let $\pm i\sigma_1,\ldots,\pm i\sigma_k$ denote the nonzero eigenvalues of $S$. Then, for every $0<\varepsilon<1$ there is a constant $K_1=K_1(A,f,v_{\infty},d,p,\varepsilon)>0$ 
  with the following property: 
  For every classical solution $v_{\star}\in C^{r+1}(\R^d,\R^m)$ of 
  \begin{align*}
    A\triangle v(x)+\left\langle Sx,\nabla v(x)\right\rangle+f(v(x))=0,\,x\in\R^d,
  \end{align*}
  satisfying
  \begin{align*}
    \sup_{|x|\geqslant R_0}\left|v_{\star}(x)-v_{\infty}\right|\leqslant K_1\text{ for some $R_0>0$,}
  \end{align*}
  the dispersion set $\sigma_{\mathrm{disp}}(\L)$ from \eqref{equ:1.12} belongs to the essential spectrum $\sigma_{\mathrm{ess}}(\L)$ of the 
  linearized operator $\L$ from \eqref{equ:1.8} in $L^p(\R^d,\C^m)$, i.e. $\sigma_{\mathrm{disp}}(\L)\subseteq\sigma_{\mathrm{ess}}(\L)$ in $L^p(\R^d,\C^m)$.
\end{theorem} 

For the proof of Theorem \ref{thm:EssSpecLRW}, we consider differential operators
\begin{equation}
  \label{equ:4.1}
  \L_Q:\left(\D^p_{\mathrm{loc}}(\L_0),\left\|\cdot\right\|_{\L_0}\right)\rightarrow \left(L^p(\R^d,\C^m),\left\|\cdot\right\|_{L^p}\right),\quad 1<p<\infty
\end{equation}
of the form
\begin{equation}
  \label{equ:1.12a}
  [\L_Q v](x) = A\triangle v(x) + \left\langle Sx,\nabla v(x)\right\rangle - B_{\infty}v(x) + Q(x) v(x),\,x\in\R^d,\,d\geqslant 2,
\end{equation}
where the function $Q\in L^{\infty}(\R^d,\C^{m,m})$ is assumed to \begriff{vanish at infinity} in the sense that
\begin{equation} \label{equ:4.2}
  \underset{|x|\geqslant R}{\esssup}\left|Q(x)\right|\rightarrow 0\text{ as }R\rightarrow\infty.
\end{equation}
The maximal domain of $\L_Q$ agrees with that of the Ornstein-Uhlenbeck operator $\L_0$ from \eqref{equ:1.6}
\begin{equation}
  \label{equ:4.3}
  \D^p_{\mathrm{loc}}(\L_0) = \left\{v\in W^{2,p}_{\mathrm{loc}}(\R^d,\C^m)\cap L^p(\R^d,\C^m):\; \L_0 v\in L^p(\R^d,\C^m)\right\}
\end{equation}
for $1<p<\infty$, see \cite[Thm.~5.1]{Otten2015a}. It is a Banach space when equipped with the graph norm
\begin{equation}
  \label{equ:4.4}
  \left\|v\right\|_{\L_0} = \left\|\L_0 v\right\|_{L^p(\R^d,\C^m)} + \left\|v\right\|_{L^p(\R^d,\C^m)},\quad v\in\D^p_{\mathrm{loc}}(\L_0),
\end{equation}
and has the following representation (see \cite[Thm.~6.1]{Otten2015a}),
\begin{equation*}
  \D^p_{\mathrm{max}}(\L_0) = \left\{v\in W^{2,p}(\R^d,\C^m):\;\left\langle S\ 
\cdot \ ,\nabla v\right\rangle\in L^p(\R^d,\C^m)\right\},
\end{equation*}
where the graph norm $\left\|\cdot\right\|_{\L_0}$ is equivalent to $\left\|\cdot\right\|_{W^{2,p}}+\left\|\left\langle S\ \cdot \ ,\nabla\right\rangle\right\|_{L^{p}}$, 
see \cite[Cor.~5.26]{Otten2014}. Such a strong characterization of the domain is rather involved to prove, but will not be needed here. 
The differential operator $\L_Q$ is a variable coefficient perturbation of the (complex-valued) Ornstein-Uhlenbeck operator $\L_0$, which is studied in
depth in \cite[Sec.~3]{BeynOtten2016a} and 
\cite[Sec.~7]{Otten2014}. In Section \ref{subsec:2.2} we continue this study and determine the essential spectrum $\sigma_{\mathrm{ess}}(\L_Q)$ 
 in $L^p$ (see Theorem \ref{thm:EssentialSpectrumOfLQ}). An application of Theorem \ref{thm:EssentialSpectrumOfLQ} to $-B_{\infty}=Df(v_{\infty})$ and 
$Q(x)=Df(v_{\star}(x))-Df(v_{\infty})$ completes the proof of Theorem \ref{thm:EssSpecLRW}. For the proof of the decay \eqref{equ:4.2} 
we use \cite[Cor.~4.3]{BeynOtten2016a} to deduce that $v_{\star}(x)\to v_{\infty}$ as $|x|\to\infty$.

In Section \ref{sec:4} we analyze Fredholm properties of the linearized operator
\begin{equation}
  \label{equ:1.12e}
  \lambda I-\L:(\D^p_{\mathrm{loc}}(\L_0),\left\|\cdot\right\|_{\L_0})\rightarrow (L^p(\R^d,\C^m),\left\|\cdot\right\|_{L^p})
\end{equation}
with $\L$ given by \eqref{equ:1.8}, and of its adjoint operator
\begin{equation}
  \label{equ:1.12f}
  (\lambda I-\L)^*:(\D^q_{\mathrm{loc}}(\L_0^*),\left\|\cdot\right\|_{\L_0^*})\rightarrow (L^q(\R^d,\C^m),\left\|\cdot\right\|_{L^q}), \quad q=\frac{p}{p-1},
\end{equation}
defined by
\begin{equation}
  \label{equ:1.12c}
  \left[\L^* v\right](x) = A^{\herm} \triangle v(x)-\left\langle Sx,\nabla v(x)\right\rangle + Df\left(v_{\star}(x)\right)^{\herm} v(x),\,x\in\R^d,\,d\geqslant 2,
\end{equation}
For values $\Re\lambda>-\bzero$, with $b_0$ the spectral bound from \eqref{equ:aminamaxazerobzero}, we show that the operator 
$\lambda I-\L$ is Fredholm of index $0$. Moreover, we prove that its formal adjoint operator $(\lambda I-\L)^*$ from \eqref{equ:1.12c} 
and its abstract adoint operator (see Definition \ref{equ:DefAdj}), coincide on their common domain 
\begin{equation}
  \label{equ:4.3_q}
  \D^q_{\mathrm{loc}}(\L_0^*) = \left\{v\in W^{2,q}_{\mathrm{loc}}(\R^d,\C^m)\cap L^q(\R^d,\C^m):\; \L_0^* v\in L^q(\R^d,\C^m)\right\}.
\end{equation}
Then the Fredholm alternative applies and leads to the following result.

\begin{theorem}[Fredholm properties of the linearization $\L$]\label{thm:FredPropLin}
  Let the assumptions \eqref{cond:A4DC} and \eqref{cond:A5}--\eqref{cond:A8} be satisfied for $\K=\C$ and for some $1<p<\infty$. 
  Moreover, let $\lambda\in\C$ with $\Re\lambda\geqslant -\bzero+\gamma$ for some $\gamma>0$, where $-\bzero = s(Df(v_{\infty}))$ 
  denotes the spectral bound of $Df(v_{\infty})$.
  Then, for every $0<\varepsilon<1$ there is a constant $K_1=K_1(A,f,v_{\infty},\gamma,d,p,\varepsilon)>0$ such that 
  for every classical solution $v_{\star}\in C^2(\R^d,\R^m)$ of 
  \begin{align}
    \label{equ:NonlinearProblemRealFormulation}
    A\triangle v(x)+\left\langle Sx,\nabla v(x)\right\rangle+f(v(x))=0,\,x\in\R^d,
  \end{align}
  satisfying
  \begin{align}
    \label{equ:BoundednessConditionForVStar}
    \sup_{|x|\geqslant R_0}\left|v_{\star}(x)-v_{\infty}\right|\leqslant K_1\text{ for some $R_0>0$},
  \end{align} 
  the following statements hold:
  \begin{itemize}[leftmargin=0.43cm]\setlength{\itemsep}{0.1cm}
  \item[a)] (Fredholm properties of $\L$). The linearized operator
  \begin{equation*}
    \lambda I-\L:(\D^p_{\mathrm{loc}}(\L_0),\left\|\cdot\right\|_{\L_0})\rightarrow(L^p(\R^d,\C^m),\left\|\cdot\right\|_{L^p})
  \end{equation*}
  is Fredholm of index $0$.
  \item[b)] (Eigenvalues of $\L$). 
  In addition to a), let Assumption \eqref{cond:A4DCq} be satisfied for $q=\frac{p}{p-1}$, and let $\lambda\in\sigma_{\mathrm{pt}}(\L)$ with geometric 
  multiplicity $1\leqslant n:=\mathrm{dim}\,\mathcal{N}(\lambda I-\L)<\infty$. 
  Then $ \mathcal{N}((\lambda I-\L)^*)\subseteq \D^q_{\mathrm{loc}}(\L_0^*)$ also has dimension $n$ and the inhomogenous equation
  \begin{equation}
    \label{equ:InhomEqu}
    (\lambda I-\L)v=g \in L^p(\R^d,\C^m)
  \end{equation}
  has at least one (not necessarily unique) solution $v\in\D^p_{\mathrm{loc}}(\L_0)$ iff $g\in(\mathcal{N}((\lambda I-\L)^*))^{\bot}$, i.e.
  \begin{equation}
    \label{equ:OrthCond}
    \langle \psi,g\rangle_{q,p} = 0,\; \text{for all} \; \psi \in
\mathcal{N}((\lambda I-\L)^*).
  \end{equation}
  If the orthogonality condition \eqref{equ:OrthCond} is satisfied, then one can select a solution $v$ of \eqref{equ:InhomEqu} with
  \begin{equation}
    \label{equ:EstimateEF}
    \left\|v\right\|_{\L_0} \leqslant C\left\|g\right\|_{L^p},\quad \left\|v\right\|_{W^{1,p}} \leqslant C\left\|g\right\|_{L^p},
  \end{equation}
  where $C$ denotes a generic constant which does not depend on $g$.
  \end{itemize}
\end{theorem}

An extension of Theorem \ref{thm:FredPropLin} provides us exponential decay of eigenfunctions and of adjoint eigenfunctions for eigenvalues $\lambda\in\C$ 
with $\Re\lambda\geqslant -\bzero+\gamma$  (Theorem \ref{thm:ExpDecEigLin}), cf. \cite[Thm.~3.5]{BeynOtten2016a} for the case of eigenfunctions. 

The idea of proof for Theorem \ref{thm:FredPropLin} is to write $\lambda=\lambda_1+\lambda_2$ with $\lambda_2:=-\bzero+\gamma$, $\lambda_1:=\lambda-\lambda_2$,  
and to decompose the variable coefficient $Q=Q_{\mathrm{s}}+Q_{\mathrm{c}}$ into the sum of a function $Q_{\mathrm{s}}$ which is small with respect 
to $\left\|\cdot\right\|_{L^{\infty}}$ and a function $Q_{\mathrm{c}}$ which is compactly supported on $\R^d$. 
This allows us to decompose the differential operator $\lambda I-\L_Q$ as follows
\begin{equation}
  \label{equ:DecompositionLQ_intro}
  \lambda I-\L_Q = (I-Q_{\mathrm{c}}(\cdot)(\lambda_1 I-\widetilde{\L}_{\mathrm{s}})^{-1})(\lambda_1 I-\widetilde{\L}_{\mathrm{s}}),
\end{equation}
where $\widetilde{\L}_{\mathrm{s}}:=\L_{\mathrm{s}}-\lambda_2 I$ and $\L_{\mathrm{s}}$ denotes a small variable coefficient perturbation, defined by
\begin{equation}
  \label{equ:Ls}
  [\L_{\mathrm{s}} v](x) = A\triangle v(x) + \left\langle Sx,\nabla v(x)\right\rangle - B_{\infty}v(x) + Q_{\mathrm{s}}(x) v(x),\,x\in\R^d.
\end{equation}
For a similar decomposition under more restrictive assumptions on $B_{\infty}$ see \cite{BeynLorenz2008,BeynOtten2016a,Otten2014}. 
Then we show that $Q_{\mathrm{c}}(\cdot)(\lambda_1 I-\widetilde{\L}_{\mathrm{s}})^{-1}$ is compact and $\lambda_1 I-\widetilde{\L}_{\mathrm{s}}$ is Fredholm of index $0$, 
which implies $\lambda I-\L_Q$ to be Fredholm of index $0$. 
A crucial ingredient for the proof of these two statements is the inclusion $\D^p_{\mathrm{loc}}(\L_0)\subset W^{1,p}(\R^d,\C^m)$, 
proved in \cite[Thm.~5.8 \& 6.8]{Otten2014}, \cite[Thm.~5.7]{Otten2014a}.   
Further, it is essential to solve the identification problem 
for the adjoint operator of $\L_Q$ in $L^q(\R^d,\C^m)$
along the lines of \cite{Otten2015a} (Lemma \ref{lem:4.4}).
We show the existence and uniqueness of a solution $\tilde{v}\in\D^q_{\mathrm{loc}}(\L_0^*)$ 
of the resolvent equation $(\lambda I-\L_Q)^*\tilde{v}=g\in L^q(\R^d,\C^m)$, 
using the corresponding result from \cite[Thm.~3.1]{BeynOtten2016a}. 
For this we employ the $L^q$-dissipativity condition \eqref{cond:A4DCq} for the adjoint operator, 
which is known to be equivalent to the $L^q$-antieigenvalue condition \eqref{cond:A4q}, see \cite{Otten2015b}. Finally,  the Fredholm alternative 
is applied to $\lambda I-\L_Q$ and $(\lambda I-\L_Q)^*$ (Theorem \ref{lem:4.5}) and exponential decay of (adojoint) eigenfunctions is shown 
(Theorem \ref{thm:APrioriEstimatesInLpRelativelyCompactPerturbation}).
 These results hold for  $-B_{\infty}=Df(v_{\infty})$ and 
$Q(x)=Df(v_{\star}(x))-Df(v_{\infty})$ and thus  complete the proof of Theorem \ref{thm:FredPropLin} (and Theorem \ref{thm:ExpDecEigLin}).
Note that a similar reasoning is used in \cite{BeynOtten2016a,Otten2014} 
to prove exponential decay of the wave profile $v_{\star}$ itself.

In Section \ref{sec:3} we investigate the point spectrum $\sigma_{\mathrm{pt}}(\L)$ of $\L$, which is determined by the symmetries of the 
underlying $\SE(d)$-group action of dimension $\frac{d(d+1)}{2}$.
By the ansatz $v=(Dv_{\star})(Ex+b)$ for $E\in\C^{d,d}$, $E^{\top}=-E$ and $b\in\C^d$, eigenfunctions of the problem \eqref{equ:1.9} (in the classical sense) 
are obtained from the $\frac{d(d+1)}{2}$-dimensional eigenvalue problem (Section \ref{subsec:3.1})
\begin{align*}
  \lambda E & = [E,S], \\
  \lambda b & = -Sb,
\end{align*}
which is to be solved for $(\lambda,E,b)$ with $\lambda\in\C$, $E\in\C^{d,d}$, $E^{\top}=-E$ and $b\in\C^d$. Defining the \begriff{symmetry set}
\begin{equation}
  \label{equ:1.13}
  \sigma_{\mathrm{sym}}(\L) := \sigma(S) \cup \{\lambda_i^S+\lambda_j^S:1\leqslant i<j\leqslant d\},
\end{equation}
we show that $\sigma_{\mathrm{sym}}(\L)$ belongs to the point spectrum $\sigma_{\mathrm{pt}}(\L)$ of $\L$ in $L^p$, determine their associated eigenfunctions, 
and show that the eigenfunctions and their adjoint counterparts decay exponentially in space.

\begin{theorem}[Point spectrum on the imaginary axis and shape of eigenfunctions]\label{thm:3.1}
  Let $f\in C^1(\R^m,\R^m)$, $S\in\R^{d,d}$ be skew-symmetric, and let $U\in\C^{d,d}$ denote the unitary matrix satisfying $\Lambda_S=U^{\herm} SU$ with  
  diagonal matrix $\Lambda_S=\diag(\lambda_1^S,\ldots,\lambda_d^S)$ and eigenvalues $\lambda_1^S,\ldots,\lambda_d^S\in\sigma(S)$. Moreover, let $v_{\star}\in C^3(\R^d,\R^m)$ be a classical 
  solution of \eqref{equ:1.5}, then the function $v:\R^d\rightarrow\C^m$ given by
  \begin{align}
    \label{equ:3.1}
    v(x) = \left\langle Ex+b,\nabla v_{\star}(x)\right\rangle = (Dv_{\star}(x))(Ex+b)
  \end{align}
  is a classical solution of the eigenvalue problem \eqref{equ:1.9} if $E\in\C^{d,d}$ and $b\in\C^d$ either satisfy
  \begin{align}
    \label{equ:3.2}
    \lambda=-\lambda_l^S,\quad E=0,\quad b=Ue_l
  \end{align}
  for some $l=1,\ldots,d$, or
  \begin{align}
    \label{equ:3.3}
    \lambda=-(\lambda_i^S+\lambda_j^S),\quad E=U(I_{ij}-I_{ji})U^T,\quad b=0
  \end{align}
  for some $i=1,\ldots,d-1$ and $j=i+1,\ldots,d$. Here, $I_{ij}\in\R^{d,d}$ denotes the matrix having the entries $1$ at the $i$-th row and $j$-th column and 
  $0$ otherwise. All the eigenvalues above lie on the imaginary axis.
\end{theorem}

\begin{theorem}[Point spectrum at localized rotating waves]\label{thm:3.1b}
  Let $f\in C^{r-1}(\R^m,\R^m)$ for some $r\in\N$ with $r\geqslant 3$ and let the assumptions \eqref{cond:A4DC}, \eqref{cond:A5}, \eqref{cond:A7}, \eqref{cond:A8} 
  and \eqref{cond:A10} be satisfied for $\K=\C$ and for some $1<p<\infty$ with $\frac{d}{p}\leqslant 2$. 
  Then, for every $0<\varepsilon<1$ there is a constant $K_1=K_1(A,f,v_{\infty},d,p,\varepsilon)>0$ with the following property: 
  For every classical solution $v_{\star}\in C^{r+1}(\R^d,\R^m)$ of 
  \begin{align*}
    A\triangle v(x)+\left\langle Sx,\nabla v(x)\right\rangle+f(v(x))=0,\,x\in\R^d,
  \end{align*}
  satisfying
  \begin{align*}
    \sup_{|x|\geqslant R_0}\left|v_{\star}(x)-v_{\infty}\right|\leqslant K_1\text{ for some $R_0>0$,}
  \end{align*}
  the symmetry set $\sigma_{\mathrm{sym}}(\L)$ from \eqref{equ:1.13} belongs to the point spectrum $\sigma_{\mathrm{pt}}(\L)$ of the 
  linearized operator $\L$ from \eqref{equ:1.8} in $L^p(\R^d,\C^m)$, i.e. $\sigma_{\mathrm{sym}}(\L)\subseteq \sigma_{\mathrm{pt}}(\L)$ in $L^p(\R^d,\C^m)$.
\end{theorem}

One can extend Theorem \ref{thm:3.1b} by an application of Theorem \ref{thm:ExpDecEigLin} to show exponential decay of eigenfunctions and adjoint eigenfunctions, 
cf. \cite[Sec.~5]{BeynOtten2016a} for the case of eigenfunctions. The eigenvalue problem for the commutator generated by a skew-symmetric matrix, is analyzed for 
example in \cite[Lem.~4 \& 5]{BlochIserles2005} and \cite[Thm.~2]{TausskyWielandt19962}.
Further, we mention that the asymptotic behavior of adjoint eigenfunctions plays a role in the study of response functions, see \cite{BiktashevaHoldenBiktashev2006}.

For the proof of Theorem \ref{thm:3.1b}, we apply Theorem \ref{thm:FredPropLin}a) to eigenvalues $\lambda\in\sigma_{\mathrm{sym}}(\L)$. Here, exponential decay of the 
rotating wave $v_{\star}$, proved in \cite[Cor.~4.1]{BeynOtten2016a}, implies that $v$ from \eqref{equ:3.1} belongs to $\D^p_{\mathrm{loc}}(\L_0)$, and hence is an eigenfunction 
of $\L$ in $L^p$. Isolatedness and finite multiplicity of $\lambda$ is obtained from spectral stability of $Df(v_{\infty})$ assured by \eqref{cond:A10} and Theorem \ref{thm:FredPropLin}a).

In Section \ref{sec:5} we apply our results to the cubic-quintic complex Ginzburg-Landau equation
\begin{equation}
  \label{equ:1.15}
  u_t = \alpha\triangle u + u\left(\delta + \beta|u|^2 + \gamma|u|^4\right)
\end{equation}
which is known to exhibit spinning soliton solutions. We rewrite \eqref{equ:1.15} as a $2$-dimensional real-valued system and 
formulate the eigenvalue problem for the associated linearization at the spinning soliton. We then compute numerical spectra and eigenfunctions 
using the freezing method from \cite{BeynOttenRottmannMatthes2013,BeynThuemmler2004} and the software {\sc Comsol}, 
\cite{ComsolMultiphysics52}. This allows to compare exact and numerical spectra as well as their associated eigenfunctions.

Let us finally discuss some related results from the literature. Spectra of Ornstein-Uhlenbeck operators in various function spaces are studied in 
\cite{DaPratoLunardi1995,Kozhan2009,Metafune2001,MetafunePallaraPriola2002,VanNeerven2005}, spectra at localized rotating waves in \cite{BeynLorenz2008,Otten2014},
and spectra at spiral waves (nonlocalized rotating waves) in \cite{FiedlerScheel2003,Otten2014,SandstedeScheel2000,SandstedeScheel2001,SandstedeScheel2006,BarkleyWheeler2006}.
For scroll waves we refer to \cite{AlonsoBaerPanfilov2013,BarkleyFoulkesBiktashevBiktasheva2010}. Exponential decay is proved in \cite{BeynOtten2016a,Otten2014} for solutions of nonlinear problems  
for Ornstein-Uhlenbeck operators (with unbounded coefficients of $\nabla u$), while  \cite{GebranStuart2010,RabierStuart2000} treat solutions of real-valued quasilinear second-order equations 
(with bounded coefficients of $\nabla u$). We also refer to
 \cite{GebranStuart2010,PoganScheel2011,Volpert2011}
for various results on Fredholm properties of elliptic partial differential on
unbounded domains in settings different from ours.
Nonlinear stability of rotating waves is investigated in \cite{BeynLorenz2008,SandstedeScheelWulff1997}. For numerical approximations of rotating waves (including wave profiles, velocities and spectra), 
based on the freezing method from \cite{BeynOttenRottmannMatthes2013,BeynThuemmler2004}, we refer to \cite{BeynOtten2016a,Otten2014}. Numerical results on rotating 
waves are studied in \cite{ErmentroutVanderventel2016} for scalar excitable media, and in \cite{BeynOttenRottmannMatthes2016b} for second order evolution equations. 
Interactions of several rotating waves is analyzed numerically in \cite{BeynOttenRottmannMatthes2013,Otten2014}.
\sect{Essential spectrum and dispersion relation}
\label{sec:2}

\subsection{Formal derivation of the dispersion relation}
\label{subsec:2.1}
In this section we discuss the essential spectrum $\sigma_{\mathrm{ess}}(\L)$ 
of the linearization $\L$ from \eqref{equ:1.8}. We compute eigenvalues 
and bounded eigenfunctions of \eqref{equ:1.9} with $\L$ replaced by its 
far-field limit. These eigenvalues are determined by the dispersion relation 
\eqref{equ:1.11}. By a standard truncation procedure we then show that
these values are not isolated and lead to unbounded resolvents, hence
belong to the essential spectrum. We proceed in several steps:

\noindent
\textbf{1. The far-field operator.} 
Let $v_{\infty}\in\R^m$ denote the constant asymptotic state of the wave profile $v_{\star}$. i.e. $f(v_{\infty})=0$ and $v_{\star}(x)\to v_{\infty}\in\R^m$ as $|x|\to\infty$. 
Assuming $f\in C^1$ and introducing $Q(x)\in\R^{m,m}$ via
\begin{equation*}
  Q(x) := Df(v_{\star}(x))-Df(v_{\infty}),\quad x\in\R^d,
\end{equation*}
allows us to write \eqref{equ:1.9} as
\begin{equation}
  \label{equ:2.3}
  (\lambda I -\L_Q)v = 0,\,x\in\R^d
\end{equation}
with $\L_Q=\L_{\infty}+Q(x)$ and far-field operator
\begin{align*}
  \L_{\infty}v = A\triangle v + \left\langle Sx,\nabla v\right\rangle + Df(v_{\infty})v.
\end{align*}
Obviously, $f\in C^1$ and $v_{\star}(x)\to v_{\infty}$ imply $Q(x)\to 0$ as $|x|\to\infty$, i.e. $Q$ vanishes at infinity.

\noindent
\textbf{2. Orthogonal transformation.} 
We next transfer the skew-symmetric matrix $S$ into quasi-diagonal real form which allows us to separate the axes of rotations in \eqref{equ:2.3}.
Let $S\in\R^{d,d}$ be skew-symmetric, then $\sigma(S)\subset i\R$ with nonzero eigenvalues $\pm i\sigma_1,\ldots,\pm i\sigma_k$ and semisimple eigenvalue $0$ of multiplicity $d-2k$. 
Here, $\sigma_l$ denotes the angular velocity in the $(y_{2l-1},y_{2l})$-plane in one of the $k$ different planes of rotation. 
Moreover, there is an orthogonal matrix $P\in\R^{d,d}$ such that
\begin{equation*}
  S=P\Lambda P^{\top}\quad\text{with}\quad\Lambda=\mathrm{diag}\left(\Lambda_1,\ldots,\Lambda_k,\mathbf{0}\right),\quad
  \Lambda_j=\begin{pmatrix}0 &\sigma_j\\-\sigma_j &0\end{pmatrix},\quad 
  \mathbf{0}\in\R^{d-2k,d-2k}.
\end{equation*}
The orthogonal transformation $\tilde{v}(y):=v(T_1(y))$ with $x=T_1(y):=Py$ for $y\in\R^d$ transfers \eqref{equ:2.3} into
\begin{equation}
  \label{equ:2.4}
  (\lambda I-\tilde{\L}_Q)\tilde{v} = 0,\,y\in\R^d
\end{equation}
with $\tilde{\L}_Q=\tilde{\L}_{\infty}+Q(T_1(y))$ and 
\begin{equation*}
  \tilde{\L}_{\infty}\tilde{v} 
  = A\sum_{j=1}^{d}\partial_{y_j}^2\tilde{v} + \sum_{l=1}^{k}\sigma_l\left(y_{2l}\partial_{y_{2l-1}}-y_{2l-1}\partial_{y_{2l}}\right)\tilde{v} + Df(v_{\infty})\tilde{v}.
\end{equation*}

\noindent
\textbf{3. Transformation into several planar polar coordinates.} 
Since we have $k$ angular derivatives in $k$ different planes it is advisable to transform each plane into planar polar coordinates via
\begin{equation*}
  \begin{pmatrix}y_{2l-1}\\y_{2l}\end{pmatrix} = T(r_l,\phi_l):=\begin{pmatrix}r_l\cos\phi_l\\r_l\sin\phi_l\end{pmatrix},\quad r_l>0,\;\phi_l\in[-\pi,\pi),\;l=1,\ldots,k.
\end{equation*}
All further coordinates, $y_{2k+1},\ldots,y_d$, remain fixed. The multiple planar polar coordinates transformation 
$\hat{v}(\psi):=\tilde{v}(T_2(\psi))$ with $T_2(\psi) = (T(r_1,\phi_1),\ldots,T(r_k,\phi_k),y_{2k+1},\ldots,y_d)$, $\psi=(r_1,\phi_1,\ldots,r_k,\phi_k,y_{2k+1},\ldots,y_d)$ 
in the domain $\Omega:=((0,\infty)\times[-\pi,\pi))^k\times\R^{d-2k}$, transfers \eqref{equ:2.4} into
\begin{equation}
  \label{equ:2.5}
  (\lambda I-\hat{\L}_Q)\hat{v}=0,\,\psi\in\Omega
\end{equation}
with $\hat{\L}_Q=\hat{\L}_{\infty}+Q(T_1(T_2(\psi)))$ and
\begin{align*}
  \hat{\L}_{\infty}\hat{v} = A\bigg[\sum_{l=1}^{k}\bigg(\partial_{r_l}^2+\frac{1}{r_l}\partial_{r_l}+\frac{1}{r_l^2}\partial_{\phi_l}^2\bigg)+\sum_{l=2k+1}^{d}\partial_{y_l}^2\bigg]\hat{v} - \sum_{l=1}^{k}\sigma_l\partial_{\phi_l}\hat{v} + Df(v_{\infty})\hat{v}. 
\end{align*}

\noindent
\textbf{4. Simplified operator (limit operator, far-field operator).} 
Since the essential spectrum depends on the limiting equation for $|x|\to\infty$, we formally let $r_l\to\infty$ for any $1\leqslant l\leqslant k$. This turns \eqref{equ:2.5} into
\begin{equation}
  \label{equ:2.6}
  (\lambda I - \L_{\infty}^{\mathrm{sim}})\hat{v} = 0,\quad \psi\in\Omega
\end{equation}
with the simplified far-field operator
\begin{align*}
  \L_{\infty}^{\mathrm{sim}}\hat{v} = A\bigg[\sum_{l=1}^{k}\partial_{r_l}^2 + \sum_{l=2k+1}^{d}\partial_{y_l}^2\bigg]\hat{v}
         -\sum_{l=1}^{k}\sigma_l \partial_{\phi_l}\hat{v} + Df(v_{\infty})\hat{v}.
\end{align*}
Note that we used the property $\left|Q(x)\right|\to 0$ as $|x|\to\infty$ which was established in step 1.

\noindent
\textbf{5. Angular Fourier transform.}
Finally, we solve for eigenvalues and eigenfunctions of $\L_{\infty}^{\mathrm{sim}}$ by an angular Fourier decomposition (separation of variables) with 
$\omega\in\R^k$, $\rho,y\in\R^{d-2k}$, $n\in\Z^k$, $z\in\C^m$, $|z|=1$, $r\in(0,\infty)^k$, $\phi\in[-\pi,\pi)^k$:
\begin{align}
  \label{equ:2.7a}
  \hat{v}(\psi) = \exp\bigg(i\sum_{l=1}^{k}\omega_l r_l\bigg)\exp\bigg(i\sum_{l=1}^{k}n_l\phi_l\bigg)\exp\bigg(i\sum_{l=2k+1}^{d}\rho_l y_l\bigg)z 
                = \exp\Big(i\langle\omega,r\rangle + i\langle n,\phi\rangle + i\langle\rho,y\rangle\Big)z.
\end{align}
Inserting \eqref{equ:2.7a} into \eqref{equ:2.6} leads to the $m$-dimensional eigenvalue problem
\begin{equation}
  \label{equ:2.9}
  \big(\lambda I_m+(|\omega|^2+|\rho|^2)A+i\langle n,\sigma\rangle I_m-Df(v_{\infty})\big)z = 0.
\end{equation}

\textbf{6. Dispersion relation and dispersion set.} The \begriff{dispersion relation} for localized rotating waves of \eqref{equ:1.1} now states that 
every $\lambda\in\C$ satisfying
\begin{equation}
  \label{equ:6.15}
  \det\big(\lambda I_m+(|\omega|^2+|\rho|^2)A+i\langle n,\sigma\rangle I_m-Df(v_{\infty})\big) = 0
\end{equation}
for some $\omega\in\R^k$, $\rho\in\R^{d-2k}$ and $n\in\Z^k$ belongs to the essential spectrum of $\L$, i.e. $\lambda\in\sigma_{\mathrm{ess}}(\L)$. 
Of course, one can replace $|\omega|^2+|\rho|^2$ by any nonnegative real number, cf. \eqref{equ:1.11}. Defining the \begriff{dispersion set}
\begin{equation}
  \label{equ:6.16}
  \sigma_{\mathrm{disp}}(\L):=\{\lambda\in\C:\;\lambda\text{ satisfies \eqref{equ:6.15} for some $\omega\in\R^k$, $\rho\in\R^{d-2k}$ and $n\in\Z^k$}\},
\end{equation}
the dispersion relation formally states that $\sigma_{\mathrm{disp}}(\L)\subseteq\sigma_{\mathrm{ess}}(\L)$. 
A rigorous justification of this result is shown in the next section using suitable function spaces (see Theorem \ref{thm:EssSpecLRW}).
Our formal calculation shows that the dispersion relation is a helpful tool to locate the essential spectrum and to verify its existence. 

\begin{remark}
\begin{itemize}[leftmargin=0.43cm]\setlength{\itemsep}{0.1cm}
  \item[a)] (Several axis of rotation).
  For space dimensions $d\geqslant 3$ the axis of rotation is in general not orthogonal to some plane $(x_i,x_j)$, 
  $1\leqslant i,j\leqslant d$. In particular, for space dimensions $d\geqslant 4$ the pattern can rotate rigidly around several axes of rotation simultaneously. 
  Therefore, the idea of step 2 is to separate the axes of rotation in such a way that they are orthogonal to (completely) different planes.
  \item[b)] (Dispersion relation for spiral waves). 
  The dispersion relation for nonlocalized rotating waves, such as spiral waves and scroll waves, is harder to derive and 
  differs from \eqref{equ:6.15}. A dispersion relation for spiral waves is developed in \cite{FiedlerScheel2003,SandstedeScheel2000,SandstedeScheel2001}. 
  Their approach is based on a Bloch wave transformation and on an application of Floquet theory. A summary of these results, 
  which is structured similar to the derivation above, can be found in \cite[Sec.~9.5]{Otten2014}. The angular Fourier decomposition is also used 
  in \cite{FiedlerScheel2003} for investigating essential spectra of spiral waves.
  For spectra of spiral waves in the FitzHugh-Nagumo system we refer to \cite{SandstedeScheel2006,BarkleyWheeler2006}. 
  Results on essential spectra of nonlocalized rotating waves for space dimensions $d\geqslant 3$, such as scroll waves, are quite rare in the literature 
  and we refer to \cite{FiedlerScheel2003,BarkleyFoulkesBiktashevBiktasheva2010} and references therein.
  \end{itemize}
\end{remark}

\subsection{Essential spectrum in $L^p$}
\label{subsec:2.2}
We now analyze the essential $L^p$-spectrum of the differential operator
\begin{equation*}
  \L_Q v = A\triangle v + \left\langle Sx,\nabla v\right\rangle -B_{\infty}v + Q(x)v,
\end{equation*}
satisfying $|Q(x)|\to 0$ as $|x|\to\infty$, in which case the dispersion relation reads as
\begin{equation}
  \label{equ:6.17}
  \det\big(\lambda I_m+(|\omega|^2+|\rho|^2)A+i\langle n,\sigma\rangle I_m+B_{\infty}\big)=0
\end{equation}
for some $\omega\in\R^k$, $\rho\in\R^{d-2k}$ and $n\in\Z^k$. The following Theorem shows that the dispersion set
\begin{equation}
  \label{equ:6.18}
  \sigma_{\mathrm{disp}}(\L_Q):=\left\{\lambda\in\C:\;\lambda\text{ satisfies \eqref{equ:6.17} for some $\omega\in\R^k$, $\rho\in\R^{d-2k}$ and $n\in\Z^k$}\right\},
\end{equation}
belongs to the essential spectrum $\sigma_{\mathrm{ess}}(\L_Q)$ of $\L_Q$ in $L^p(\R^d,\C^m)$ for $1<p<\infty$. 
Note that it does not prove equality $\sigma_{\mathrm{disp}}(\L_Q)=\sigma_{\mathrm{ess}}(\L_Q)$. 
Applying this result to
\begin{equation}
  \label{equ:EssSpecSettingLRW}
  -B_{\infty} = Df(v_{\infty}),\quad Q(x)=Df(v_{\star}(x))-Df(v_{\infty}),\quad x\in\R^d,
\end{equation}
then implies the inclusion $\sigma_{\mathrm{disp}}(\L)\subseteq\sigma_{\mathrm{ess}}(\L)$ for the linearized operator $\L$ (Theorem \ref{thm:EssSpecLRW}). 

\begin{theorem}[Essential spectrum of $\L_Q$]\label{thm:EssentialSpectrumOfLQ}
  Let the assumptions \eqref{cond:A4DC}, \eqref{cond:A5}, \eqref{cond:A8B} and
  \begin{equation}
    \label{equ:PropertyQ}
    Q\in L^{\infty}(\R^d,\K^{m,m})\quad\text{with}\quad\eta_R:=\underset{|x|\geqslant R}{\esssup}\left|Q(x)\right|\rightarrow 0\text{ as }R\rightarrow\infty
  \end{equation}
  be satisfied for $1<p<\infty$ and $\K=\C$. Moreover, let $\pm i\sigma_1,\ldots,\pm i\sigma_k$ with $\sigma_1,\ldots,\sigma_k\in\R$ denote the nonzero eigenvalues of $S$. 
  Then the dispersion set $\sigma_{\mathrm{disp}}(\L_Q)$ from \eqref{equ:6.18} belongs to the essential spectrum $\sigma_{\mathrm{ess}}(\L_Q)$ of $\L_Q$ in $L^p(\R^d,\C^m)$, 
  i.e. $\sigma_{\mathrm{disp}}(\L_Q)\subseteq\sigma_{\mathrm{ess}}(\L_Q)$ in $L^p(\R^d,\C^m)$.
\end{theorem}
Let us first discuss some consequences of Theorem \ref{thm:EssentialSpectrumOfLQ}.
\begin{remarks}
  \begin{itemize}[leftmargin=0.43cm]\setlength{\itemsep}{0.1cm}
  \item[a)] (Location of the dispersion set and the effect of assumption \eqref{cond:A9B}).
  From the dispersion relation \eqref{equ:6.17} and conditions
\eqref{cond:A3}, \eqref{cond:A8B} one infers $\sigma_{\mathrm{disp}}(\L_Q)\subseteq\C_{b_0}$, where
$\C_{b_0}=\left\{\lambda\in\C:\,\Re\lambda\leqslant -\bzero\right\}$
  and $-\bzero=s(-B_{\infty})$ is the spectral bound of $-B_{\infty}$. If
 in addition, the stability condition \eqref{cond:A9B} holds, then $-\bzero=s(-B_{\infty})<0$ and
 $\sigma_{\mathrm{disp}}(\L_Q)$ is located in the left half-plane.
  
  \item[b)] (Density of dispersion set in a half-plane).
  Moreover, if there exist indices $n,j \in \{1,\ldots,k\}$ such that 
$\sigma_j\neq 0$  and $\sigma_n\sigma_j^{-1}\notin\mathbb{Q}$, then $\sigma_{\mathrm{disp}}(\L_Q)$ is dense 
  in the half-plane $\C_{b_0}$, 
  which implies  $\sigma_{\mathrm{ess}}(\L_Q)=\C_{b_0}$. If on the other hand, $\sigma_n\sigma_j^{-1}\in\mathbb{Q}$
for all $n,j$, then 
  the dispersion set $\sigma_{\mathrm{disp}}(\L_Q)$ is a discrete subgroup of 
$\C_{b_0}$ which is independent of $p$. 
  The reason for this conclusion is given by Metafune in \cite[Thm.~2.6]{Metafune2001}. Therein it is proved that the essential spectrum of the drift term  $v\mapsto\left\langle S\cdot,\nabla v\right\rangle)$,
agrees with $i\R$, if and only if there exists $0\neq\sigma_n,\sigma_j\in\R$ such that 
  $\sigma_n\sigma_j^{-1}\notin\mathbb{Q}$. Otherwise, the essential spectrum
is a discrete subgroup of $i\R$ which is 
  independent of $p$.

  \begin{figure}[ht]
      \centering      
      \subfigure[$d\in\{2,3\}$, not dense]{\includegraphics[page=1, height=3.5cm] {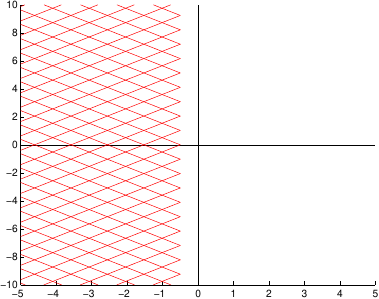} \label{fig:essentialspectrum_dequal2or3}}
      \subfigure[$d=4$, not dense]{\includegraphics[page=2, height=3.5cm] {Images.pdf} \label{fig:essentialspectrum_dequal4(notdense)}}
      \subfigure[$d=4$, dense]{\includegraphics[page=3, height=3.5cm] {Images.pdf} \label{fig:essentialspectrum_dequal4(dense)}}
      \caption{Dispersion set $\sigma_{\mathrm{disp}}(\L_Q)$ of $\L_Q$ from \eqref{equ:1.12a} for parameters $A=\frac{1}{2}\left(1+i\right)$, $B_{\infty}=\frac{1}{2}$ and $Q=0$.}
      \label{fig:EssentialSpectrum}
  \end{figure}

  \noindent
  Figure \ref{fig:EssentialSpectrum} illustrates the set $\sigma_{\mathrm{disp}}(\L_Q)$ in the scalar complex case for $A=\frac{1}{2}\left(1+i\right)$, $B_{\infty}=\frac{1}{2}$ and $Q=0$. 
  Figure \ref{fig:EssentialSpectrum}(a) shows $\sigma_{\mathrm{disp}}(\L_Q)$
for $\sigma_1=1.027$ and space dimension $d=2$ and $d=3$ (see the examples
in Section \ref{sec:5}). 
  In this case $\sigma_{\mathrm{disp}}(\L_Q)$ forms a zig-zag curve, see \cite{BeynLorenz2008} for $d=2$, and is not dense in 
  $\C_{-\frac{1}{2}}$. Note that density of $\sigma_{\mathrm{disp}}(\L_Q)$ can only occur 
   for space dimensions $d\geqslant 4$. Figures \ref{fig:EssentialSpectrum}(b)(c) show two such cases for $d=4$.
  In the first case $\sigma_1=1$, $\sigma_2=1.5$, hence
 $\sigma_1\sigma_2^{-1}\in\mathbb{Q}$ and  
  $\sigma_{\mathrm{disp}}(\L_Q)$ is not dense in $\C_{- \frac{1}{2}}$. The second case belongs to 
   $\sigma_1=1$, $\sigma_2=\frac{1}{2}\exp(1)$ for which
density occurs. 
 This shows  that $\sigma_{\mathrm{disp}}(\L_Q)$  
  may change dramatically with the eigenvalues of $S$.

  \item[c)] (Sectorial operators and analytic semigroups). 
  For $S\neq 0$, Theorem \ref{thm:EssentialSpectrumOfLQ} implies that the operator $\L_Q$ is not sectorial in $L^p(\R^d,\C^m)$, and the 
  corresponding semigroup is not analytic on $L^p(\R^d,\C^m)$ for every $1<p<\infty$, see \cite[Cor.~7.10]{Otten2014}. For the scalar real-valued case 
  we refer to \cite{Metafune2001,PruessRhandiSchnaubelt2006,VanNeervenZabczyk1999}.
  \end{itemize}
\end{remarks}

\begin{proof}
Let $R\geqslant 2$ and let $\chi_R\in C^{\infty}_{\mathrm{c}}([0,\infty),\R)$ be a cut-off function with derivatives bounded uniformly w.r.t. $R$, and
\begin{align*}
    \chi_R(r)= 0, \ r\in I_1 \cup I_5,  \quad
 \chi_R(r)=1,   \ r\in I_3, \quad
    \chi_R(r)          \in[0,1],\ r\in I_2 \cup I_4, 
   \end{align*}
$I_1=[0,R-1]$, $I_2=[R-1,R]$, $I_3=[R,2R]$, $I_4=[2R,2R+1]$, $I_5=[2R+1,\infty)$. Introducing
\begin{align*}
  w_R:=\frac{v_R}{\left\|v_R\right\|_{L^p}},\quad\;
  v_R(T_1(T_2(\psi))) := \hat{v}_R(\psi),\quad\;
  \hat{v}_R(\psi) := \bigg(\prod_{l=1}^{k}\chi_R(r_l)\bigg)\chi_R(|\tilde{y}|)\hat{v}(\psi),\quad\;
  \text{$\hat{v}$ from \eqref{equ:2.7a},} 
\end{align*}
for $\psi=\left(r_1,\phi_1,\ldots,r_k,\phi_k,\tilde{y}\right)$, $\tilde{y}=\left(y_{2k+1},\ldots,y_d\right)$, $\phi=(\phi_1,\ldots,\phi_k)\in[-\pi,\pi)^k$, 
$\mathbf{r}=(r_1,\ldots,r_k)\in(0,\infty)^k$, and $T_1,T_2$ as in Section \ref{subsec:2.1}, we have $w_R\in\D^p_{\mathrm{loc}}(\L_0)$ and show
\begin{align*}
             \left\|\left(\lambda I-\L_Q\right)w_R\right\|_{L^p}^p
          =  \frac{\left\|\left(\lambda I-\L_Q\right)v_R\right\|_{L^p}^p}{\left\|v_R\right\|_{L^p}^p} 
  \leqslant  \frac{CR^{d-1}+CR^d\eta_R}{CR^d}
          =  \frac{C}{R}+\eta_R\rightarrow 0\text{ as }R\rightarrow\infty.
\end{align*}
By this estimate continuity of the resolvent cannot hold, hence $\lambda\in\sigma(\L_Q)$. Further,
varying $\omega$ or $\rho$ in \eqref{equ:6.17} shows that $\lambda$ is not isolated, so that $\lambda\notin\sigma_{\mathrm{pt}}(\L_Q)$ and 
therefore $\lambda\in\sigma_{\mathrm{ess}}(\L_Q)$. It remains 
to verify
\begin{align*}
  \left\|v_R\right\|_{L^p}^p\geqslant CR^d,\quad\; \left\|\left(\lambda I-\L_Q\right)v_R\right\|_{L^p}^p\leqslant CR^{d-1}+CR^d\eta_R.
\end{align*}

\begin{itemize}[leftmargin=0.43cm]\setlength{\itemsep}{0.1cm}
\item[1.] The property $\chi_R(r)=0$ for $r\in I_1\cup I_5$ implies
\begin{align}
  \label{equ:EssentialSpectrumProperty1}
  (\lambda I-\L_{\infty}^{\mathrm{sim}})\hat{v}_R(\psi) = 0,\quad\text{if $|\tilde{y}|\in I_1\cup I_5$ or $r_l\in I_1\cup I_5$ for some $1\leqslant l\leqslant k$.}
\end{align}
Similarly, $\lambda\in\sigma_{\mathrm{disp}}(\L_Q)$ and $\chi_R''(r)=\chi_R'(r)=0$ for $r\in I_3$ imply
\begin{align}
  \label{equ:EssentialSpectrumProperty2}
  (\lambda I-\L_{\infty}^{\mathrm{sim}})\hat{v}_R(\psi) = 0,\quad\text{if $|\tilde{y}|\in I_3$ and $r_l\in I_3$ for every $1\leqslant l\leqslant k$.}
\end{align}
Next compute the partial derivatives 
\begin{align*}
  &\partial_{r_l}^2 \hat{v}_R(\psi) = \bigg[\frac{\chi_R''(r_l)}{\chi_R(r_l)}+2i\omega_l\frac{\chi_R'(r_l)}{\chi_R(r_l)}-\omega_l^2\bigg]\hat{v}_R(\psi),\,l=1,\ldots,k, \\
  &\partial_{y_l}^2 \hat{v}_R(\psi) = \bigg[\frac{y_l^2}{|\tilde{y}|^2}\frac{\chi_R''(|\tilde{y}|)}{\chi_R(|\tilde{y}|)}+\bigg(\frac{|\tilde{y}|^2-y_l^2}{|\tilde{y}|^3}+2i\rho_l\frac{y_l}{|\tilde{y}|}\bigg)\frac{\chi_R'(|\tilde{y}|)}{\chi_R(|\tilde{y}|)}-\rho_l^2\bigg]\hat{v}_R(\psi),\,l=2k+1,\ldots,d,
\end{align*}
and consider the case $|\tilde{y}|\in I_2\cup I_3\cup I_4$ and $r_l\in I_2\cup I_3\cup I_4$ for all $1\leqslant l\leqslant k$.
Then we use $\lambda\in\sigma_{\mathrm{disp}}(\L_Q)$ and the estimates $\left|\chi_R(r)\right|\leqslant 1$, $\chi_R'(r)\leqslant\left\|\chi_R\right\|_{C^2_{\mathrm{b}}}$, 
$\chi_R''(r)\leqslant\left\|\chi_R\right\|_{C^2_{\mathrm{b}}}$, $\left|\hat{v}(\psi)\right|=1$,  $\big|\tfrac{\hat{v}_R(\psi)}{\chi_R(r_l)}\big|\leqslant 1$, 
 $\big|\tfrac{\hat{v}_R(\psi)}{\chi_R(|\tilde{y}|)}\big|\leqslant 1$, $\frac{1}{|\tilde{y}|}\leqslant\frac{1}{R-1}\leqslant 1$ to obtain
\begin{align}
           & \left|\left(\lambda I-\L_{\infty}^{\mathrm{sim}}\right)v_{R}(\psi)\right| 
          =  \bigg|\bigg(\lambda I-A\bigg[\sum_{l=1}^{k}\partial_{r_l}^2+\sum_{l=2k+1}^{d}\partial_{y_l}^2\bigg]+\sum_{l=1}^{k}\sigma_l\partial_{\phi_l}+B_{\infty}\bigg)\hat{v}_R(\psi)\bigg|\nonumber\\
          =& \bigg|\Big(\lambda I_m-(|\omega|^2+|\rho|^2)A+i\langle n,\sigma\rangle+B_{\infty}\Big)v_{R}(\psi) 
             -A\sum_{l=1}^{k}\left(\chi_R''(r_l)+2i\omega_l\chi_R'(r_l)\right)\frac{\hat{v}_R(\psi)}{\chi_R(r_l)} \nonumber\\
           & -A\sum_{l=2k+1}^{d}\bigg(\frac{y_l^2}{|\tilde{y}|^2}\chi_R''(|\tilde{y}|)+\bigg(\frac{|\tilde{y}|^2-y_l^2}{|\tilde{y}|^3}+2i\rho_l\frac{y_l}{|\tilde{y}|}\bigg)\chi_R'(|\tilde{y}|)\bigg)\frac{\hat{v}_R(\psi)}{\chi_R(|\tilde{y}|)}\bigg| \label{equ:step1}\\
  \leqslant& |A|\sum_{l=1}^{k}\left(1+2|\omega_l|\right)\left\|\chi_R\right\|_{C^2_{\mathrm{b}}}
             +|A|\sum_{l=2k+1}^{d}\left(3+2|\rho_l|\right)\left\|\chi_R\right\|_{C^2_{\mathrm{b}}} \nonumber\\
  \leqslant& |A|\left(k+2|\omega|\sqrt{k}+3(d-2k)+2|\rho|\sqrt{d-2k}\right)\left\|\chi_R\right\|_{C^2_{\mathrm{b}}} =:C.\nonumber
\end{align}

\item[2.] Transforming variables, setting $\langle \mathbf{r} \rangle =\prod_{l=1}^k r_l$ and using $|\hat{v}(\psi)|=1$, $\chi_R(r)\geqslant 0$ $(r\in I_2\cup I_4)$, 
$\chi_R(r)=1$ $(r\in I_3)$ leads to
\begin{align*}
           & \left\|v_R\right\|_{L^p}^p 
          =  \int_{\R^d}\left|v_R(x)\right|^p dx 
          = \int_{0}^{\infty}\int_{-\pi}^{\pi}\cdots\int_{0}^{\infty}\int_{-\pi}^{\pi}\int_{\R^{d-2k}}\langle \mathbf{r} \rangle\left|\hat{v}_R(\psi)\right|^p d\psi \\
          =& \int_{R-1}^{2R+1}\int_{-\pi}^{\pi}\cdots\int_{R-1}^{2R+1}\int_{-\pi}^{\pi}\int_{R-1\leqslant|\tilde{y}|\leqslant 2R+1}\langle \mathbf{r} \rangle\left|\hat{v}_R(\psi)\right|^p d\psi \\
          =& \int_{R-1}^{2R+1}\int_{-\pi}^{\pi}\cdots\int_{R-1}^{2R+1}\int_{-\pi}^{\pi}\int_{R-1\leqslant|\tilde{y}|\leqslant 2R+1}\langle \mathbf{r} \rangle\bigg(\prod_{l=1}^{k}\chi_R^p(r_l)\bigg)
             \chi_R^p(|\tilde{y}|) d\psi \\
          =& \int_{R-1\leqslant|\tilde{y}|\leqslant 2R+1}\chi_R^p(|\tilde{y}|)d\tilde{y}\prod_{l=1}^{k}\int_{R-1}^{2R+1}\int_{-\pi}^{\pi}r_l\chi_R^p(r_l)d\phi_l dr_l \\
          =& \bigg(\int_{R-1\leqslant|\tilde{y}|\leqslant R}\chi_R^p(|\tilde{y}|)d\tilde{y}+\int_{R\leqslant|\tilde{y}|\leqslant 2R}\chi_R^p(|\tilde{y}|)d\tilde{y}
             +\int_{2R\leqslant|\tilde{y}|\leqslant 2R+1}\chi_R^p(|\tilde{y}|)d\tilde{y}\bigg) \\
           & \cdot\prod_{l=1}^{k}2\pi\bigg(\int_{R-1}^{R}r_l\chi_R^p(r_l)dr_l+\int_{R}^{2R}r_l\chi_R^p(r_l)dr_l+\int_{2R}^{2R+1}r_l\chi_R^p(r_l)dr_l\bigg) \\
  \geqslant& \bigg(\int_{R\leqslant|\tilde{y}|\leqslant 2R}1d\tilde{y}\bigg)\cdot\bigg(\prod_{l=1}^{k}2\pi\int_{R}^{2R}r_ldr_l\bigg)
          =  CR^{\tilde{d}}\prod_{l=1}^{k}3\pi R^2 = (3\pi)^k C R^{2k+\tilde{d}} = CR^{d},
\end{align*}
where $C$ is independent of $R$, $d\psi:=d\tilde{y}d\phi_k dr_k\cdots d\phi_1 dr_1$ and $\tilde{d}:=d-2k$. In the trivial case $\tilde{d}=0$ the first
integral is set to $1$, while in case $\tilde{d}\geqslant 1$ the term
$CR^{\tilde{d}}$ follows from the well-known formula
\begin{align}
  \label{equ:EssentialSpectrumProperty3}
  \tilde{d}\Gamma\big(\tfrac{\tilde{d}}{2}\big)\int_{a\leqslant|\tilde{y}|\leqslant b}1d\tilde{y}= 
2\pi^{\frac{\tilde{d}}{2}}
(b^{\tilde{d}}-a^{\tilde{d}}) \quad \text{for}\; 0< a < b < \infty.
\end{align}
\item[3.] The transformation theorem and \eqref{equ:EssentialSpectrumProperty1} imply
\begin{align*}
  & \left\|\left(\lambda I-\L_{\infty}^{\mathrm{sim}}\right)v_R\right\|_{L^p}^p
 =  \int_{\R^d}\left|\left(\lambda I-\L_{\infty}^{\mathrm{sim}}\right)v_R(x)\right|^p dx \\
 =& \int_{0}^{\infty}\int_{-\pi}^{\pi}\cdots\int_{0}^{\infty}\int_{-\pi}^{\pi}\int_{\R^{d-2k}}\langle \mathbf{r} \rangle\left|\left(\lambda I-\L_{\infty}^{\mathrm{sim}}\right)\hat{v}_R(\psi)\right|^p
    d\psi\\
 =& \int_{R-1}^{2R+1}\int_{-\pi}^{\pi}\cdots\int_{R-1}^{2R+1}\int_{-\pi}^{\pi}\int_{R-1\leqslant|\tilde{y}|\leqslant 2R+1} \langle \mathbf{r} \rangle 
    \left|\left(\lambda I-\L_{\infty}^{\mathrm{sim}}\right)\hat{v}_R(\psi)\right|^p d\psi.
\end{align*}
 We distinguish the following cases for $\tilde{d}=d-2k$.\\
\textbf{Case 1:} ($\tilde{d}=0$). From \eqref{equ:step1}, \eqref{equ:EssentialSpectrumProperty2}, the multinomial theorem and
\begin{align}
  \label{equ:EssentialSpectrumProperty4}
  \int_{R-1}^{R}r_l dr_l=\frac{1}{2}(2R-1),\quad\; \int_{R}^{2R}r_l dr_l=\frac{1}{2}3R^2,\quad\; \int_{2R}^{2R+1}r_l dr_l=\frac{1}{2}(4R+1),
\end{align}
we further obtain
\begin{align*}
          =& \int_{R-1}^{2R+1}\int_{-\pi}^{\pi}\cdots\int_{R-1}^{2R+1}\int_{-\pi}^{\pi}\
          \langle \mathbf{r} \rangle 
             \left|\left(\lambda I-\L_{\infty}^{\mathrm{sim}}\right)\hat{v}_R(\psi)\right|^p d\phi_k dr_k\cdots d\phi_1 dr_1 \\
  \leqslant& \sum_{j_1+j_2+j_3=k\atop j_2\neq k}\begin{pmatrix}k\\j_1,j_2,j_3\end{pmatrix}\left(\int_{R-1}^{R}\right)^{j_1}\left(\int_{R}^{2R}\right)^{j_2}\left(\int_{2R}^{2R+1}\right)^{j_3}
             C^p \langle \mathbf{r} \rangle (2\pi)^k dr_1\cdots dr_k \\
          =& \sum_{j_1+j_2+j_3=k\atop j_2\neq k}\begin{pmatrix}k\\j_1,j_2,j_3\end{pmatrix} \frac{C^p (2\pi)^k}{2^k} (2R-1)^{j_1}(3R^2)^{j_2}(4R+1)^{j_3}
  \leqslant  CR^{d-1}.
\end{align*}
For the last inequality we estimate powers of $R$ by
$j_1+2 j_2+j_3 = k+j_2 \leqslant 2k-1=d-1$ for $j_2 \neq k$.

\noindent
\textbf{Case 2:} ($\tilde{d}\geqslant 1$). Similarly, using \eqref{equ:step1} and
\eqref{equ:EssentialSpectrumProperty2}, \eqref{equ:EssentialSpectrumProperty3}, the multinomial theorem gives (abbreviating $d\mathbf{r}:=dr_1\cdots dr_k$ )
\begin{align*}
  \leqslant& \sum_{j_1+j_2+j_3=k}\begin{pmatrix}k\\j_1,j_2,j_3\end{pmatrix}\left(\int_{R-1}^{R}\right)^{j_1}\left(\int_{R}^{2R}\right)^{j_2}\left(\int_{2R}^{2R+1}\right)^{j_3} 
           \int_{R-1\leqslant|\tilde{y}|\leqslant R}C^p \langle \mathbf{r} \rangle (2\pi)^k d\tilde{y}d\mathbf{r} \\
           &+\sum_{j_1+j_2+j_3=k\atop j_2\neq k}\begin{pmatrix}k\\j_1,j_2,j_3\end{pmatrix}\left(\int_{R-1}^{R}\right)^{j_1}\left(\int_{R}^{2R}\right)^{j_2}\left(\int_{2R}^{2R+1}\right)^{j_3} 
           \int_{R\leqslant|\tilde{y}|\leqslant 2R}C^p \langle \mathbf{r} \rangle (2\pi)^k d\tilde{y}d\mathbf{r} \\
           &+\sum_{j_1+j_2+j_3=k}\begin{pmatrix}k\\j_1,j_2,j_3\end{pmatrix}\left(\int_{R-1}^{R}\right)^{j_1}\left(\int_{R}^{2R}\right)^{j_2}\left(\int_{2R}^{2R+1}\right)^{j_3} 
           \int_{2R\leqslant|\tilde{y}|\leqslant 2R+1}C^p \langle \mathbf{r} \rangle (2\pi)^k d\tilde{y}d\mathbf{r} \\
          =& \sum_{j_1+j_2+j_3=k}\begin{pmatrix}k\\j_1,j_2,j_3\end{pmatrix}\frac{C^p (2\pi)^k}{2^k}(2R-1)^{j_1}(3R^2)^{j_2}(4R+1)^{j_3} 
           \frac{2\pi^{\frac{\tilde{d}}{2}}}{\tilde{d}\Gamma\left(\frac{\tilde{d}}{2}\right)}(R^{\tilde{d}}-(R-1)^{\tilde{d}}) \\
           &+\sum_{j_1+j_2+j_3=k\atop j_2\neq k}\begin{pmatrix}k\\j_1,j_2,j_3\end{pmatrix}\frac{C^p (2\pi)^k}{2^k}(2R-1)^{j_1}(3R^2)^{j_2}(4R+1)^{j_3} 
           \frac{2\pi^{\frac{\tilde{d}}{2}}}{\tilde{d}\Gamma\left(\frac{\tilde{d}}{2}\right)}((2R)^{\tilde{d}}-R^{\tilde{d}}) \\
           &+\sum_{j_1+j_2+j_3=k}\begin{pmatrix}k\\j_1,j_2,j_3\end{pmatrix}\frac{C^p (2\pi)^k}{2^k}(2R-1)^{j_1}(3R^2)^{j_2}(4R+1)^{j_3} 
           \frac{2\pi^{\frac{\tilde{d}}{2}}}{\tilde{d}\Gamma\left(\frac{\tilde{d}}{2}\right)}((2R+1)^{\tilde{d}}-(2R)^{\tilde{d}}) \\
  \leqslant& \sum_{j_1+j_2+j_3=k}\begin{pmatrix}k\\j_1,j_2,j_3\end{pmatrix}CR^{j_1+2j_2+j_3+\tilde{d}-1} 
            + \sum_{j_1+j_2+j_3=k\atop j_2\neq k}\begin{pmatrix}k\\j_1,j_2,j_3\end{pmatrix}CR^{j_1+2j_2+j_3+\tilde{d}}\\
           &+\sum_{j_1+j_2+j_3=k}\begin{pmatrix}k\\j_1,j_2,j_3\end{pmatrix}CR^{j_1+2j_2+j_3+\tilde{d}-1} 
  \leqslant CR^{d-1}.
\end{align*}
This shows that $\|(\lambda I-\L_{\infty}^{\mathrm{sim}})v_R\|_{L^p}^p \leqslant CR^{d-1}$.
\item[4.] For the operator $\hat{\L}_{Q}=\hat{\L}_{\infty}+Q(T_1(T_2(\psi)))$,
equation \eqref{equ:EssentialSpectrumProperty1} and $\chi_R(r)=0$ for $r\in I_1\cup I_5$ imply
\begin{align}
  \label{equ:EssentialSpectrumProperty5}
  (\lambda I-\hat{\L}_{Q})\hat{v}_R(\psi)=0,\quad\text{if $|\tilde{y}|\in I_1\cup I_5$ or $r_l\in I_1\cup I_5$ for some $1\leqslant l\leqslant k$.}
\end{align}
Moreover, if $|\tilde{y}|\in I_3$ and $r_l\in I_3$ for every $1\leqslant l\leqslant k$, then we obtain from \eqref{equ:EssentialSpectrumProperty2}, $\chi_R'(r)\leqslant\left\|\chi_R\right\|_{C^2_{\mathrm{b}}}$, $|\hat{v}_R(\psi)|\leqslant 1$, and 
$\frac{1}{r_l}\leqslant\frac{1}{R}\leqslant 1$,
\begin{align*}
           & |(\lambda I-\hat{\L}_{Q})\hat{v}_R(\psi)| 
          =  \bigg|\left(\lambda I-\L_{\infty}^{\mathrm{sim}}\right)\hat{v}_R(\psi) - A\sum_{l=1}^{k}\left(\frac{1}{r_l}\partial_{r_l}+\frac{1}{r_l^2}\partial_{\phi_l}^2\right)\hat{v}_R(\psi)
             -Q(T_1(T_2(\psi)))\hat{v}_R(\psi) \bigg| \\
          =& \bigg|A\sum_{l=1}^{k}\bigg(\frac{i\omega_l}{r_l}+\frac{\chi_R'(r_l)}{r_l\chi_R(r_l)}-\frac{n_l^2}{r_l^2}\bigg)\hat{v}_R(\psi)+Q(T_1(T_2(\psi)))\hat{v}_R(\psi) \bigg| \\
  \leqslant& |A| \sum_{l=1}^{k}\bigg(\frac{|\omega_l|}{r_l}+\frac{\left\|\chi_R\right\|_{C^2_{\mathrm{b}}}}{r_l}+\frac{n_l^2}{r_l^2}\bigg) + \left|Q(T_1(T_2(\psi)))\right| 
  \leqslant \bigg(|A| \sum_{l=1}^{k}\left(|\omega_l|+\left\|\chi_R\right\|_{C^2_{\mathrm{b}}}+n_l^2\right)\frac{1}{r_l}+\eta_R\bigg)^{\frac{1}{p}}.
\end{align*}
Similarly, from \eqref{equ:step1}, $\left|\chi_R(r)\right|\leqslant 1$, $\chi_R'(r)\leqslant\left\|\chi_R\right\|_{C^2_{\mathrm{b}}}$,  
$\left|\hat{v}_R(\psi)\right|\leqslant 1$, $\big|\tfrac{\hat{v}_R(\psi)}{\chi_R(r_l)}\big|\leqslant 1$, $\frac{1}{r_l}\leqslant\frac{1}{R-1}\leqslant 1$, $\frac{1}{r_l^2}\leqslant 1$,
and $Q\in L^{\infty}$ we find in case $|\tilde{y}|\in I_2\cup I_3\cup I_4$ and $r_l\in I_2\cup I_3\cup I_4$ for every $1\leqslant l\leqslant k$:
\begin{align*}
           & |(\lambda I-\hat{\L}_{Q})\hat{v}_R(\psi)| 
          =  \bigg|\left(\lambda I-\L_{\infty}^{\mathrm{sim}}\right)\hat{v}_R(\psi) - A\sum_{l=1}^{k}\left(\frac{1}{r_l}\partial_{r_l}+\frac{1}{r_l^2}\partial_{\phi_l}^2\right)\hat{v}_R(\psi)
             -Q(T_1(T_2(\psi)))\hat{v}_R(\psi) \bigg| \\
          =& \bigg|\left(\lambda I-\L_{\infty}^{\mathrm{sim}}\right)\hat{v}_R(\psi)-A\sum_{l=1}^{k}\bigg(\frac{i\omega_l}{r_l}+\frac{\chi_R'(r_l)}{r_l\chi_R(r_l)}-\frac{n_l^2}{r_l^2}\bigg)\hat{v}_R(\psi)
             -Q(T_1(T_2(\psi)))\hat{v}_R(\psi) \bigg| \\
  \leqslant& \Big|\left(\lambda I-\L_{\infty}^{\mathrm{sim}}\right)\hat{v}_R(\psi)\Big| + |A|\sum_{l=1}^{k}\bigg(\frac{|\omega_l|}{r_l}+\frac{\left\|\chi_R\right\|_{C^2_{\mathrm{b}}}}{r_l}+\frac{n_l^2}{r_l^2}\bigg) + \Big|Q(T_1(T_2(\psi)))\Big| \\
  \leqslant& C + |A|\Big(|\omega|\sqrt{k}+k\left\|\chi_R\right\|_{C^2_{\mathrm{b}}}+|n|^2\Big) + \left\|Q\right\|_{L^{\infty}}
          =  C.
\end{align*}

\item[5.] Finally, let us consider $(\lambda I-\L_{Q})v_R$ in $L^p$. From the transformation theorem and \eqref{equ:EssentialSpectrumProperty5} we obtain
\begin{align*}
   & \|(\lambda I-\L_{Q})v_R\|_{L^p}^p 
  =  \int_{\R^d} |(\lambda I-\L_Q)v_R(x)|^p dx \\
  =& \int_{0}^{\infty}\int_{-\pi}^{\pi}\cdots\int_{0}^{\infty}\int_{-\pi}^{\pi}\int_{\R^{d-2k}}\langle \mathbf{r} \rangle |(\lambda I-\hat{\L}_{Q})\hat{v}_R(\psi)|^p d\psi \\
  =& \int_{R-1}^{2R+1}\int_{-\pi}^{\pi}\cdots\int_{R-1}^{2R+1}\int_{-\pi}^{\pi}\int_{R-1\leqslant|\tilde{y}|\leqslant 2R+1}\langle \mathbf{r} \rangle |(\lambda I-\hat{\L}_{Q})\hat{v}_R(\psi)|^p d\psi.
\end{align*}
Again we distinguish two cases for $\tilde{d}:=d-2k$:\\
\textbf{Case 1:} ($\tilde{d}=0$). From step 4, equation \eqref{equ:EssentialSpectrumProperty4}, and $d=2k$ we deduce
\begin{align*}
  \leqslant& \int_{R}^{2R}\int_{-\pi}^{\pi}\cdots\int_{R}^{2R}\int_{-\pi}^{\pi}\langle \mathbf{r} \rangle \left[|A|\sum_{l=1}^{k}(|\omega|+\left\|\chi_R\right\|_{C^2_{\mathrm{b}}}+n_l^2)
\frac{1}{r_l}+\eta_R\right]d\phi_k dr_k\cdots d\phi_1 dr_1 \\
           & +\sum_{j_1+j_2+j_3\atop j_2\neq k}\begin{pmatrix}k\\j_1,j_2,j_3\end{pmatrix}\left(\int_{R-1}^{R}\right)^{j_1}\left(\int_{R}^{2R}\right)^{j_2}\left(\int_{2R}^{2R+1}\right)^{j_3}C^p \langle \mathbf{r} \rangle (2\pi)^k dr_1\cdots dr_k \\
  \leqslant& (2\pi)^k\int_{R}^{2R}\cdots\int_{R}^{2R}\Bigg[|A|\Bigg(\sum_{l=1}^{k}\Bigg(\prod_{j=1\atop j\neq l}^{k}r_j\Bigg)\left(|\omega_l|+\left\|\chi_R\right\|_{C^2_{\mathrm{b}}}+n_l^2\right)\Bigg)
+\langle \mathbf{r} \rangle \eta_R\Bigg] d\mathbf{r} +CR^{d-1} \\
          =& (2\pi)^k\Bigg[|A|\sum_{l=1}^{k}\left(|\omega_l|+\left\|\chi_R\right\|_{C^2_{\mathrm{b}}}+n_l^2\right)\int_{R}^{2R}\cdots\int_{R}^{2R}\Bigg(\prod_{j=1\atop j\neq l}^{k}r_j\Bigg) d\mathbf{r} 
             +\eta_R \int_{R}^{2R}\cdots\int_{R}^{2R} \langle \mathbf{r} \rangle d\mathbf{r}\bigg] + CR^{d-1} \\
          =& (2\pi)^k|A|\sum_{l=1}^{k}\left(|\omega_l|+\left\|\chi_R\right\|_{C^2_{\mathrm{b}}}+n_l^2\right)\Bigg(\prod_{j=1\atop j\neq l}^{k}\int_{R}^{2R}r_j dr_j\Bigg)\int_{R}^{2R}dr_l 
             +(2\pi)^k \eta_R \prod_{j=1}^{k}\int_{R}^{2R}r_j dr_j +CR^{d-1} \\
          =& (2\pi)^k|A|\left(\sum_{l=1}^{k}\left(|\omega_l|+\left\|\chi_R\right\|_{C^2_{\mathrm{b}}}+n_l^2\right)\left(\frac{3}{2}\right)^{k-1}R^{2k-1}\right) 
             +(2\pi)^k \eta_R \left(\frac{3}{2}\right)^k R^{2k} +CR^{d-1} \\
  \leqslant& CR^{d-1}+CR^{d}\eta_R.
\end{align*}
For the first inequality we refer to case 1 of step 3.

\noindent
\textbf{Case 2:} ($\tilde{d}\geqslant 1$). From the procedure used in case 2 of step 5 and in case 1 and \eqref{equ:EssentialSpectrumProperty3} we obtain
\begin{align*}
  \leqslant& \int_{R}^{2R}\int_{-\pi}^{\pi}\cdots\int_{R}^{2R}\int_{-\pi}^{\pi}\int_{R\leqslant|\tilde{y}|\leqslant 2R}\langle \mathbf{r} \rangle \left[|A|\sum_{l=1}^{k}(|\omega_l|+\left\|\chi_R\right\|_{C^2_{\mathrm{b}}}+n_l^2)
\frac{1}{r_l}+\eta_R\right]
             d\psi +CR^{d-1} \\
 \leqslant& \left(CR^{2k-1}+CR^{2k}\eta_R\right)\int_{R\leqslant|\tilde{y}|\leqslant 2R}d\tilde{y} + CR^{d-1} 
 \leqslant  CR^{2k-1+\tilde{d}}+CR^{d-1}+CR^{2k+\tilde{d}}\eta_R \\
         =&  CR^{d-1} + CR^{d}\eta_R. 
\end{align*}
The constant $CR^{d-1}$ in the first inequality comes from an estimate of three sums, compare case 2 from step 3. For the second inequality compare case 1. 
This shows that $\|(\lambda I-\L_{Q})v_R\|_{L^p}^p \leqslant CR^{d-1} + CR^{d}\eta_R$.
\end{itemize}
\end{proof}

We now apply Theorem \ref{thm:EssentialSpectrumOfLQ} to \eqref{equ:EssSpecSettingLRW}
which provides us the essential $L^p$-spectrum of the linearization $\L$ from \eqref{equ:1.8}. 
We emphasize that the first condition in \eqref{equ:PropertyQ} requires $v_{\star}\in C_{b}(\R^d,\R^m)$ and $f\in C^2(\R^m,\R^m)$, whereas the second condition 
in \eqref{equ:PropertyQ} even requires $|v_{\star}(x)-v_{\infty}|\to 0$ as $|x|\to\infty$. The convergence follows from an application of \cite[Cor.~4.3]{BeynOtten2016a}, 
which provides us pointwise exponential decay estimates for $|v_{\star}(x)-v_{\infty}|$ as $|x|\to\infty$, but (depending on $d$ and $p$) \cite[Cor.~4.3]{BeynOtten2016a} 
requires more smoothness on $v_{\star}$ and $f$.


\begin{proof}[Proof (of Theorem \ref{thm:EssSpecLRW})]
  The assertion directly follows from an application of Theorem \ref{thm:EssentialSpectrumOfLQ} to the matrices $B_{\infty}$ and $Q(x)$ from \eqref{equ:EssSpecSettingLRW}. 
  Since \eqref{cond:A4DC}, \eqref{cond:A5} are satisfied and since \eqref{cond:A8} implies \eqref{cond:A8B}, it remains to check \eqref{equ:PropertyQ}. 
  From Taylor's theorem we obtain
  \begin{align}
    \label{equ:PointwiseEstimateQ}
    \begin{split}
    \left|Q(x)\right| \leqslant \int_{0}^{1}\left|D^2f(v_{\infty}+s(v_{\star}(x)-v_{\infty}))\right|ds\left|v_{\star}(x)-v_{\infty}\right|\quad\forall\,x\in\R^d.
    \end{split}
  \end{align}
  Since $f\in C^2(\R^m,\R^m)$ and $v_{\star}\in C_{\mathrm{b}}(\R^d,\R^m)$, estimate \eqref{equ:PointwiseEstimateQ} implies boundedness of $Q$ on $\R^d$, 
  i.e. $Q\in L^{\infty}(\R^d,\R^{m,m})$, which proves the first condition in \eqref{equ:PropertyQ}. An application of \cite[Cor.~4.3]{BeynOtten2016a} 
  to $|\alpha|=0$ shows a pointwise exponential estimate
  \begin{equation}
    \label{equ:PointwiseEstimateVstar}
    \left|v_{\star}(x)-v_{\infty}\right| \leqslant C_1\exp\big(-\mu\sqrt{|x|^2+1}\big)\quad\forall\,x\in\R^d\;\forall\,\mu\in[0,\mu_{\mathrm{max}})\quad\text{with}\quad\mu_{\mathrm{max}}=\frac{\sqrt{\azero\bzero}}{\amax p},
  \end{equation}
  where $\amax=\rho(A)$ denotes the spectral radius of $A$, $-\azero=s(-A)$ the spectral bound of $-A$, and $-\bzero=s(Df(v_{\infty}))$ the spectral bound of $Df(v_{\infty})$.
  Combining \eqref{equ:PointwiseEstimateQ} and \eqref{equ:PointwiseEstimateVstar} yields
  \begin{align}
    \label{equ:PointwiseEstimateQ2}
    \begin{split}
    \left|Q(x)\right| \leqslant C\exp\big(-\mu\sqrt{|x|^2+1}\big)\quad\forall\,x\in\R^d\;\forall\,\mu\in[0,\mu_{\mathrm{max}})
    \end{split}
  \end{align}
  with $C=C_1 \sup_{|y-v_{\infty}|\leqslant C_1}|D^2 f(y)|$. 
  Take a fixed $\mu\in(0,\mu_{\mathrm{max}})$ so that \eqref{equ:PointwiseEstimateQ2} implies
  \begin{equation*}
    \eta_R:=\underset{|x|\geqslant R}{\esssup}\left|Q(x)\right|\leqslant C\,\underset{|x|\geqslant R}{\esssup}\exp\left(-\mu\sqrt{|x|^2+1}\right)
    = C \exp\left(-\mu\sqrt{R^2+1}\right)\to 0\quad\text{as}\quad R\to\infty.
  \end{equation*}
  This proves the second condition in \eqref{equ:PropertyQ}.
\end{proof}

%
%
\sect{Application of Fredholm theory in $L^p(\R^d,\C^m)$}
\label{sec:4}

\subsection{Fredholm operator of index $0$}
\label{subsec:4.2}
In this section we show that the differential operator 
\begin{equation*}
  \lambda I-\L_Q:\left(\D^p_{\mathrm{loc}}(\L_0),\left\|\cdot\right\|_{\L_0}\right)\rightarrow \left(L^p(\R^d,\C^m),\left\|\cdot\right\|_{L^p}\right),\quad 1<p<\infty
\end{equation*} 
from \eqref{equ:1.12a}--\eqref{equ:4.4} is Fredholm of index $0$ provided that $\Re\lambda>-\bzero$. The matrix-valued function $Q\in C(\R^d,\C^{m,m})$ is 
assumed to be \begriff{asymptotically small}, i.e. for $|x|$ large it falls
below a certain computable threshold similar to \eqref{equ:BoundednessConditionForVStar}.
We further need the following Lemma (see \cite[Lem.~4.1]{BeynLorenz2008} for $p=2$)
which is a consequence of Sobolev imbedding and the compactness criterion
in $L^p$-spaces. 

\begin{lemma}[Compactness of multiplication operator]\label{lem:4.2}
  Let
  \begin{equation}
    \label{equ:4.5}
    M\in C(\R^d,\C^{m,m})\quad\text{with}\quad \lim_{R\to\infty}\sup_{|x|\geqslant R}\left|M(x)\right|\rightarrow 0.
  \end{equation}
  Then the operator of multiplication
  \begin{equation*}
    \widetilde{M}:(W^{1,p}(\R^d,\C^m),\left\|\cdot\right\|_{W^{1,p}})\rightarrow (L^p(\R^d,\C^m),\left\|\cdot\right\|_{L^p}),\quad u(\cdot)\longmapsto \widetilde{M}u(\cdot) := M(\cdot)u(\cdot),
  \end{equation*}
  is compact for any $1<p<\infty$.
\end{lemma}

We are now ready to prove that $\lambda I-\L_Q$ is Fredholm of index $0$.

\begin{theorem}\label{thm:4.3}
  Let the assumptions \eqref{cond:A4DC}, \eqref{cond:A5} and \eqref{cond:A8B} be satisfied for $\K=\C$ and for some $1<p<\infty$. 
  Moreover, let $\lambda\in\C$ with $\Re\lambda\geqslant-\bzero+\gamma$ for some $\gamma>0$, let $0<\varepsilon<1$, and let $Q\in C(\R^d,\C^{m,m})$ satisfy
  \begin{equation}
    \label{equ:PropertyQ3}
    \sup_{|x|\geqslant R_0}|Q(x)|\leqslant\frac{\varepsilon\gamma}{2}\min\left\{\frac{1}{\kappa\aone},\frac{1}{C_{0,\varepsilon}}\right\}\quad\text{for some $R_0>0$,}
  \end{equation}
  where $-\bzero=s(-B_{\infty})$ denotes the spectral bound of $-B_{\infty}$, $\kappa=\mathrm{cond}(Y)$ the condition number of $Y$ from \eqref{cond:A8B}, 
  $\aone$ the constant from \eqref{equ:aminamaxazerobzero}, and $C_{0,\varepsilon}=C_{0,\varepsilon}(d,p,\varepsilon,\kappa,\aone)>0$ the constant from 
  \cite[Thm.~2.10]{BeynOtten2016a}. Then, the operator
  \begin{align*}
    \lambda I-\L_Q:(\D^p_{\mathrm{loc}}(\L_0),\left\|\cdot\right\|_{\L_0})\rightarrow (L^p(\R^d,\C^m),\left\|\cdot\right\|_{L^p})
  \end{align*}
  is Fredholm of index $0$. In particular, the operator $\lambda I-\L_Q$ has finite-dimensional kernel and cokernel, i.e.
  $\mathrm{dim}\,\mathcal{N}(\lambda I-\L_Q)<\infty$ and $\mathrm{codim}\,\mathcal{R}(\lambda I-\L_Q)<\infty$.
\end{theorem}

\begin{proof}
The proof follows our outline in formulas \eqref{equ:DecompositionLQ_intro},
\eqref{equ:Ls}.

  \begin{itemize}[leftmargin=0.43cm]\setlength{\itemsep}{0.1cm}
  \item[1.] Let us write $\lambda=\lambda_1+\lambda_2$ with $\lambda_2:=-\bzero+\gamma$ and $\lambda_1:=\lambda-\lambda_2$. Further, take cut-off function
  $\chi_1\in C_{\mathrm{c}}^{\infty}(\R^d,[0,1])$ with $\chi_1(x)= 1$ $(|x| \le 1)$, $\chi_1(x)=0$ $(|x| \geqslant 2)$, and define $\chi_R\in C_{\mathrm{c}}^{\infty}(\R^d,[0,1])$ 
  via $\chi_R(x)=\chi_1\left(\frac{x}{R}\right)$ for $R>0$.
  With $R_0$ from \eqref{equ:PropertyQ3} we then write
  \begin{align*}
    Q(x) = Q_{\mathrm{s}}(x) + Q_{\mathrm{c}}(x),\quad Q_{\mathrm{s}}(x):=(1-\chi_{R_0}(x))Q(x),\quad  Q_{\mathrm{c}}(x):=\chi_{R_0}(x)Q(x)
      \end{align*}
      and define $\L_{\mathrm{s}}$ as in \eqref{equ:Ls}.
  Then  $Q_{\mathrm{c}}$ has compact support in $B_{2R_0}(0)$ and $Q_{\mathrm{s}}$ satisfies due to \eqref{equ:PropertyQ3} 
  \begin{equation}
    \label{equ:BoundQs}
    \left\|Q_{\mathrm{s}}\right\|_{L^{\infty}}
    \leqslant \left\|1-\chi_{R_0}\right\|_{\infty} \sup_{|x|\geqslant R_0}|Q(x)|
    \leqslant \frac{\varepsilon\gamma}{2} \min\left\{\frac{1}{\kappa\aone},\frac{1}{C_{0,\varepsilon}}\right\}.
  \end{equation}
  Setting $\widetilde{\L}_{\mathrm{s}}:=\L_{\mathrm{s}}-\lambda_2 I$ we can factorize $\lambda I-\L_Q$ as in \eqref{equ:DecompositionLQ_intro},
  \begin{align*}
    \lambda I-\L_Q = \lambda_1 I-(\L_{\mathrm{s}}-\lambda_2 I)-Q_{\mathrm{c}}(\cdot) = \left(I-Q_{\mathrm{c}}(\cdot)(\lambda_1 I-\widetilde{\L}_{\mathrm{s}})^{-1}\right)(\lambda_1 I-\widetilde{\L}_{\mathrm{s}}).
  \end{align*}

  \item[2.] We verify the following two conditions
  \begin{itemize}[leftmargin=1.2cm]\setlength{\itemsep}{0.0cm}
    \item[a)] $\lambda_1 I-\widetilde{\L}_{\mathrm{s}}:(\D^p_{\mathrm{loc}}(\L_0),\left\|\cdot\right\|_{\L_0})\rightarrow(L^p(\R^d,\C^m),\left\|\cdot\right\|_{L^p})$ is a linear homeomorphism,
    \item[b)] $\widetilde{Q}_c(\lambda_1 I-\widetilde{\L}_{\mathrm{s}})^{-1}:(L^p(\R^d,\C^m),\left\|\cdot\right\|_{L^p})\rightarrow(L^p(\R^d,\C^m),\left\|\cdot\right\|_{L^p})$ is a compact operator,
  \end{itemize}
  where $\widetilde{Q}_c$ denotes the operator of multiplication by $Q_{\mathrm{c}}$ which is defined by
  \begin{align}\label{equ:defmultQ}
    \widetilde{Q}_c:(W^{1,p}(\R^d,\C^m),\left\|\cdot\right\|_{W^{1,p}})\rightarrow(L^p(\R^d,\C^m),\left\|\cdot\right\|_{L^p}),\quad [\widetilde{Q}_cv](x)=Q_{\mathrm{c}}(x)v(x).
  \end{align}
  Then the compact perturbation of the identity from \cite[Satz~9.8]{Alt2006}, the product theorem of Fredholm operators from \cite[Thm.~13.1]{TaylorLay1980} imply that both operators
  \begin{equation*}
    I-Q_{\mathrm{c}}(\cdot)(\lambda_1 I-\widetilde{\L}_{\mathrm{s}})^{-1}:(L^p(\R^d,\C^m),\left\|\cdot\right\|_{L^p})\rightarrow (L^p(\R^d,\C^m),\left\|\cdot\right\|_{L^p})
  \end{equation*}
  and $\lambda I -\L_{Q}$ are Fredholm of index $0$. It remains to check whether the conditions a) and b) are satisfied.

  \begin{itemize}[leftmargin=0.43cm]\setlength{\itemsep}{0.1cm}
  \item[a)]  The boundedness of $\lambda_1 I-\widetilde{\L}_{\mathrm{s}}$ 
  with respect to the graph norm $\left\|v\right\|_{\L_0}:=\left\|\L_0 v\right\|_{L^p}+\left\|v\right\|_{L^p}$ follows from
  \begin{equation*}
  \begin{aligned}
    \|(\lambda_1 I-\widetilde{\L}_{\mathrm{s}})v\|_{L^p} \leqslant & \left\|\L_0 v\right\|_{L^p} + \left|\lambda I + B_{\infty}\right| \left\|v\right\|_{L^p}
      + \left\|Q_{\mathrm{s}}\right\|_{L^{\infty}}\left\|v\right\|_{L^p}
    \leqslant C \left\|v\right\|_{\L_0}\quad \forall\,v\in D^p_{\mathrm{loc}}(\L_0).
  \end{aligned}
  \end{equation*}
  The  unique solvability of $(\lambda_1 I-\widetilde{\L}_{\mathrm{s}})v=g$ and
   thus continuity of the inverse by the inverse mapping theorem, is a consequece of
   \cite[Thm.~3.2]{BeynOtten2016a}. Note that 
  \cite[Thm.~3.2]{BeynOtten2016a} is formulated for $(\lambda I-\L_{\mathrm{s}})v=g$ and must be applied to $(\lambda_1 I-\widetilde{\L}_{\mathrm{s}})v=g$, 
  using the shifted data
  \begin{align*}
    (\widetilde{\L}_{\mathrm{s}}=\L_{\mathrm{s}}-\lambda_2 I,\lambda_1=\lambda-\lambda_2,\widetilde{B}_{\infty}=B_{\infty}+\lambda_2 I,\tilde{b}_0=\bzero+\lambda_2=\gamma)\quad\text{instead of}\quad 
    (\L_{\mathrm{s}},\lambda,B_{\infty},\bzero).
  \end{align*}
  The assumptions of \cite[Thm.~3.2]{BeynOtten2016a} are
  $\tilde{b}_0:=-s(-B_{\infty}-\lambda_2 I)>0$ and $\Re\lambda_1\geqslant -(1-\varepsilon)\tilde{b}_0$, which in our case follow from
  \begin{align*}
    &\tilde{b}_0:=-s(-B_{\infty}-\lambda_2 I) = -s(-B_{\infty})-s(-\lambda_2 I) = \bzero + \lambda_2 = \gamma >0,\\
    &\Re\lambda_1 = \Re\lambda - \lambda_2 \geqslant -\bzero + \gamma - \lambda_2 = 0 > -(1-\varepsilon)\gamma = -(1-\varepsilon)\tilde{b}_0.
  \end{align*}

  \item[b)] The operator $\widetilde{Q}_c$ is a linear bounded and compact operator. While linearity and boundedness are clear, compactness 
   follows from Lemma \ref{lem:4.2} with $M=Q_{\mathrm{c}}$. The condition \eqref{equ:4.5} is obviously satisfied 
  since $Q_{\mathrm{c}}$ is continuous and has compact support. As shown in a), $(\lambda_1 I-\widetilde{\L}_{\mathrm{s}})^{-1}:L^p(\R^d,\C^m)\rightarrow\D^p_{\mathrm{loc}}(\L_0)$ 
  is a linear bounded operator with dense range $\D^p_{\mathrm{loc}}(\L_0)\subseteq L^p(\R^d,\C^m)$, hence Fredholm of index $0$. Moreover, a key result from \cite[Sec.~5]{Otten2015a}
  guarantees a continuous imbedding  $\D^p_{\mathrm{loc}}(\L_0)\subseteq W^{1,p}(\R^d,\C^m)$. Hence Lemma \ref{lem:4.2} shows that 
  $\widetilde{Q}_c(\lambda_1 I-\widetilde{\L}_{\mathrm{s}})^{-1}:L^p(\R^d,\C^m)\rightarrow L^p(\R^d,\C^m)$ is compact.
  \end{itemize}
  \end{itemize}
\end{proof}

\begin{remark}
  \begin{itemize}[leftmargin=0.43cm]\setlength{\itemsep}{0.1cm}
  \item[a)] (Decomposition). 
 We note that the factorization \eqref{equ:DecompositionLQ_intro} is a new feature of our approach. Standard factorizations often split off 
a constant coefficient operator rather than $\L_{\mathrm{s}}$.
 For example, in  \cite{BeynLorenz2008} the following factorization is used
  \begin{equation}
    \label{equ:DecompositionLQOLD}
    \lambda I-\L_Q = \left(I-Q(\cdot)(\lambda_1 I-\widetilde{\L}_{\infty})^{-1}\right)(\lambda_1 I-\widetilde{\L}_{\infty})
  \end{equation}
  where $\widetilde{\L}_{\infty}:=\L_{\infty}-\lambda_2 I$ and $\L_{\infty}$ denotes the constant coefficient operator defined by
  \begin{equation*}
    [\L_{\infty} v](x) = A\triangle v(x) + \left\langle Sx,\nabla v(x)\right\rangle - B_{\infty}v(x),\,x\in\R^d.
  \end{equation*}
  For this factorization one only needs a resolvent estimate for
  the constant coefficient operator $\lambda_1 I-\widetilde{\L}_{\infty}$. But
  this must be compensated by the
  stronger assumption $\left|Q(x)\right|\to 0$ as $|x|\to\infty$ (instead of \eqref{equ:PropertyQ3}) in order to apply the compactness 
  from Lemma \ref{lem:4.2} with $M=Q$.
  \item[b)] (Positivity). 
  If we additionally require the positivity assumption \eqref{cond:A9B} in Theorem \ref{thm:4.3}, then one can omit the decomposition of $\lambda$. 
  In this case $\lambda I-\L_Q$ can easily be written as 
  \label{equ:DecompositionLQ}
  \begin{equation*}
    \lambda I-\L_Q = \left(I-Q_{\mathrm{c}}(\cdot)(\lambda I-\L_{\mathrm{s}})^{-1}\right)(\lambda I-\L_{\mathrm{s}}).
  \end{equation*}
  Alternatively, we may factorize $\lambda I-\L_Q$ into
  \begin{equation*}
    \lambda I-\L_Q = \left(I-Q(\cdot)(\lambda I-\L_{\infty})^{-1}\right)(\lambda I-\L_{\infty})
  \end{equation*}
  provided $Q$ satisfies $\left|Q(x)\right|\to 0$ as $|x|\to\infty$.
  \end{itemize}
\end{remark}
For reasons of completeness we formulate a result, analogous
to Theorem \ref{thm:4.3}, based on the factorization \eqref{equ:DecompositionLQOLD}.

\begin{corollary}\label{cor:4.4}
  Let the assumptions \eqref{cond:A4DC}, \eqref{cond:A5}, \eqref{cond:A8B} and 
  \begin{equation}
    \label{equ:PropertyQ2}
    Q\in C(\R^d,\K^{m,m})\quad\text{with}\quad \lim_{R\to\infty}\sup_{|x|\geqslant R}\left|Q(x)\right|=0
  \end{equation}
  be satisfied for $\K=\C$ and for some $1<p<\infty$. 
  Moreover, let $\lambda\in\C$ with $\Re\lambda\geqslant-\bzero+\gamma$ for some $\gamma>0$, where $-\bzero=s(-B_{\infty})$ denotes the spectral bound of $-B_{\infty}$.
  Then, the operator
  \begin{align*}
    \lambda I-\L_Q:(\D^p_{\mathrm{loc}}(\L_0),\left\|\cdot\right\|_{\L_0})\rightarrow (L^p(\R^d,\C^m),\left\|\cdot\right\|_{L^p})
  \end{align*}
  is Fredholm of index $0$.
\end{corollary}

\subsection{The adjoint operator}
\label{subsec:4.3}
In this section we analyze and identify the abstract adoint operator of $\L_Q$. Let us first review some results from 
\cite{Otten2014,Otten2014a,Otten2015a,Otten2015b} for the complex-valued Ornstein-Uhlenbeck operator $\L_0$ in $L^p(\R^d,\C^m)$ and its 
constant coefficient perturbation $\L_{\infty}$.

Assuming \eqref{cond:A2}, \eqref{cond:A5}, \eqref{cond:A8B} for $\K=\C$ it is shown in \cite[Thm.~4.4]{Otten2014}, \cite[Thm.~3.1]{Otten2014a} 
that the function $H_{\infty}:\R^d\times\R^d\times(0,\infty)\rightarrow\C^{m,m}$ defined by
\begin{align}
  \label{equ:HeatKernel}
  H_{\infty}(x,\xi,t)=(4\pi t A)^{-\frac{d}{2}}\exp\left(-B_{\infty}t-(4tA)^{-1}\left|e^{tS}x-\xi\right|^2\right),
\end{align}
is a heat kernel of the perturbed Ornstein-Uhlenbeck operator $\L_{\infty}$, cf. \eqref{equ:Linfty}. Under the same assumptions it is 
proved in \cite[Thm.~5.3]{Otten2014a} that the family of mappings $T_{\infty}(t):L^p(\R^d,\C^m)\rightarrow L^p(\R^d,\C^m)$, $t\geqslant 0$,
\begin{align}
  \left[T_{\infty}(t)v\right](x):= \begin{cases}
                              \int_{\R^d}H_{\infty}(x,\xi,t)v(\xi)d\xi &\text{, }t>0 \\
                              v(x) &\text{, }t=0
                            \end{cases}\quad ,\,x\in\R^d,
  \label{equ:OrnsteinUhlenbeckSemigroupLp}
\end{align}
defines a strongly continuous semigroup for each $1\leqslant p<\infty$. The semigroup $\left(T_{\infty}(t)\right)_{t\geqslant 0}$ is called an 
\begriff{Ornstein-Uhlenbeck semigroup}. The infinitesimal generator $\A_p:L^p(\R^d,\C^m)\supseteq\D(\A_p)\rightarrow L^p(\R^d,\C^m)$ of the 
semigroup $\left(T_{\infty}(t)\right)_{t\geqslant 0}$ has domain of definition
\begin{align*}
  \D(\A_p):=\left\{v\in L^p(\R^d,\C^m)\mid \A_p v :=\lim_{t\downarrow 0}t^{-1}(T_{\infty}(t)v-v)\text{ exists in $L^p(\R^d,\C^m)$}\right\},
\end{align*}
and satisfies resolvent estimates, see \cite[Cor.~6.7]{Otten2014}, \cite[Cor.~5.5]{Otten2014a}. The identification problem requires to
represent the maximal domain $\D(\A_p)$ in terms of Sobolev spaces,  and to show that the generator $\A_p$ and the differential operator 
$\L_{\infty}$ conincide on this domain. This problem is solved in \cite{Otten2015a}. Assuming \eqref{cond:A4DC}, \eqref{cond:A5}, \eqref{cond:A8B} 
for $\K=\C$ and for some $1<p<\infty$, it is shown in \cite[Thm.~5.1]{Otten2015a} that
\begin{align}
  \D(\A_p) = \D^p_{\mathrm{loc}}(\L_0)\quad\text{and}\quad \A_p v=\L_{\infty}v\text{ for all }v\in\D(\A_p).
\end{align}
Moreover, in \cite[Thm.~5.7]{Otten2014a} a-priori estimates are used to show $\D^p_{\mathrm{loc}}(\L_0)\subseteq W^{1,p}(\R^d,\C^m)$.

Next consider the variable coefficient operator $\L_Q$ and assume \eqref{cond:A2}, \eqref{cond:A5}, \eqref{cond:A8B}.
For $Q\in L^{\infty}(\R^d,\C^{m,m})$ let $\widetilde{Q}$ denote the multiplication operator in $L^p(\R^d,\C^m)$ as in \eqref{equ:defmultQ} 
and apply the bounded perturbation theorem \cite[III.1.3]{EngelNagel2000} to conclude that $\B_p := \A_p + \widetilde{Q}$ with $\D(\B_p):=\D(\A_p)$
generates a strongly continuous semigroup $(T_Q(t))_{t\geqslant 0}$ in $L^p(\R^d,\C^m)$. If we restrict $1<p<\infty$ and assume the stronger assumption 
\eqref{cond:A4DC} (or equivalently \eqref{cond:A4}) instead of \eqref{cond:A2}, an application of \cite[Thm.~5.1]{Otten2015a} solves the identification 
problem for $\B_p$, namely $\D(\B_p) = \D^p_{\mathrm{loc}}(\L_0)$ and $\B_p v = \L_{\infty}v + Qv = \L_Q v$ for all $v\in\D(\B_p)$. In particular, we 
obtain from \cite[Thm.~5.7]{Otten2014a} that $\D^p_{\mathrm{loc}}(\L_0)\subseteq W^{1,p}(\R^d,\C^m)$.

In the following we continue the process of identification for the
adjoint differential operator and relate it to the abstract definition.  
\begin{definition}[Adjoint operator]\label{equ:DefAdj}
  Let $X,Y$ be Banach spaces over $\C$ with dual spaces
  $X^{*},Y^{*}$ and duality pairings
  $\langle\cdot,\cdot\rangle_Y:Y^{*}\times Y\rightarrow\C$ and 
  $\langle\cdot,\cdot\rangle_X:X^{*}\times X\rightarrow\C$. For a densely
  defined operator $\A:X\supseteq\D(\A)\rightarrow Y$ the
   \begriff{abstract adjoint operator} $\A^{*}:Y^{*}\supseteq\D(\A^{*})\rightarrow X^{*}$ is defined by
    \begin{equation*}
      \D(\A^{*}) = \left\{y^{*}\in Y^*\mid\;\exists\,x^{*}\in X^*:\;\langle y^{*},\A x\rangle_Y = \langle x^*,x\rangle_X\;\forall\,x\in\D(\A)\right\},\quad \A^{*} y^*:=x^*.
    \end{equation*}
\end{definition}

Let us assume \eqref{cond:A4DC}, \eqref{cond:A5}, \eqref{cond:A8B} for $\K=\C$ and some $1<p<\infty$, and let $Q\in C(\R^d,\C^{m,m})$. Then we apply  
Definition \ref{equ:DefAdj} to the infinitesimal generator $\A=\B_p$ using the setting
\begin{equation} \label{equ:setspaces}
\begin{aligned}
  X=Y=L^p(\R^d,\C^m),\quad X^{*}=Y^{*}=L^q(\R^d,\C^m),\quad 1<p,q<\infty,\quad \frac{1}{p}+\frac{1}{q}=1,\\ 
  \langle w,v\rangle_{q,p}=\langle w,v\rangle_X=\langle w,v\rangle_Y=\int_{\R^d}w(x)^{\herm} v(x) dx,\quad w\in L^q,\quad v\in L^p.
\end{aligned}
\end{equation}
The abstract adjoint operator $\A^{*}=\B_p^{*}$ has maximal domain
\begin{equation}
  \label{equ:4.13}
  \D(\B_p^*) = \left\{v\in L^q(\R^d,\C^m)\mid\; \exists\,w\in L^q(\R^d,\C^m):\;\langle v,\L_Q u\rangle_{q,p}=\langle w,u\rangle_{q,p}\;\forall\,u\in\D^p_{\mathrm{loc}}(\L_0)\right\},
\end{equation}
and is defined through
\begin{equation}
  \label{equ:4.12}
  \B_p^{*}:\big(\D(\B_p^{*}),\left\|\cdot\right\|_{\B_p^*}\big)\rightarrow \left(L^q(\R^d,\C^m),\left\|\cdot\right\|_{L^q}\right),\quad \B_p^*v:=w,\quad \text{$w$ from \eqref{equ:4.13}}.
\end{equation}

Note that the element $w\in L^q(\R^d,\C^m)$ from \eqref{equ:4.13} is uniquely determined.
We compare this with the \begriff{formal adjoint (differential) operator}
 $ \L_Q^{*}:\big(\D^q_{\mathrm{loc}}(\L_0^*),\left\|\cdot\right\|_{\L_0^*}\big)\rightarrow \left(L^q(\R^d,\C^m),\left\|\cdot\right\|_{L^q}\right)$,
defined by
\begin{equation}
  \label{equ:4.10}
  [\L_Q^{*} v](x) = A^{\herm}\triangle v(x) - \left\langle Sx,\nabla v(x)\right\rangle - B_{\infty}^{\herm} v(x) + Q(x)^{\herm} v(x),\,x\in\R^d
\end{equation}
on its domain
\begin{equation*}
  \D^q_{\mathrm{loc}}(\L_0^*) = \left\{v\in W^{2,q}_{\mathrm{loc}}(\R^d,\C^m)\cap L^q(\R^d,\C^m):\; \L_0^* v = A^{\herm}\triangle v - \left\langle S\cdot,\nabla v\right\rangle\in L^q(\R^d,\C^m)\right\}, 1<q<\infty.
\end{equation*}
Definition \eqref{equ:4.10} is motivated by the following relation obtained
via integration by parts 
\begin{equation}
  \label{equ:4.11}
  \langle v,\L_Q u\rangle_{q,p} = \langle \L_Q^* v,u\rangle_{q,p}\quad\forall\,u\in\D^p_{\mathrm{loc}}(\L_0)\;\forall\,v\in\D^q_{\mathrm{loc}}(\L_0^*).
\end{equation}
The following result solves the identification problem for the adjoint operator. The proof is based on an application of \cite[Thm.~3.1]{BeynOtten2016a} 
to $(A^{\herm},-S,B_{\infty}^{\herm},Q(x)^{\herm},q=\frac{p}{p-1})$ instead of $(A,S,B_{\infty},Q(x),p)$. This  
requires the matrix $A^{\herm}$ to additonally satisfy the  $L^q$-dissipativity condition \eqref{cond:A4DCq} for the conjugate index $q:=\frac{p}{p-1}$.

\begin{lemma}[Identification of adjoint operator]\label{lem:4.4}
  Let the assumptions \eqref{cond:A4DC}, \eqref{cond:A4DCq}, \eqref{cond:A5}, \eqref{cond:A8B} and $Q\in L^{\infty}(\R^d,\K^{m,m})$ be satisfied for 
  $\K=\C$, for some $1<p<\infty$ and $q=\frac{p}{p-1}$. Then the formal adjoint operator $\L_Q^*$ and the abstract adjoint operator $\B_p^*$ 
  coincide, i.e.
  \begin{equation*}
    \D(\B_p^*) = \D^q_{\mathrm{loc}}(\L_0^*)\quad\text{and}\quad \B_p^*=\L_Q^*.
  \end{equation*}
  In particular, the corresponding graph norms are equivalent.
\end{lemma}

\begin{proof} 
  For the proof we abbreviate $\langle\cdot,\cdot\rangle=\langle\cdot,\cdot\rangle_{q,p}$.
  \begin{itemize}[leftmargin=0.43cm]\setlength{\itemsep}{0.1cm}
  \item $\D^q_{\mathrm{loc}}(\L_0^*)\subseteq\D(\B_p^*)$: Let $v\in\D^q_{\mathrm{loc}}(\L_0^*)$ and choose $w=\L_Q^* v\in L^q(\R^d,\C^m)$, then \eqref{equ:4.11} implies
  \begin{equation*}
    \langle v,\L_Q u\rangle = \langle \L_Q^* v,u\rangle = \langle w,u\rangle\quad\forall\,u\in\D^p_{\mathrm{loc}}(\L_0),
  \end{equation*}
  which yields $v\in\D(\B_p^*)$.
  \item $\D^q_{\mathrm{loc}}(\L_0^*)\supseteq\D(\B_p^*)$: Let $v\in\D(\B_p^*)$ and let $w\in L^q(\R^d,\C^m)$ be defined according to \eqref{equ:4.13}. 
  By an application of \cite[Thm.~3.1]{BeynOtten2016a} we have a unique solution $\tilde{v}\in\D^q_{\mathrm{loc}}(\L_0^*)$ of 
  $(\bar{\lambda} I-\L_Q^*)\tilde{v}=\bar{\lambda} v-w$ in $L^q(\R^d,\C^m)$ for any $\lambda\in\C$ with $\Re\lambda>-\bzero+\kappa\aone\left\|Q\right\|_{L^{\infty}}$. 
  Therefore, from \eqref{equ:4.11} and \eqref{equ:4.13} we obtain
  \begin{equation*}
  \begin{aligned}
    \langle v,(\lambda I-\L_Q)u\rangle
    =& \langle v,\lambda u\rangle - \langle v,\L_Q u\rangle
    = \lambda \langle v,u\rangle - \langle w,u\rangle
    = \langle \bar{\lambda} v-w,u\rangle
    = \langle (\bar{\lambda} I-\L_Q^*)\tilde{v},u\rangle \\
    =& \lambda \langle \tilde{v},u\rangle - \langle \L_Q^*\tilde{v},u\rangle
    = \lambda \langle \tilde{v},u\rangle - \langle\tilde{v},\L_Q u\rangle 
    = \langle \tilde{v},(\lambda I-\L_Q)u\rangle\quad\forall\,u\in\D^p_{\mathrm{loc}}(\L_0).
  \end{aligned}
  \end{equation*}
  Since $\lambda I-\L_Q$ is onto, this implies $v=\tilde{v}\in\D^q_{\mathrm{loc}}(\L_0^*)\subset L^q(\R^d,\C^m)$. 
  \end{itemize}
\end{proof}

\begin{remark}
  \begin{itemize}[leftmargin=0.43cm]\setlength{\itemsep}{0.1cm}
  \item[a)] (Shifted adjoint operator). 
  Consider the shifted operator $\lambda I-\B_p$ and its abstract adjoint operator $(\lambda I-\B_p)^*$ for some $\lambda\in\C$, then Lemma 
  \ref{lem:4.4} directly implies
  \begin{equation}
    \label{equ:4.11a}
    \D((\lambda I-\B_p)^*) = \D^q_{\mathrm{loc}}(\L_0^*)\quad\text{and}\quad (\lambda I-\B_p)^*=(\lambda I-\L_Q)^*=\bar{\lambda}I-\L_Q^*,
  \end{equation}
  i.e. the statement from Lemma \ref{lem:4.4} remains valid when shifting the operator by a complex value $\lambda$. In particular, the domain does 
  not depend on $\lambda$. To prove the generalized version \eqref{equ:4.11a} note that \eqref{cond:A8B} implies $\widetilde{B}_{\infty}:=B_{\infty}+\lambda I$ 
  to be simultaneously diagonalizable which allows us to apply \cite[Thm.~3.1]{BeynOtten2016a} to $\lambda I-\L_Q$. This observation will  
  be crucial in Section \ref{subsec:4.4} below.
  \item[b)] (Second characterization of the domain). 
  We emphasize that according to \cite[Thm.~6.1]{Otten2015a}, the maximal domain $\D^q_{\mathrm{loc}}(\L_0^*)$ of the adjoint operator coincides with
  \begin{equation*}
    \D^q_{\mathrm{max}}(\L_0^*) = \left\{v\in W^{2,q}(\R^d,\C^m):\;\left\langle S\cdot,\nabla v\right\rangle\in L^q(\R^d,\C^m)\right\},\; 1<q<\infty.
  \end{equation*}
   \end{itemize}
\end{remark}

\subsection{Fredholm alternative}
\label{subsec:4.4}
With the abstract operators identified, let us apply the Fredholm alternative to $\lambda I-\L_Q$ and its adjoint. First, consider the abstract
setting of Definition \ref{equ:DefAdj} and let $\A$ be a Fredholm operator of index $0$ with adjoint $\A^*:Y^*\supseteq\D(\A^*)\rightarrow X^*$.
Then the \begriff{Fredholm alternative} states:
\begin{itemize}[leftmargin=0.43cm]\setlength{\itemsep}{0.1cm}
  \item \textbf{either} the homogeneous equations
        \begin{equation*}
          \A x=0\quad\text{and}\quad \A^*x^*=0
        \end{equation*}
        have only the trivial solutions $x=0\in\D(\A)$ and $x^*=0\in\D(\A^*)$,
        in which case the inhomogeneous equations
        \begin{equation*}
          \A x=y\quad\text{and}\quad \A^*x^*=y^*
        \end{equation*}
        have unique solutions $x\in\D(\A)$ and $x^*\in\D(\A^*)$ for any $y\in Y$ and $y^*\in X^*$,
  \item \textbf{or} the homogeneous equations
        \begin{equation*}
          \A x=0\quad\text{and}\quad \A^*x^*=0
        \end{equation*}
        have exactly $1\leqslant n:=\mathrm{dim}\,\mathcal{N}(\A)<\infty$ (nontrivial) linearly independent solutions $x_1,\ldots,x_n\in\D(\A)$ and 
        $x^*_1,\ldots,x^*_n\in\D(\A^*)$, in which case the inhomogeneous equation
        \begin{equation*}
          \A x=y,\quad y\in Y
        \end{equation*}
        admits at least one (not necessarily unique) solution $x\in\D(\A)$ if and only if $y\in(\mathcal{N}(\A^*))^{\bot}$.
\end{itemize}
We apply the Fredholm alternative to the operator $\A=\lambda I-\B_p$ and its adjoint $\A^*=(\lambda I-\B_p)^*$ for $\lambda\in\C$ with $\Re\lambda>-\bzero$  
using the setting \eqref{equ:setspaces} and domains $\D(\lambda I-\B_p)=\D^p_{\mathrm{loc}}(\L_0)$, $\D((\lambda I-\B_p)^*)=\D^q_{\mathrm{loc}}(\L_0^*)$.
\begin{lemma}\label{lem:4.5}
  Let the assumptions \eqref{cond:A4DC}, \eqref{cond:A4DCq}, \eqref{cond:A5}, \eqref{cond:A8B} be satisfied for $\K=\C$, for some $1<p<\infty$ and for $q=\frac{p}{p-1}$. 
  Moreover, let $\lambda\in\C$ with $\Re\lambda\geqslant-\bzero+\gamma$ for some $\gamma>0$, let $0<\varepsilon<1$, and let $Q\in C(\R^d,\C^{m,m})$ satisfy \eqref{equ:PropertyQ3}, 
  where $-\bzero=s(-B_{\infty})$ denotes the spectral bound of $-B_{\infty}$.
  Then
  \begin{itemize}[leftmargin=0.43cm]\setlength{\itemsep}{0.1cm}
  \item \textbf{either} the homogeneous equations
        \begin{equation*}
          (\lambda I-\L_Q)v=0\quad\text{and}\quad (\lambda I-\L_Q)^*\psi=0
        \end{equation*}
        have only the trivial solutions $v=0\in\D^p_{\mathrm{loc}}(\L_0)$ and $\psi=0\in\D^q_{\mathrm{loc}}(\L_0^*)$, in which case 
        the inhomogeneous equations
        \begin{equation*}
          (\lambda I-\L_Q)v=h\quad\text{and}\quad (\lambda I-\L_Q)^*\psi=\phi
        \end{equation*}
        have unique solutions $v\in\D^p_{\mathrm{loc}}(\L_0)$ and $\psi\in\D^q_{\mathrm{loc}}(\L_0^*)$ for any $h\in L^p(\R^d,\C^m)$ and $\phi\in L^q(\R^d,\C^m)$.
  \item \textbf{or} the homogeneous equations
        \begin{equation*}
          (\lambda I-\L_Q)v=0\quad\text{and}\quad (\lambda I-\L_Q)^*\psi=0
        \end{equation*}
        have exactly $1\leqslant n:=\mathrm{dim}\,\mathcal{N}(\lambda I-\L_Q)<\infty$ (nontrivial) linearly independent solutions $v_1,\ldots,v_n\in\D^p_{\mathrm{loc}}(\L_0)$ and 
        $\psi_1,\ldots,\psi_n\in\D^q_{\mathrm{loc}}(\L_0^*)$, in which case
        the inhomogeneous equation
        \begin{equation*}
          (\lambda I-\L_Q)v=h,\quad h\in L^p(\R^d,\C^m)
        \end{equation*}
        admits at least one (not necessarily unique) solution $v\in\D^p_{\mathrm{loc}}(\L_0)$ if and only if $h\in(\mathcal{N}((\lambda I-\L_Q)^*))^{\bot}$.
\end{itemize}
\end{lemma}

\begin{proof}
  The assertion follows from Fredholm's alternative applied to $(\A,\D(\A))=(\lambda I-\B_p,\D(\lambda I-\B_p))$ and its adjoint $(\A^*,\D(\A^*))=((\lambda I-\B_p)^*,\D((\lambda I-\B_p)^*))$. For this purpose 
  recall $\D(\lambda I-\B_p)=\D^p_{\mathrm{loc}}(\L_0)$ and $\lambda I-\B_p=\lambda I-\L_Q$ from above as well as
  $\D((\lambda I-\B_p)^*)=\D^q_{\mathrm{loc}}(\L_0^*)$ and $(\lambda I-\B_p)^*=(\lambda I-\L_Q)^*$ from Lemma \ref{lem:4.4}. 
  Finally, Theorem \ref{thm:4.3} shows that $\lambda I-\B_p=\lambda I-\L_Q$ is Fredholm of index $0$.
\end{proof}

\subsection{Exponential decay}
\label{subsec:4.5}
Next we prove that any solution $v$ of $(\lambda I-\L_Q)v=0$ decays exponentially in space. The proof is based on an application of \cite[Thm.~3.5]{BeynOtten2016a}. The result is formulated in terms of radial weight functions  
  \begin{align}
    \label{equ:WeightFunctionsLQ}
    \theta(x,\mu) = \exp\left(\mu\sqrt{|x|^2+1}\right),\;x\in\R^d,\;\mu\in\R
  \end{align}
and the associated \begriff{exponentially weighted Lebesgue} and 
\begriff{Sobolev spaces} for $1 \le p < \infty$ and $k\in \N_0$
\begin{align*}
  L_{\theta}^{p}(\R^d,\K^m)   :=& \{u\in L^1_{\mathrm{loc}}(\R^d,\K^m) : 
\|u\|_{L^p_{\theta}}=\left\| \theta u \right\|_{L^p}<\infty\},
 \\
  W_{\theta}^{k,p}(\R^d,\K^m) :=& \{u\in L^p_{\theta}(\R^d,\K^m):
\|u\|^p_{W_{\theta}^{k,p}}= \sum_{|\beta|\le k} \|D^{\beta}u\|_{L^p_{\theta}}^p < \infty \}.
\end{align*}
The following theorem also uses the constants $\azero,\aone,\amax$ from \eqref{equ:aminamaxazerobzero}, $\gamma_A$ from \eqref{cond:A4DC}, $\delta_A$ from \eqref{cond:A4DCq}, $\bzero,\kappa$ as in Theorem \ref{thm:4.3}, $\beta_{\infty}\in\R$ such that $\Re\langle w,B_{\infty}w\rangle\geqslant\beta_{\infty}|w|^2$ for all $w\in\C^m$, and 
the constant $C_{0,\varepsilon}$ from \cite[Thm.~2.10 with $C_{\theta}=1$]{BeynOtten2016a} for $0<\varepsilon<1$.

\begin{theorem}[A-priori estimates in weighted $L^p$-spaces]\label{thm:APrioriEstimatesInLpRelativelyCompactPerturbation}
  Let the assumptions \eqref{cond:A4DC}, \eqref{cond:A4DCq}, \eqref{cond:A5}, \eqref{cond:A8B} be satisfied for $\K=\C$, for some $1<p<\infty$ and for $q=\frac{p}{p-1}$. 
  Moreover, let $\lambda\in\C$ with $\Re\lambda\geqslant-\bzero+\gamma$ for some $\gamma>0$, let $0<\varepsilon<1$ and let $Q\in C(\R^d,\C^{m,m})$ satisfy
  \begin{align}
    \label{equ:ConditionOnQ}
    \underset{|x|\geqslant R_0}{\esssup}\left|Q(x)\right| \leqslant
    \frac{\varepsilon\gamma}{2} \min\left\{\frac{1}{\kappa\aone},\frac{1}{C_{0,\varepsilon}},\frac{\beta_{\infty}-\bzero}{\gamma}+1\right\}\,\text{for some $R_0>0$}.
  \end{align}
 Consider weight functions $\theta_j(x)=\theta(x,\mu_j),j=1,2,3,4$ 
with exponents $\mu_j$ satisfying
  \begin{align}
    \label{equ:ExponentialRatesLQ}
-\sqrt{\varepsilon\frac{\gamma_A(\beta_{\infty}-\bzero+\gamma)}{2d|A|^2}}\leqslant\mu_1\leqslant 
0\leqslant\mu_2\leqslant\varepsilon\frac{\sqrt{\azero\gamma}}{\amax p},
  \end{align}
  and 
  \begin{align}
    \label{equ:ExponentialRatesLQstart}
    -\sqrt{\varepsilon\frac{\delta_A(\beta_{\infty}-\bzero+\gamma)}{2d|A|^2}}\leqslant\mu_3\leqslant 0\leqslant\mu_4\leqslant\varepsilon\frac{\sqrt{\azero\gamma}}{\amax q}.
  \end{align}
  Then every solution $v\in W^{2,p}_{\mathrm{loc}}(\R^d,\C^m)\cap L^p_{\theta_1}(\R^d,\C^m)$ resp. $\psi\in W^{2,q}_{\mathrm{loc}}(\R^d,\C^m)\cap L^q_{\theta_3}(\R^d,\C^m)$ of
  \begin{equation*}
    (\lambda I-\L_Q)v=g\quad\text{in $L^p_{\mathrm{loc}}(\R^d,\C^m)$}\quad\text{resp.}\quad (\lambda I-\L_Q)^{*}\psi=\phi\quad\text{in $L^q_{\mathrm{loc}}(\R^d,\C^m)$}
  \end{equation*} 
 with $g\in L^p_{\theta_2}(\R^d,\C^m)$ resp. $\phi\in L^q_{\theta_4}(\R^d,\C^m)$ 
satisfies $v\in W^{1,p}_{\theta_2}(\R^d,\C^m)$ resp. $\psi\in W^{1,q}_{\theta_4}(\R^d,\C^m)$. Moreover, the following estimates hold:
  \begin{align}
    \left\|v\right\|_{W^{k,p}_{\theta_2}} \leqslant &
 C_1 \left(\Re\lambda+\bzero\right)^{-\frac{k}{2}}
 \left(\left\|v\right\|_{L^p_{\theta_1}}
   +\left\|g\right\|_{L^p_{\theta_2}}\right), \quad k=0,1, \label{equ:equ:ExpDecStatVstarVariableCoefficientPerturbation}\\
       \left\|\psi\right\|_{W^{k,q}_{\theta_4}} \leqslant& 
C_3 \left(\Re\lambda+\bzero \right)^{-\frac{k}{2}}\left(\left\|\psi\right\|_{L^q_{\theta_3}}
      +\left\|\phi\right\|_{L^q_{\theta_4}}\right), \quad k=0,1. \label{equ:equ:ExpDecStatVstarVariableCoefficientPerturbationAdjoint}
  \end{align}
\end{theorem}

\begin{remark} The constants in \eqref{equ:equ:ExpDecStatVstarVariableCoefficientPerturbation}, \eqref{equ:equ:ExpDecStatVstarVariableCoefficientPerturbationAdjoint} 
are of the form $C_j=C_j(\gamma,|\mu_{j+1}-\mu_j|,R_0,\left\|Q\right\|_{L^{\infty}},\varepsilon)$, $j=1,3$,
and do not depend on $\lambda$, $v$ and $\psi$. Their precise dependence can be traced back to \cite[Thm.~2.10 with $C_{\theta}=1$]{BeynOtten2016a}.
Due to the choice of exponents in \eqref{equ:ExponentialRatesLQ}, \eqref{equ:ExponentialRatesLQstart},
the main effect is to show that solutions of inhomogeneous equations lie in a small space
of exponentially decaying solutions, provided they come from a large space of exponentially
growing solutions and provided the inhomogeneity belongs to the same small space of exponentially
decreasing functions.
\end{remark}

\begin{proof}
  Decompose $\lambda\in\C$ into $\lambda=\lambda_1+\lambda_2$ with $\lambda_2:=-\bzero+\gamma$, $\lambda_1:=\lambda-\lambda_2$, and 
  write $\lambda I-\L_Q = \lambda_1 I-\widetilde{\L}_Q$ with $\widetilde{\L}_Q:=\L_Q-\lambda_2 I$.
  This implies
  \begin{align}
    \label{equ:ModifiedResEq}
    g = (\lambda I-\L_Q)v = (\lambda_1-\widetilde{\L}_Q)v.
  \end{align}
  Introducing the matrix $\widetilde{B}_{\infty}:=B_{\infty}+\lambda_2 I$ and the two quantities $\tilde{b}_0 := \bzero + \lambda_2 = \gamma$, 
  $\tilde{\beta}_{\infty} := \beta_{\infty} + \lambda_2$, which satisfy $0<\tilde{b}_0\leqslant\tilde{\beta}_{\infty}$, we now apply 
  \cite[Thm.~3.5]{BeynOtten2016a} to \eqref{equ:ModifiedResEq} with
  \begin{align*}
    (\widetilde{L}_Q,\lambda_1,\widetilde{B}_{\infty},\tilde{b}_0,\tilde{\beta}_{\infty})\quad\text{instead of}\quad (\L_Q,\lambda,B_{\infty},\bzero,\beta_{\infty}).
  \end{align*}
  For this purpose, one must check the following three properties
  \begin{align*}
    \tilde{\beta}_{\infty}>0,\quad \Re\lambda_1\geqslant -(1-\varepsilon)\tilde{\beta}_{\infty},\quad\text{and}\quad
    \underset{|x|\geqslant R_0}{\esssup}\left|Q(x)\right| \leqslant \frac{\varepsilon}{2}\min\left\{\frac{\tilde{b}_0}{\kappa\aone},\frac{\tilde{b}_0}{C_{0,\varepsilon}},\tilde{\beta}_{\infty}\right\}.
  \end{align*}
  Using $0<\bzero\leqslant\tilde{\beta}_{\infty}$, $0<\varepsilon<1$, $\Re\lambda\geqslant-\bzero+\gamma$, $\lambda=\lambda_1+\lambda_2$ and $\lambda_2=-\bzero+\gamma$, we obtain
  \begin{align*}
    \tilde{\beta}_{\infty} =  \beta_{\infty} - \bzero + \gamma \geqslant \gamma >0,
\qquad
    \Re\lambda_1 = \Re\lambda - \lambda_2 \geqslant  0 \geqslant -(1-\varepsilon)\tilde{\beta}_{\infty}.
  \end{align*} 
  The $Q$-estimate follows from \eqref{equ:ConditionOnQ} using $\gamma=\tilde{b}_0$ and $\beta_{\infty}-\bzero+\gamma=\tilde{\beta}_{\infty}-\tilde{b}_0+\gamma=\tilde{\beta}_{\infty}$. 
  Replacing $(\eqref{cond:A4DC},p,v,g,\mu_1,\mu_2)$ by $(\eqref{cond:A4DCq},q,\psi,\phi,\mu_3,\mu_4)$, the same approach yields the assertion for 
  solutions $\psi$ of the adjoint problem $(\lambda I-\L_Q^*)\psi=\phi$.
\end{proof}

\subsection{Fredholm properties of the linearized operator and exponential decay of eigenfunctions}
\label{subsec:4.6}
We now apply the previous results from Section \ref{sec:4} to
\begin{align*}
  -B_{\infty}=Df(v_{\infty})\quad\text{and}\quad Q(x)=Df(v_{\star}(x))-Df(v_{\infty})
\end{align*}
in which case the linearization $\L$ from \eqref{equ:1.8} coincides with the variable coefficient operator $\L_Q$ from \eqref{equ:1.12a}. 
This allows us to transfer the Fredholm property (Corollary \ref{cor:4.4}), the Fredholm alternative (Lemma \ref{lem:4.5}) and the exponential 
decay (Theorem \ref{thm:APrioriEstimatesInLpRelativelyCompactPerturbation}) to the linearized operator $\L$ and its adjoint $\L^*$. 

Our main Theorem \ref{thm:FredPropLin} generalizes \cite[Lem.~2.4]{BeynLorenz2014} to higher space dimensions $d$ and to the general $L^p$-case. Moreover, it 
provides us exponential decay for the eigenfunctions of the linearized operator $\L$ and for the eigenfunctions of the adjoint operator $\L^*$.
In the following, let $\amax=\rho(A)$ denote the spectral radius of $A$, $-\azero=s(-A)$ the spectral bound of $-A$, $-\bzero=s(Df(v_{\infty}))$ 
the spectral bound of $Df(v_{\infty})$ and let $\beta_{\infty}$ be from \eqref{cond:A10}.

\begin{proof}[Proof (of Theorem \ref{thm:FredPropLin})]
  With $-B_{\infty}=Df(v_{\infty})$ and $Q(x)=Df(v_{\star}(x))-Df(v_{\infty})$ we obtain $\L=\L_Q$.
  \begin{itemize}[leftmargin=0.43cm]\setlength{\itemsep}{0.1cm}
  \item[a)] An application of Theorem \ref{thm:4.3} proves that $\lambda I-\L$ is Fredholm of index $0$ with finite-dimensional kernel and cokernel. In order to apply 
  Theorem \ref{thm:4.3}, note that the assumptions \eqref{cond:A4DC} and \eqref{cond:A5} are directly satisfied, and \eqref{cond:A8B} follows from \eqref{cond:A8}. 
  The property $Q\in C(\R^d,\C^{m,m})$ follows from \eqref{cond:A6} and $v_{\star}\in C^2(\R^d,\R^m)$. Similarly, condition \eqref{equ:PropertyQ3} follows from 
  \eqref{cond:A6} and \eqref{equ:BoundednessConditionForVStar}
  \begin{align*}
               \left|Q(x)\right|
            =& \left|Df(v_{\star}(x))-Df(v_{\infty})\right|
    \leqslant  \int_{0}^{1}\left|D^2 f(v_{\infty}+s\left(v_{\star}(x)-v_{\infty}\right)\right| ds \left|v_{\star}(x)-v_{\infty}\right| \\
      \leqslant& K_1\bigg(\sup_{z\in B_{K_1}(v_{\infty})}\left|D^2 f(z)\right|\bigg)
    \leqslant \frac{\varepsilon\gamma}{2}  \min\left\{\frac{1}{\kappa\aone},\frac{1}{C_{0,\varepsilon}}\right\},
  \end{align*}
  provided we choose $K_1=K_1(A,f,v_{\infty},\gamma,d,p,\varepsilon)>0$ such that
  \begin{align}
    \label{equ:thresholdconstant}
    K_1\bigg(\sup_{z\in B_{K_1}(v_{\infty})}\left|D^2 f(z)\right|\bigg)\leqslant
    \frac{\varepsilon\gamma}{2} \min\left\{\frac{1}{\kappa\aone},\frac{1}{C_{0,\varepsilon}}\right\}.
  \end{align}
  \item[b)] Since $\lambda\in\sigma_{\mathrm{pt}}(\L)$ has geometric multiplicity $n=\mathrm{dim}\,\mathcal{N}(\lambda I-\L)$ for some $n\in\N$, we deduce from 
  Lemma \ref{lem:4.5} (\textbf{or} case) that the homogeneous equations $(\lambda I-\L)v=0$ and $(\lambda I-\L)^*\psi=0$ have exactly $n$ (nontrivial) linearly 
  independent solutions $v_1,\ldots,v_n\in\D^p_{\mathrm{loc}}(\L_0)$ and $\psi_1,\ldots,\psi_n\in\D^q_{\mathrm{loc}}(\L_0^*)$. Further, Lemma \ref{lem:4.5} 
  implies that for any $g\in L^p(\R^d,\C^m)$ the inhomogeneous equation $(\lambda I-\L)v=g$ has at least one (not necessarily unique) solution 
  $v\in\D^p_{\mathrm{loc}}(\L_0)$ if and only if $g\in(\mathcal{N}((\lambda I-\L)^*))^{\bot}$, which corresponds \eqref{equ:OrthCond}. Finally, the estimates from 
  \eqref{equ:EstimateEF} follow from abstract results of Fredholm theory.
  \end{itemize}
\end{proof}

\begin{theorem}[Exponential decay of eigenfunctions]\label{thm:ExpDecEigLin}
  Let all assumptions of Theorem \ref{thm:FredPropLin} a)-b) hold.
  \begin{itemize}[leftmargin=0.43cm]\setlength{\itemsep}{0.1cm}
  \item[a)] (Exponential decay of eigenfunctions in weighted $L^p$-spaces). 
Consider weight functions $\theta_j(x)=\theta(x,\mu_j)$, $j=1,\ldots,4$ 
with exponents that satisfy
 \eqref{equ:ExponentialRatesLQ} and \eqref{equ:ExponentialRatesLQstart}.
   
  Then every classical solution $v\in C^2(\R^d,\C^m)$ and $\psi\in C^2(\R^d,\C^m)$ of the eigenvalue problems 
  \begin{align}
    \label{equ:NonlinearEigenvalueProblemRealFormulation}
    (\lambda I -\L)v = 0\quad\text{and}\quad (\lambda I -\L)^*\psi = 0,
  \end{align}
  such that $v\in L^p_{\theta_1}(\R^d,\C^m)$ and $\psi\in L^q_{\theta_3}(\R^d,\C^m)$ 
  satisfies $v\in W^{1,p}_{\theta_2}(\R^d,\C^m)$ and $\psi\in W^{1,q}_{\theta_4}(\R^d,\C^m)$.
  \item[b)] (Pointwise exponential decay of eigenfunctions). In addition to a),
let $\frac{d}{p}\leqslant 2$, $f\in C^{k}(\R^m,\R^m)$, $v_{\star}\in C^{k+1}(\R^d,\R^m)$ and 
  $v\in C^{k+1}(\R^d,\C^m)$ for some $k\in\N$ with $k\geqslant 2$. Then $v$ belongs to $W^{k,p}_{\theta_2}(\R^d,\C^m)$ and satisfies the pointwise estimate
  \begin{align} \label{equ:ratederiv}
    \left|D^{\alpha}v(x)\right| \leqslant C\exp\left(-\mu_2\sqrt{|x|^2+1}\right),
    \quad x\in\R^d
  \end{align}
  for every exponential decay rate $0\leqslant\mu_2\leqslant\varepsilon\frac{\sqrt{\azero\gamma}}{\amax p}$ and for every multi-index $\alpha\in\N_0^d$ 
  with $|\alpha|<k-\tfrac{d}{p}$. 
  \item[c)] (Pointwise exponential decay of adjoint eigenfunctions). In addition to b), let $\max\{\tfrac{d}{p},\tfrac{d}{q}\}\leqslant 2$ and $\psi\in C^{k+1}(\R^d,\C^m)$. 
  Then $\psi$ belongs to $W^{k,q}_{\theta_4}(\R^d,\C^m)$ and satisfies the pointwise estimate
  \begin{align} \label{equ:ratederivadjoint}
    \left|D^{\alpha}\psi(x)\right| \leqslant C\exp\left(-\mu_4\sqrt{|x|^2+1}\right),
    \quad x\in\R^d
  \end{align}
  for every exponential decay rate $0\leqslant\mu_4\leqslant\varepsilon\frac{\sqrt{\azero\gamma}}{\amax q}$ and for every multi-index $\alpha\in\N_0^d$ 
  with $|\alpha|<k-\tfrac{d}{q}$. 
  \end{itemize}
\end{theorem}
\begin{proof}
  As in the proof of Theorem \ref{thm:FredPropLin}. let $-B_{\infty}=Df(v_{\infty})$, $Q(x)=Df(v_{\star}(x))-Df(v_{\infty})$. 
 Assertion a) follows directly from an application of Theorem \ref{thm:APrioriEstimatesInLpRelativelyCompactPerturbation} if 
  $K_1$ from \eqref{equ:thresholdconstant} is chosen such that
  \begin{align*}
    K_1\bigg(\sup_{z\in B_{K_1}(v_{\infty})}\left|D^2 f(z)\right|\bigg)\leqslant
    \frac{\varepsilon\gamma}{2}\min\left\{\frac{1}{\kappa\aone},\frac{1}{C_{0,\varepsilon}},\frac{\beta_{\infty}-\bzero}{\gamma}+1\right\}.
  \end{align*}
  The proof of b) works in quite an analogous fashion as in \cite[Thm.~5.1(2)]{BeynOtten2016a} and will not be repeated here. 
 Similarly, assertion c) follows when applying the theory from \cite{BeynOtten2016a} to the adjoint operator. 
\end{proof}

%
%
\sect{Point spectrum and exponential decay of eigenfunctions}
\label{sec:3}

\subsection{Formal derivation of isolated eigenvalues and eigenfunctions}
\label{subsec:3.1}
In this section we compute the set of eigenvalues $\lambda$ (and associated eigenfunctions $v$) of \eqref{equ:1.9}, which are caused by 
symmetries of the $\SE(d)$-group action and belong to the point spectrum $\sigma_{\mathrm{pt}}(\L)$ of the linearization $\L$ (Theorem \ref{thm:3.1}).

Consider the eigenvalue problem
\begin{align}
  \label{equ:1}
  0 = (\lambda I-\L)v = \lambda v - A\triangle v - (Dv)(Sx) + Df(v_{\star})v,\,x\in\R^d.
\end{align}
Assume an eigenfunction $v$ of the form
\begin{align}
  \label{equ:2}
  v = (D v_{\star})(Ex+b)\quad\quad\text{for some $E\in\C^{d,d}$, $b\in\C^d$, $E^{\top}=-E$, $v_{\star}\in C^3(\R^d,\R^m)$.}
\end{align}
Plugging \eqref{equ:2} into \eqref{equ:1} and using the equalities
\begin{align}
  &\lambda v = (Dv_{\star})(\lambda(Ex+b)),
  \label{equ:3} \\
  &A\triangle v = (D(A\triangle v_{\star}))(Ex+b),
  \label{equ:4} \\
  &(D v)(Sx) = (D((Dv_{\star})(Sx)))(Ex+b) + (Dv_{\star})([E,S]x-Sb),
  \label{equ:5} \\
  &Df(v_{\star})v = (D(f(v_{\star})))(Ex+b),
  \label{equ:6} 
\end{align}
with Lie brackets $[E,S]:=ES-SE$, we obtain
\begin{align}
  \label{equ:7}
  0 = (Dv_{\star})\left((\lambda E-[E,S])x + (\lambda b+Sb)\right) - D\left(A\triangle v_{\star}+(Dv_{\star})(Sx)+f(v_{\star})\right)(Ex+b).
\end{align}
Since $v_{\star}$ satisfies the rotating wave equation
\begin{align}
  \label{equ:8}
  0 = A\triangle v_{\star}+(Dv_{\star})(Sx)+f(v_{\star}),\,x\in\R^d,
\end{align}
the second term in \eqref{equ:7} vanishes and we end up with
\begin{align}
  \label{equ:9}
  0 = (Dv_{\star})\left((\lambda E-[E,S])x + (\lambda b+Sb)\right),\,x\in\R^d.
\end{align}
Comparing coefficients in \eqref{equ:9} yields the finite-dimensional eigenvalue problem
\begin{subequations}
  \label{equ:10}
  \begin{align}
  \lambda E &= [E,S],
  \label{equ:10a} \\
  \lambda b &= -Sb,
  \label{equ:10b}  
  \end{align}
\end{subequations}
which is to be solved for $(\lambda,E,b)$. Since $E$ is required to be skew-symmetric, we expect $\frac{d(d+1)}{2}$ nontrivial solutions. 
If $(\lambda,E)$ is a solution of \eqref{equ:10a}, then $(\lambda,E,0)$ solves \eqref{equ:10}, and
if $(\lambda,b)$ is a solution of \eqref{equ:10b}, then $(\lambda,0,b)$ solves \eqref{equ:10}. The eigenvalue problem \eqref{equ:10a} is treated in 
\cite{BlochIserles2005}, but for completeness we discuss its solution here. 
Since $S$ is skew-symmetric there  
is a unitary matrix $U\in\C^{d,d}$  such that $S=U\Lambda_S U^{\herm}$, 
where $\Lambda_S=\diag(\lambda_1^S,\ldots,\lambda_d^S)$ and $\lambda_1^S,\ldots,\lambda_d^S\in i\R$ denote the eigenvalues of $S$.
In particular, this implies $S^{\top}=\overline{U}\Lambda_S U^{\top}$.
First of all, equation \eqref{equ:10b} has solutions $(\lambda,b)=(-\lambda_l^S,Ue_l)$, so that
 \eqref{equ:10} has solutions $(\lambda,E,b)=(-\lambda_l^S,0,Ue_l)$ for $l=1,\ldots,d$.
Next we solve \eqref{equ:10a}. Multiplying \eqref{equ:10a} from the left by $U^{\herm}$, from the right by $\bar{U}$, defining $\tilde{E}=U^{\herm}E\overline{U}$, 
and using skew-symmetry of $S$ and $\tilde{E}$, we obtain
\begin{align}
  \label{equ:12}
  \lambda\tilde{E} = \lambda U^{\herm}E\overline{U} = U^{\herm}[E,S]\overline{U} = -U^{\herm}\big(E\overline{U}\Lambda_S U^{\top} + U\Lambda_S U^{\herm}E\big)\overline{U} 
  = -\tilde{E}\Lambda_S - \Lambda_S\tilde{E} = \tilde{E}^{\top}\Lambda_S - \Lambda_S\tilde{E}.
\end{align}
Equation \eqref{equ:12} has solutions $(\lambda,\tilde{E})=(-(\lambda_i^S+\lambda_j^S),I_{ij}-I_{ji})$,
where $I_{ij}=e_ie_j^{\top}$ has the entry $1$ in the $i$th row and the $j$th column and is $0$ otherwise. Therefore, equation \eqref{equ:10a} has solutions 
$(\lambda,E)=(-(\lambda_i^S+\lambda_j^S),U(I_{ij}-I_{ji})U^{\top})$, and \eqref{equ:10} has solutions $(\lambda,E,b)=(-(\lambda_i^S+\lambda_j^S),U(I_{ij}-I_{ji})U^{\top},0)$ 
for $i=1,\ldots,d-1$, $j=i+1,\ldots,d$.

Collecting these eigenvalues and using skew-symmetry of $S$ once more, we find that the \begriff{symmetry set}
\begin{align}
  \label{equ:3.12b}
  \sigma_{\mathrm{sym}}(\L) = \sigma(S) \cup \{\lambda_i^S+\lambda_j^S:1\leqslant i<j\leqslant d\}
\end{align}
belongs to  $\sigma_{\mathrm{pt}}(\L)$. The rigorous version of this result is Theorem \ref{thm:3.1b}, which
will be proved in the next section.

\subsection{Point spectrum in $L^p$ and exponential decay of eigenfunctions}
\label{subsec:3.2}


\begin{proof}[Proof (of Theorem \ref{thm:3.1b})]
  We show that any $\lambda\in\sigma_{\mathrm{sym}}(\L)$ belongs to $\sigma(\L)$ and is an isolated eigenvalue of finite multiplicity. Then, we obtain 
  $\lambda\in\sigma_{\mathrm{pt}}(\L)$, hence $\sigma_{\mathrm{sym}}(\L)\subseteq\sigma_{\mathrm{pt}}(\L)$, by the definition of the point spectrum, cf. 
  Definition \ref{def:1.1}d). 
  Let $\lambda\in\sigma_{\mathrm{sym}}(\L)$ and let $E,b$ be given by \eqref{equ:3.2} and \eqref{equ:3.3}. Then $v(x)=(Dv_{\star}(x))(Ex+b)$ is a 
  classical solution of \eqref{equ:1.5} according to Theorem \ref{thm:3.1}. An application of \cite[Cor.~4.1]{BeynOtten2016a} implies $v_{\star}\in W^{3,p}_{\theta}(\R^d,\R^m)$, 
  thus $v\in W^{2,p}(\R^d,\C^m)$ and $\L_0 v\in L^p(\R^d,\C^m)$, and hence $v\in\D^p_{\mathrm{loc}}(\L_0)$ solves \eqref{equ:1.5} in $L^p$. 
  Therefore, $v$ is an eigenfunction of $\L$ in $L^p$ with eigenvalue $\lambda$. From $v\in\D^p_{\mathrm{loc}}(\L_0)$ we further deduce that 
  $\mathcal{N}(\lambda I-\L)\neq\emptyset$, hence $\lambda I-\L$ is not injective and thus $\lambda\in\sigma(\L)$. 
  The spectral stability of $Df(v_{\infty})$ from \eqref{cond:A10} implies that $\sigma_{\mathrm{ess}}(\L)$ lies in the left half-plane $\{\lambda\in\C:\Re\lambda\leqslant -\bzero\}$, 
  cf. Theorem \ref{thm:EssSpecLRW}. Thus,  $\lambda\notin\sigma_{\mathrm{ess}}(\L)$, therefore $\lambda\in\sigma_{\mathrm{pt}}(\L)$. 
  In fact, applying Theorem \ref{thm:FredPropLin}a) yields $\lambda I-\L$ to be Fredholm of index $0$, hence the eigenvalue $\lambda$ is isolated and of finite multiplicity. 
\end{proof}

\begin{remark}
  \begin{itemize}[leftmargin=0.43cm]\setlength{\itemsep}{0.1cm}
  \item[a)] (Rotational term).
  Since $v_{\star}\in C^3(\R^d,\R^m)$, the rotational term
  \begin{align*}
    v(x)=\left\langle Sx,\nabla v_{\star}(x)\right\rangle\in C^2(\R^d,\R^m),\;d\geqslant 2
  \end{align*} 
  is a classical solution of $\L v=0$, i.e. $v$ is an eigenfunction of $\L$ with eigenvalue $\lambda=0$. This can either be shown directly 
  or it can be deduced from Theorem \ref{thm:3.1} with $(\lambda,E,b)=(0,S,0)$, cf. \cite{BeynLorenz2008} for $d=2$.
  \item[b)] (Multiplicities of isolated eigenvalues). 
  Theorem \ref{thm:3.1} gives also information about the multiplicity of the isolated eigenvalues of $\L$. More precisely, for 
  any fixed skew-symmetric $S\in\R^{d,d}$, Theorem \ref{thm:3.1} yields a lower bound for the multiplicities. But note that 
  multiplicities depend on $S$, i.e. varying $S$ may lead to different eigenvalues with different multiplicities.
  \end{itemize}
\end{remark}

Figure \ref{fig:PointSpectrumOfTheLinearization} shows the eigenvalues $\lambda\in\sigma_{\mathrm{sym}}(\L)$ from Theorem 
\ref{thm:3.1} and lower bounds for their corresponding multiplicities for different space dimensions $d=2,3,4,5$. Corresponding to 
\eqref{equ:3.12b}, the eigenvalues $\lambda\in\sigma(S)$ are illustrated by the blue circles, the eigenvalues 
$\lambda\in\left\{\lambda_i+\lambda_j\mid \lambda_i,\lambda_j\in\sigma(S),\,1\leqslant i<j\leqslant d\right\}$ 
are illustrated by the green crosses. The imaginary values to the right of the symbols denote the precise values of eigenvalues 
and the numbers to the left the lower bounds for their corresponding multiplicities. We observe that for space dimension $d$ 
there are $\frac{d(d+1)}{2}$ eigenvalues on the imaginary axis that are caused by the symmetries of the $\SE(d)$-group action.

\begin{example}[Point spectrum of $\L$ for $d=2$]\label{exa:PointSpectrum2D} 
  In case $d=2$ the skew-symmetric matrix $S\in\R^{2,2}$, the diagonal matrix $\Lambda_S\in\C^{2,2}$ and the unitary matrix $U\in\C^{2,2}$, satisfying 
  $S=U\Lambda_S U^{\herm}$, are given by
  \begin{align*}
    S=\begin{pmatrix}0 &S_{12}\\-S_{12}&0\end{pmatrix},\quad
    \Lambda_S=\begin{pmatrix}i\sigma_1&0\\0&-i\sigma_1\end{pmatrix},\quad
    U=\frac{1}{\sqrt{2}}\begin{pmatrix}1 &1\\i&-i\end{pmatrix}
  \end{align*}
  with $\sigma_1=S_{12}$, $k=1$, $\lambda_1^S=i\sigma_1$, $\lambda_2^S=-i\sigma_1$.
  Therefore, using the relation $U(I_{12}-I_{21})U^T=-i(I_{12}-I_{21})$, Theorem \ref{thm:3.1} implies the following eigenvalues and eigenfunctions of $\L$, 
  cf. \cite[Lem.~2.3]{BeynLorenz2008},
  \begin{equation}
    \begin{aligned}
    \label{equ:Eigenfunctions2D}
    \begin{split}
      \lambda_1     &= 0,              &&v_1=D^{(1,2)}v_{\star},\\
      \lambda_{2,3} &= \pm i\sigma_1,  &&v_{2,3}=D_1 v_{\star}\pm i D_2 v_{\star}
    \end{split}
    \end{aligned}
  \end{equation}
\end{example}

\begin{figure}[ht]
  \centering
  \subfigure[$d=2$]{\includegraphics[page=4, height=7.6cm]{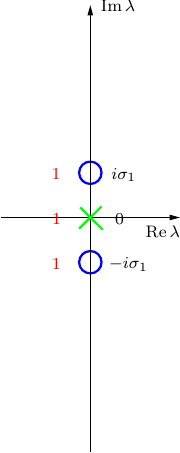} \label{fig:Pointspectrum_d2}}\hspace*{0.5cm}
  \subfigure[$d=3$]{\includegraphics[page=5, height=7.6cm]{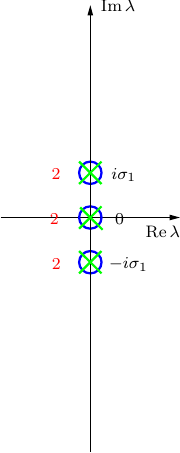} \label{fig:Pointspectrum_d3}}\hspace*{0.5cm}
  \subfigure[$d=4$]{\includegraphics[page=6, height=7.6cm]{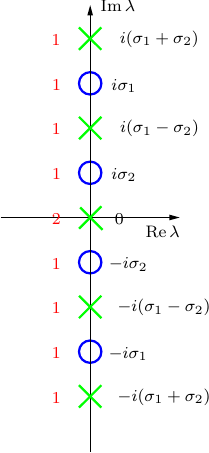} \label{fig:Pointspectrum_d4}}
  \subfigure[$d=5$]{\includegraphics[page=7, height=7.6cm]{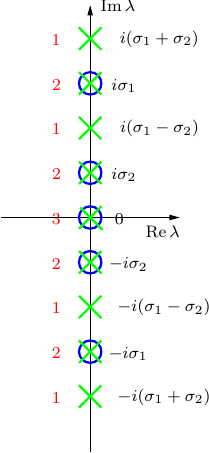}\label{fig:Pointspectrum_d5}}
  \caption{Point spectrum of the linearization $\L$ on the imaginary axis $i\R$ for space dimensions $d=2,3,4,5$ and $\mathrm{dim}\,\SE(d)=3,6,10,15$ given by Theorem \ref{thm:3.1}.}
  \label{fig:PointSpectrumOfTheLinearization}
\end{figure}

\begin{example}[Point spectrum of $\L$ for $d=3$]\label{exa:PointSpectrum3D} 
  In case $d=3$ the skew-symmetric matrix $S\in\R^{3,3}$, the diagonal matrix $\Lambda_S\in\C^{3,3}$ and the unitary matrix $U\in\C^{3,3}$, satisfying 
  $S=U\Lambda_S U^{\herm}$, are given by
  \begin{align*}
    S=\begin{pmatrix}0 &S_{12} &S_{13} \\ -S_{12} &0 &S_{23} \\ -S_{13} & -S_{23} &0\end{pmatrix},\;
    \Lambda_S=\begin{pmatrix}i\sigma_1&0&0\\0&-i\sigma_1&0\\0&0&0\end{pmatrix},\;
    U= \frac{1}{\sigma_1}\begin{pmatrix}
         \frac{\sigma_1 S_{13}-iS_{12}S_{23}}{\sigma_{12}} 
        &\frac{\sigma_1 S_{13}+iS_{12}S_{23}}{\sigma_{12}}
        &S_{23} \\ 
         \frac{\sigma_1 S_{23}+iS_{12}S_{13}}{\sigma_{12}} 
        &\frac{\sigma_1 S_{23}-iS_{12}S_{13}}{\sigma_{12}}
        &-S_{13} \\ 
         \frac{i(S_{13}^2+S_{23}^2)}{\sigma_{12}} 
        &\frac{-i(S_{13}^2+S_{23}^2)}{\sigma_{12}}
        &S_{12}
      \end{pmatrix},
  \end{align*}
  with $\sigma_1=\sqrt{S_{12}^2+S_{13}^2+S_{23}^2}$, $\sigma_{12}=\sqrt{2(S_{13}^2+S_{23}^2)}$, $k=1$, $\lambda_1^S=i\sigma_1$, $\lambda_2^S=-i\sigma_1$ and $\lambda_3^S=0$.
  Therefore, using the relations
  \begin{align*}
    &U(I_{12}-I_{21})U^T=\frac{i}{\sigma_1}S, \\
    &U(I_{13}-I_{31})U^T=\frac{1}{2 \sigma_1}
                        \begin{pmatrix}
                           0
                          &-\sigma_{12}
                          &\frac{2(S_{12}S_{13}+i\sigma_1 S_{23})}{\sigma_{12}} \\
                           \sigma_{12}
                          &0
                          &-\frac{2(-S_{12}S_{23}+i\sigma_1 S_{13})}{\sigma_{12}} \\
                           -\frac{2(S_{12}S_{13}+i\sigma_1 S_{23})}{\sigma_{12}}
                          &\frac{2(-S_{12}S_{23}+i\sigma_1 S_{13})}{\sigma_{12}}
                          &0
                        \end{pmatrix}, \\
    &U(I_{23}-I_{32})U^T= \frac{1}{2\sigma_1}\begin{pmatrix}
                           0
                          &-\sigma_{12}
                          &-\frac{2(-S_{12}S_{13}+i\sigma_1 S_{23})}{\sigma_{12}} \\
                           \sigma_{12}
                          &0
                          &\frac{2(S_{12}S_{23}+i\sigma_1 S_{13})}{\sigma_{12}} \\
                           \frac{2(-S_{12}S_{13}+i\sigma_1 S_{23})}{\sigma_{12}}
                          &-\frac{2(S_{12}S_{23}+i\sigma_1 S_{13})}{\sigma_{12}}
                          &0
                        \end{pmatrix},
  \end{align*}
  Theorem \ref{thm:3.1} yields the following eigenvalues and eigenfunctions of $\L$, 
  \begin{equation}
  \begin{aligned}
  \label{equ:Eigenfunctions3D}
  \begin{split}
    \lambda_1 &= 0,          &&v_1=S_{12}D^{(1,2)}v_{\star} + S_{13}D^{(1,3)}v_{\star} + S_{23}D^{(2,3)}v_{\star},\\
    \lambda_2 &= 0,          &&v_2=S_{23}D_1 v_{\star} - S_{13}D_2 v_{\star} + S_{12}D_3 v_{\star},\\
    \lambda_{3,4} &= \pm i\sigma_1,  &&v_{3,4}=(\sigma_1 S_{13}\pm iS_{12}S_{23})D_1 v_{\star} 
                                      + (\sigma_1 S_{23}\pm iS_{12}S_{13})D_2 v_{\star} \pm i\frac{\sigma_{12}^2}{2}D_3 v_{\star}, \\
    \lambda_{5,6} &= \pm i\sigma_1,  &&v_{5,6}=-\frac{\sigma_{12}^2}{2}D^{(1,2)}v_{\star}+(S_{12}S_{13}\mp i\sigma_1 S_{23})D^{(1,3)}v_{\star} 
                +(S_{12}S_{23}\pm i\sigma_1 S_{13})D^{(2,3)}v_{\star}.
  \end{split}
  \end{aligned}
  \end{equation}
\end{example}


%
%
\sect{Numerical spectra and eigenfunctions at spinning solitons in the cubic-quintic complex Ginzburg-Landau equation}
\label{sec:5}
Consider the \begriff{cubic-quintic complex Ginzburg-Landau equation (QCGL)}, \cite{GinzburgLandau1950},
\begin{align}
  \label{equ:ComplexQuinticCubicGinzburgLandauEquation}
  u_t = \alpha\triangle u + u\left(\delta + \beta\left|u\right|^2 + \gamma\left|u\right|^4\right)
\end{align}
where $u:\R^d\times[0,\infty)\rightarrow\C$, $d\in\{2,3\}$, $\alpha,\beta,\gamma,\delta\in\C$ with $\Re\alpha>0$ and $f:\C\rightarrow\C$ given by
\begin{align}
  \label{equ:NonlinearityGLComplexVersion}
  f(u) := u\left(\delta + \beta\left|u\right|^2 + \gamma\left|u\right|^4\right).
\end{align}

For the parameters, see \cite{CrasovanMalomedMihalache2001},
\begin{align}
  \label{equ:SpinningSolitonsParameters} 
  \alpha=\frac{1}{2}+\frac{1}{2}i,\quad\beta=\frac{5}{2}+i,\quad\gamma=-1-\frac{1}{10}i,\quad\mu=-\frac{1}{2}
\end{align}
this equation exhibits so called \begriff{spinning soliton} solutions.

\begin{figure}[ht]
  \includegraphics[page=8, width=0.97\textwidth]{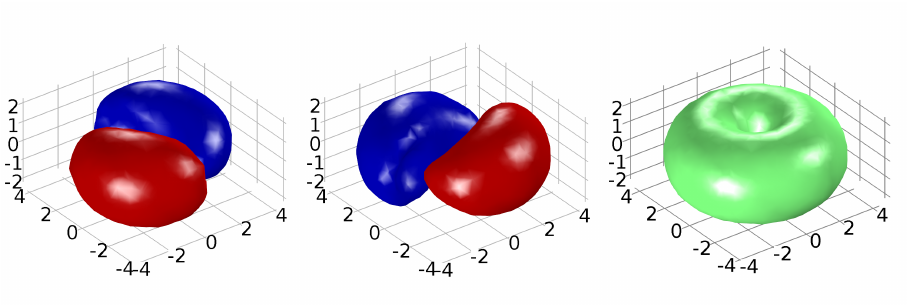}
  \label{fig:QCGL3DProfileSpinningSoliton}
  \caption{Isosurfaces of spinning solitons of QCGL \eqref{equ:ComplexQuinticCubicGinzburgLandauEquation} for parameters \eqref{equ:SpinningSolitonsParameters} and $d=3$: $\Re v_{\star}(x)=\pm0.5$ (left), $\Im v_{\star}(x)=\pm0.5$ (middle), $|v_{\star}(x)|=0.5$ (right).}
\end{figure}

Figure \ref{fig:QCGL3DProfileSpinningSoliton} shows the isosurfaces of $\Re v_{\star}(x)=\pm0.5$ (left), $\Im v_{\star}(x)=\pm0.5$ (middle), and 
$|v_{\star}(x)|=0.5$ (right) of a spinning soliton profile $v_{\star}$ for $d=3$. The rotational velocity matrix $S$ from Example \ref{exa:PointSpectrum3D} takes the values
$(S_{12},S_{13},S_{23}) = (0.6888,-0.0043,-0.0043)$. Therefore, the eigenvalues of $S$ are $\sigma(S)=\{0,\pm i\sigma_1\}$ with $\sigma_1=\sqrt{S_{12}^2+S_{13}^2+S_{23}^2}\approx 0.6888$. 
Moreover, the temporal period $T^{3d}$ for exactly one rotation of the soliton is $T^{3d}=\frac{2\pi}{|\sigma_1|}=9.1216$. The profile $v_{\star}$ and the velocity matrix $S$ 
of the spinning soliton are computed simultaneously by the \begriff{freezing method} from \cite{BeynOttenRottmannMatthes2013,BeynThuemmler2004}. For more detailed information concerning the 
computation of $v_{\star}$ and $S$ we refer to \cite{Otten2014}. To our knowledge there is no explicit formula for spinning soliton solutions of 
\eqref{equ:ComplexQuinticCubicGinzburgLandauEquation}, only implicit formulas and numerical approximations are available. 

The real-valued version of \eqref{equ:ComplexQuinticCubicGinzburgLandauEquation} reads as follows
\begin{align}
  \label{equ:ComplexQuinticCubicGinzburgLandauEquationRealVersion}
  \mathbf{u}_t = \mathbf{A}\triangle\mathbf{u} + \mathbf{f}(\mathbf{u})\quad\text{with}\quad 
  \textbf{A}:= \begin{pmatrix} \alpha_1 & -\alpha_2 \\ \alpha_2 & \alpha_1\end{pmatrix},\quad
  \mathbf{u}=\begin{pmatrix} u_1 \\ u_2 \end{pmatrix}
\end{align}
and $\mathbf{f}:\R^2\rightarrow\R^2$ given by
\begin{align}
  \label{equ:NonlinearityGLRealVersion}
  \mathbf{f}\begin{pmatrix} u_1 \\ u_2 \end{pmatrix} := 
  \begin{pmatrix} \left(u_1\delta_1-u_2\delta_2\right) + \left(u_1\beta_1-u_2\beta_2\right)\left(u_1^2+u_2^2\right) + \left(u_1\gamma_1-u_2\gamma_2\right)\left(u_1^2+u_2^2\right)^2 \\
                  \left(u_1\delta_2+u_2\delta_1\right) + \left(u_1\beta_2+u_2\beta_1\right)\left(u_1^2+u_2^2\right) + \left(u_1\gamma_2+u_2\gamma_1\right)\left(u_1^2+u_2^2\right)^2\end{pmatrix},
\end{align}
where $u=u_1+iu_2$, $\alpha=\alpha_1+i\alpha_2$, $\beta=\beta_1+i\beta_2$, $\gamma=\gamma_1+i\gamma_2$, $\delta=\delta_1+i\delta_2$.

The real-valued formulation \eqref{equ:ComplexQuinticCubicGinzburgLandauEquationRealVersion} of the QCGL \eqref{equ:ComplexQuinticCubicGinzburgLandauEquation} yields the following
constants
\begin{align*}
  \azero=\Re\alpha,\quad \amin=\amax=|\textbf{A}|=|\alpha|,\quad \aone=\left(\frac{|\alpha|}{\Re\alpha}\right)^{\frac{d}{2}},\quad
  \bzero=\beta_{\infty}=-\Re\delta,\quad v_{\infty}=0,
\end{align*}
it satisfies our assumptions \eqref{cond:A1}--\eqref{cond:A10} provided that
\begin{align}
  \label{equ:AssumptionsQCGL}
  \Re\alpha>0,\quad \Re\delta<0,\quad p_{\min}= \frac{2|\alpha|}{|\alpha|+\Re\alpha}<p<\frac{2|\alpha|}{|\alpha|-\Re\alpha}= p_{\max}.
\end{align}
In particular, \eqref{cond:A4} and \eqref{cond:A4q} for $1<p<\infty$ and $q=\frac{p}{p-1}$ lead to the same restriction on $p$, namely
\begin{align*}
  \frac{\Re\alpha}{|\alpha|}=\mu_1(\alpha)=\mu_1(\bar{\alpha}) > \frac{|q-2|}{2}=\frac{|p-2|}{2},
\end{align*}
which is equivalent to the last condition in \eqref{equ:AssumptionsQCGL} and $q_{\min}:=p_{\min}<q<p_{\max}=:q_{\max}$. In particular, if $p$ approaches $p_{\max}$ 
(or $p_{\min}$) then $q$ approaches $p_{\min}$ (or $p_{\max}$). Note that the application of Theorem \ref{thm:EssSpecLRW} and \ref{thm:3.1b} additionally requires 
$\tfrac{d}{p}\leqslant 2$. For the parameter values \eqref{equ:SpinningSolitonsParameters} this allows us to choose $p$ such that
\begin{align}
  \label{equ:PBoundGinzburgLandau}
  1.1716\approx\frac{4}{2+\sqrt{2}}=p_{\min}<p<p_{\max}=\frac{4}{2-\sqrt{2}}\approx 6.8284\quad\text{and}\quad p\geqslant \frac{d}{2},
\end{align}
e.g. $p=2,\ldots,6$. For a more detailed discussion of the assumptions \eqref{cond:A1}--\eqref{cond:A10} we refer to \cite{BeynOtten2016a,Otten2014}. 

Next we study the application of our spectral theorems and compare with
numerical eigenfunction computations.
Transforming the real-valued version \eqref{equ:ComplexQuinticCubicGinzburgLandauEquationRealVersion} into a co-rotating frame, linearizing at 
a rotating wave solution of \eqref{equ:ComplexQuinticCubicGinzburgLandauEquationRealVersion}, and applying temporal Fourier transform leads us 
to the eigenvalue problem $(\lambda I-\L)\mathbf{v}=0$ for the linearized operator
\begin{align*}
  \L \mathbf{v}(x) = \mathbf{A}\triangle\mathbf{v}(x) + \left\langle Sx,\nabla \mathbf{v}(x)\right\rangle + D\mathbf{f}(\mathbf{v_{\star}}(x))\mathbf{v}(x),
\end{align*}
where $D\mathbf{f}(\mathbf{u})$ denotes the Jacobian of $\mathbf{f}:\R^2\rightarrow\R^2$ defined by $(D\mathbf{f}(\mathbf{u}))_{ij}=
\frac{\partial\mathbf{f}_i}{\partial\mathbf{u}_j}(\mathbf{u})$, $i,j\in\{1,2\}$, $u\in\R^2$.
The associated adjoint eigenvalue problem reads as $(\lambda I-\L)^*\mathbf{\psi}=0$ with adjoint linearized operator
\begin{align*}
  \L^* \mathbf{\psi}(x) = \mathbf{A}^T\triangle\mathbf{\psi}(x) - \left\langle Sx,\nabla \mathbf{\psi}(x)\right\rangle + D\mathbf{f}(\mathbf{v_{\star}}(x))^T\mathbf{\psi}(x).
\end{align*}

We solve numerically the eigenvalues problems $(\lambda I-\L)\mathbf{v}=0$ and $(\lambda I-\L)^*\mathbf{\psi}=0$, obtained from the QCGL 
\eqref{equ:ComplexQuinticCubicGinzburgLandauEquationRealVersion}.
 The computations are realized by the CAE software Comsol Multiphysics \cite{ComsolMultiphysics52}. We use 
continuous piecewise linear finite elements with maximal stepsize $\triangle x=0.8$, and homogeneous Neumann boundary conditions to 
compute $\textrm{neig}=800$ eigenvalues which are located near $\sigma=-\bzero$ (measured radially) and satisfy the eigenvalue tolerance 
$\textrm{etol}=10^{-7}$. The profile $v_{\star}$ and the velocity matrix $S$ 
are obtained by simulation. A more detailed discussion 
of how to get these quantities may be found in \cite{Otten2014}.


Figure \ref{fig:QCGL3DSpectrumSpinningSoliton}(a) shows the dispersion set $\sigma_{\mathrm{disp}}(\mathcal{L})$ (red lines) and the symmetry set $\sigma_{\mathrm{sym}}(\mathcal{L})$ 
(blue circles). Both of them belong to the analytical spectrum $\sigma(\L)$ of $\L$. Recall the entries and the spectrum of the velocity matrix $S\in\R^{3,3}$, namely
\begin{align}
  \label{equ:QCGLSpinningSoliton3DEssentialSpectrumParameters}
  (S_{12},S_{13},S_{23}) = (0.6888,-0.0043,-0.0043),\;\; \sigma(S)=\left\{\pm\sigma_1 i\right\},\;\; \sigma_1=\sqrt{S_{12}^2+S_{13}^2+S_{23}^2}=0.6888.
\end{align} 
Therefore, the symmetry set reads as follows, cf. \eqref{equ:3.12b},
\begin{align}
  \label{equ:SymmetrySetQCGL}
  \sigma_{\mathrm{sym}}(\mathcal{L}) = \{0,\pm i\sigma_1\}
\end{align}
and the dispersion set as, cf. \eqref{equ:6.16},
\begin{align}
  \label{equ:DispersionSetQCGL}
  \sigma_{\mathrm{disp}}(\mathcal{L}) = \{\lambda = -\eta^2\alpha_1 + \delta_1 + i(\pm\eta^2\alpha_2 \mp\delta_2 - n\sigma_1): \eta\in\R,\, n\in\Z\}.
\end{align}
As shown in Theorem \ref{thm:3.1}, Example \ref{exa:PointSpectrum3D} and Figure \ref{fig:PointSpectrumOfTheLinearization}(b), the eigenvalues from $\sigma_{\mathrm{sym}}(\mathcal{L})$, caused by the group symmetries of $\SE(3)$, 
lie on the imaginary axis, have (at least) multiplicity $2$, and belong to the point spectrum $\sigma_{\mathrm{pt}}(\L)$ in $L^p$. 
Similarly, as shown in Theorem \ref{thm:EssSpecLRW} and Figure \ref{fig:EssentialSpectrum}(a), the eigenvalues from $\sigma_{\mathrm{disp}}(\mathcal{L})$ belong to the 
essential spectrum $\sigma_{\mathrm{ess}}(\L)$ in $L^p$ and form a zig zag structure consisting of infinitely many copies of cones. The cones open to the left and 
their tips are located at $-\bzero+in\sigma_1$, $n\in\Z$. This is easily  seen from \eqref{equ:DispersionSetQCGL}. Therefore, the distance of two neighboring tips of 
the cones equals $\sigma_1=0.6888$. Theorem \ref{thm:FredPropLin} shows that $\lambda I-\L$ is Fredholm of index $0$, provided that $\Re\lambda$ is located to the right 
of $-\bzero$ (back dashed line). Therefore, there is no essential spectrum
to the right of the line $-\bzero$, i.e. all values $\lambda$ with $\Re\lambda>-\bzero$ either 
belong to the resolvent set $\rho(\L)$ or to the point spectrum $\sigma_{\mathrm{pt}}(\L)$. For the regions enclosed between the black dashed line and the essential spectrum we believe that 
the operator $\lambda I-\L$ is Fredholm of index $0$, but we don't have
a proof. Similarly, the Fredholm index for those $\lambda$ lying in the
 rhombic regions within the dispersion set 
remains an open problem. To conclude, we suggest that in general both sets $\sigma_{\mathrm{pt}}(\L)$ and $\sigma_{\mathrm{ess}}(\L)$ may be larger than 
$\sigma_{\mathrm{sym}}(\mathcal{L})$ and $\sigma_{\mathrm{disp}}(\mathcal{L})$, respectively. Moreover, the spectrum of the adjoint $\L^*$ coincides with the spectrum of $\L$. 

\begin{figure}[ht]
  \centering
  \subfigure[Analytical spectrum]{\includegraphics[page=10, width=0.49\textwidth]{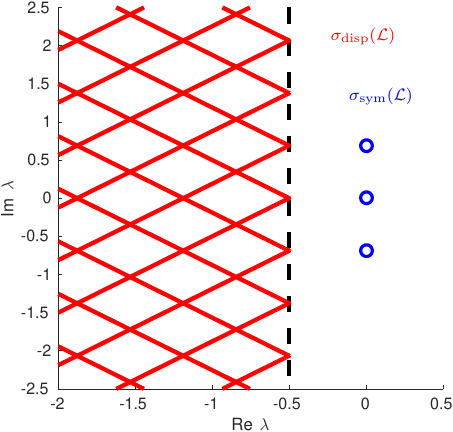}}
  \subfigure[Numerical spectrum]{\includegraphics[page=9, width=0.49\textwidth]{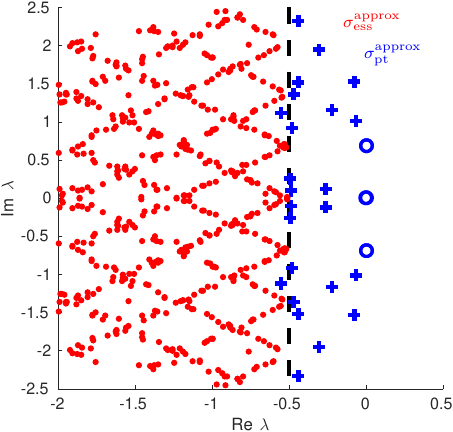}}
  \caption{Spectrum of QCGL linearized at a spinnig soliton for parameters \eqref{equ:SpinningSolitonsParameters} and $d=3$.}
  \label{fig:QCGL3DSpectrumSpinningSoliton}
\end{figure}

Figure \ref{fig:QCGL3DSpectrumSpinningSoliton}(b) shows an approximation $\sigma^{\mathrm{approx}}$ of the spectrum $\sigma(\L)$ of $\L$ linearized about the spinning soliton 
$v_{\star}$ for $d=3$. The numerical spectrum $\sigma^{\mathrm{approx}}$ of $\L$ is divided into the approximation of the essential spectrum $\sigma^{\mathrm{approx}}_{\mathrm{ess}}$ 
(red dots) and the approximation of the point spectrum $\sigma^{\mathrm{approx}}_{\mathrm{pt}}$ (blue circles and plus signs). The set $\sigma^{\mathrm{approx}}_{\mathrm{ess}}$ 
is an approximation of $\sigma_{\mathrm{ess}}(\L)$. The result shows that (at least in this case) we can expect $\sigma_{\mathrm{ess}}(\L)=\sigma_{\mathrm{disp}}(\mathcal{L})$. 
Similarly, the set $\sigma^{\mathrm{approx}}_{\mathrm{pt}}$ is an approximation of $\sigma_{\mathrm{pt}}(\L)$, which contains the approximation of $\sigma_{\mathrm{sym}}(\mathcal{L})$ 
(blue circles) and 12 additional complex-conjugate pairs of isolated eigenvalues satisfying $\Re\lambda>-\bzero$. In particular, one of these pairs lie between the black dashed 
line and the essential spectrum. Further computations show that they seem to persist under spatial mesh refinement and also when enlargeing the spatial domain. The case $d=2$ 
is also treated in \cite[Sec.~8]{BeynLorenz2008}.

Let us briefly return to Figure \ref{fig:QCGL3DSpectrumSpinningSoliton}(a) and discuss analytical results on eigenfunctions and their adjoints. In Theorem \ref{thm:3.1} we 
derived explicit formulas for the eigenfunctions associated to eigenvalues from the symmetry set $\sigma_{\mathrm{sym}}(\mathcal{L})$. In case $d=3$, the six eigenfunctions 
are those from Example \ref{exa:PointSpectrum3D}. We briefly recall that each $\lambda$ which is located to the right of $-\bzero$, either belongs to $\rho(\L)$ or to 
$\sigma_{\mathrm{pt}}(\L)$. In case $\lambda\in\sigma_{\mathrm{pt}}(\L)$ with $\Re\lambda>-\bzero$, Theorem \ref{thm:ExpDecEigLin} shows that the associated eigenfunction 
and adjoint eigenfunction decay exponentially in space with exponential decay rates given by, see \eqref{equ:ExponentialRatesLQstart} with $\gamma=\Re\lambda+\bzero$,
\begin{align}
  \label{equ:bound_p} 
  0\leqslant\mu_2\leqslant\varepsilon\frac{\sqrt{\Re\alpha(\Re\lambda-\Re\delta)}}{|\alpha|p}<\frac{\sqrt{\Re\alpha(\Re\lambda-\Re\delta)}}{|\alpha|\max\{p_{\min},\frac{d}{2}\}},
\end{align}
and
\begin{align}
  \label{equ:bound_q}
  0\leqslant\mu_4\leqslant\varepsilon\frac{\sqrt{\Re\alpha(\Re\lambda-\Re\delta)}}{|\alpha|q}<\frac{\sqrt{\Re\alpha(\Re\lambda-\Re\delta)}}{|\alpha|p_{\min}}.
\end{align}
The upper bounds show that the decay rates are affected by the spectral gap $\Re\lambda-\Re\delta$ between the eigenvalue $\lambda\in\sigma_{\mathrm{pt}}(\L)$ and 
the spectral bound $\bzero=-\Re\delta$ (black dashed line) of $D\mathbf{f}(\mathbf{v_{\infty}})$. Therefore, the decay rates are large for eigenvalues farther to 
the right of $\bzero$, and they become smaller as $\Re\lambda$ approaches the spectral bound $\bzero=-\Re\delta$. For the eigenvalues from the symmetry set 
$\sigma_{\mathrm{sym}}(\mathcal{L})$, parameters from \eqref{equ:SpinningSolitonsParameters}, and $d=3$, we obtain the following upper bounds for the exponential decay 
rates of the eigenfunctions and their adjoints
\begin{align*}
  0\leqslant\mu_2<\frac{\sqrt{2}}{3}\approx 0.4714\quad\text{and}\quad 0\leqslant\mu_4<\frac{4}{1+\sqrt{2}}\approx 1.6569.
\end{align*}
Note that the bounds \eqref{equ:bound_p} and \eqref{equ:bound_q} are uniform in $p$ and $q$, since they are computed from the bound $p_{\mathrm{min}}$. It is easily 
seen that the bounds coincide for $p=q=2$. But the bounds change when $p\neq 2\neq q$: For $p\to \max\{p_{\min},\frac{d}{2}\}$ we obtain $q\to\min\{p_{\max},\frac{d}{2}\}$ 
and thus the estimated decay rate of eigenfunction becomes maximal, while the 
estimated decay rate of the adjoint becomes minimal. Conversely, if $p\to p_{\max}$ we obtain $q\to p_{\min}$ and 
thus the decay rate of eigenfunction becomes minimal, while the decay rate of adjoint becomes maximal. 

\begin{figure}[ht]
  \centering
  \includegraphics[page=11, width=0.9\textwidth]{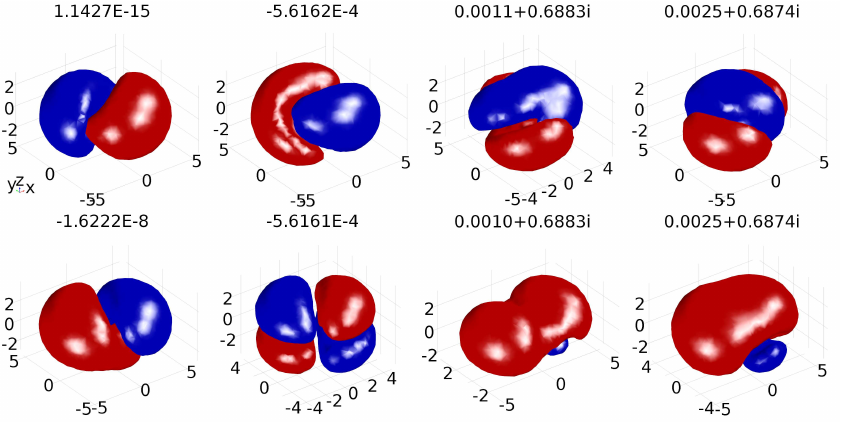}
  \caption{Eigenfunctions and adjoint eigenfunctions of QCGL linearized at a spinning soliton with $d=3$}
  \label{fig:QCGL3DEigenfunctionsSpinningSoliton}
\end{figure}

In Figure \ref{fig:QCGL3DEigenfunctionsSpinningSoliton} we visualize the approximate eigenfunctions $\mathbf{v}:\R^3\rightarrow\C^2$ (upper row) from 
\eqref{equ:Eigenfunctions3D} for  eigenvalues $\lambda_1,\lambda_2,\lambda_3,\lambda_5 \in \sigma_{\mathrm{sym}}(\L)$, and adjoint eigenfunctions $\mathbf{\psi}:\R^3\rightarrow\C^2$ (lower row). 
More precisely, Figure \ref{fig:QCGL3DEigenfunctionsSpinningSoliton} shows the isosurfaces at level values $\{-1.2,1.2\}$ of the real 
parts of their first components, i.e. $\mathbf{v}_j(x)=\mathbf{\psi}_j(x)=-\frac{1}{2}$ (red surface) and $\mathbf{v}_j(x)=\mathbf{\psi}_j(x)=\frac{1}{2}$ (blue surface). 
The approximate eigenvalues from $\sigma_{\mathrm{sym}}(\L)$ are displayed in the title of each subfigure. For a more detailed comparison  of numerical 
and theoretical decay rates of 
eigenfunctions we refer to \cite[Sec.~6.3]{BeynOtten2016a}.


\def\cprime{$'$} \def\cprime{$'$}

\end{document}